 \documentclass{amsart}
\usepackage{amsmath,amsthm,latexsym,amssymb,hyperref}
\usepackage{eucal}
\usepackage{mathrsfs}
\usepackage[all]{xy}
\xyoption{2cell}

\numberwithin{equation}{section}
\newtheorem{thm}{Theorem}[section]
\newtheorem{lem}[thm]{Lemma}
\newtheorem{prop}[thm]{Proposition}
\newtheorem{cor}[thm]{Corollary}

\theoremstyle{definition}
\newtheorem{defn}[thm]{Definition}
\theoremstyle{remark}
\newtheorem{rmk}[thm]{Remark}

\newtheorem{ex}[thm]{Example}

\newtheorem{notn}[thm]{Notation}
\newtheorem{convention}[thm]{Convention}

\newcommand\Om{\Omega}

\newcommand\vp{\varphi}
\newcommand \ve{\varepsilon}

\newcommand{\cat}{\mathbf}
\newcommand{\op}{\mathcal}

\newcommand{\ob}{\operatorname{Ob}}
\newcommand{\mor}{\operatorname{Mor}}

\newcommand\cotens [3]{{#1}\underset {#3}\square {#2}}

\newcommand\egal[2]{\overset {#1}{\underset {#2}\rightrightarrows }}
\newcommand\adjunct[2]{\overset {#1}{\underset {#2}\rightleftarrows }}
\newcommand\fib{\ar @{->>} [r]} 
 
\newcommand\cof{\;\ar@{ >->}[r]} 
\newcommand\ccof{\;\ar@{ >->}[rr]} 

\newdir{ >}{{}*!/-10pt/\dir{>}}
\newcommand\wefib{\fib ^{\sim}} 
\newcommand\wecof{\cof^{\sim}} 
\renewcommand{\Bar}{\mathscr B}

\renewcommand{\Bar}{\mathscr B}

\newcommand\D{\mathcal D}

\newcommand\tot{\operatorname{Tot}}
\newcommand\map{\operatorname{Map}}
\newcommand\can{\operatorname{Can}}
\newcommand\ext{\operatorname{Ext}}

\newcommand\colim{\operatorname{colim}}
\newcommand\hocolim{\operatorname{hocolim}}
\newcommand\T{\mathbb T}
\newcommand\K{\mathbb K}

\begin{document}

\title[Homotopic descent and codescent]{A general framework for\\ homotopic descent and codescent}

\author{Kathryn Hess}

\address{Institut de g\'eom\'etrie, alg\`ebre et topologie (IGAT) \\
    \'Ecole Polytechnique F\'ed\'erale de Lausanne \\
    CH-1015 Lausanne \\
    Switzerland}
\email{kathryn.hess@epfl.ch}
    
\begin{abstract} In this paper we elaborate a general homotopy-theoretic framework in which to study problems of descent  and completion and  of their duals, codescent and cocompletion.  Our approach to homotopic (co)descent  and to derived (co)completion can be viewed as $\infty$-category-theoretic, as our framework is constructed in the universe of simplicially enriched categories, which are a model for $(\infty, 1)$-categories.

We provide general criteria, reminiscent of Mandell's theorem on $E_{\infty}$-algebra models of $p$-complete spaces, under which homotopic (co)descent is satisfied. Furthermore, we construct general descent and codescent spectral sequences, which we interpret in terms of derived (co)completion and homotopic (co)descent.

 We show that a number of very well-known spectral sequences, such as the unstable and stable Adams spectral sequences, the Adams-Novikov spectral sequence and the descent spectral sequence of a map, are examples of general (co)descent spectral sequences.  There is also a close relationship between the Lichtenbaum-Quillen conjecture and homotopic descent along the Dwyer-Friedlander map from algebraic K-theory to \'etale K-theory.  Moreover, there are intriguing analogies between derived cocompletion (respectively, completion) and homotopy left (respectively, right) Kan extensions and their associated assembly (respectively, coassembly) maps.
\end{abstract}

\date{\today}

\maketitle

\tableofcontents

\section{Introduction}

The notions of descent and completion have long played a significant role in algebraic geometry, number theory, category theory and homotopy theory.  In this paper we elaborate a general homotopy-theoretic framework in which to study problems of descent  and completion and  of their duals, codescent and cocompletion.  Our approach to homotopic (co)descent  and to derived (co)completion can be viewed as $\infty$-category-theoretic, as our framework is constructed in the universe of simplicially enriched categories, which are a model for $(\infty, 1)$-categories (or $\infty$-categories, as homotopy theorists often call them).

We provide general criteria, expressed in terms of derived (co)completeness in a way reminiscent of the Main Theorem in \cite{mandell}, under which homotopic (co)descent is satisfied (Theorems \ref{thm:equiv-monad} and \ref{thm:equiv-comonad}).  We also prove homotopic versions of Beck's classical characterization of (co)monads of (co)descent type (Theorems \ref{thm:htpic-coBeck} and \ref{thm:htpic-Beck}). Furthermore, we construct general descent and codescent spectral sequences, which we interpret in terms of derived (co)completion and homotopic (co)descent.

To illustrate the breadth and flexibility of the framework built here, we explore the relationship between  Baum-Connes and Farrell-Jones-type isomorphism conjectures and derived cocompletion (Corollary \ref{cor:assembly}) on the one hand and   between the embedding calculus and derived completion on the other hand (Remark \ref {rmk:calculus}). We show moreover that a number of very well-known spectral sequences, such as the unstable and stable Adams spectral sequences (sections \ref{sec:descSS} and \ref{sec:ass}), the Adams-Novikov spectral sequence (section \ref{sec:anss}) and the descent spectral sequence of a map (section \ref{sec:dualhtpicGroth}), are examples of general (co)descent spectral sequences.  Finally, we describe the close relationship between the Lichtenbaum-Quillen conjecture and homotopic descent along the Dwyer-Friedlander map from algebraic K-theory to \'etale K-theory (section \ref{sec:lq}).

We begin in section \ref{sec:classical}  by recalling the framework of classical descent and codescent theory, expressed in terms of monads and comonads and of their associated Eilenberg-Moore categories of algebras and coalgebras  (Definitions \ref{defn:Talg} and \ref{defn:Kcoalg}).    We also remind the reader of important, well-known examples: Grothendieck descent theory along a ring homomorphism and codescent along a continuous map, with its application to bundle theory.  We generalize both of these examples to arbitrary monoidal categories,  describing descent along a morphism of monoids  and codescent along a morphism of comonoids. We emphasize the description of Grothendieck descent theory in terms of the descent co-ring associated to the fixed monoid morphism (Example \ref{ex:co-ring}) and its dual in terms of the codescent ring associated to the fixed comonoid morphism (Example \ref{ex:codescring}).

Section \ref{sec:simplicial} is devoted to the study of simplicial structures related to descent and codescent.  We begin by recalling the well-known cobar construction associated to a monad and bar construction associated to a comonad, of which the Amitsur complex (Remark \ref{rmk:amitsur})  and the \v Cech nerve (Remark \ref{rmk:cech}) are important examples, respectively. We prove a technical homotopy-theoretic result about each of these constructions (Lemmas \ref{lem:Reedyfib} and \ref{lem:Reedycofib}) required at two critical junctures later in the paper.  We then discuss how simplicial enrichment of a category induces simplicial enrichment on associated Eilenberg-Moore categories of algebras and coalgebras and explore various properties of these induced enrichments.  In particular, we prove that if a (co)monad is appropriately compatible with the simplicial enrichment of the category on which it acts, then the associated canonical (co)descent data functor (Definitions \ref{defn:desccat} and \ref{defn:codesccat}) is simplicial (Proposition \ref{prop:simplfunct}), and certain of its components are actually isomorphisms (Proposition \ref{prop:adj-iso}), which turns out to be of crucial importance in the next sections.

Derived completion along a monad and cocompletion along a comonad are the subjects of section \ref{sec:derivedcompletion}.  Given a simplicial model category $\cat M$ on which a monad $\T$ acts compatibly with the simplicial structure, we define notions of $\T$-equivalence and of $\T$-complete objects (Definitions \ref{defn:Tequiv} and \ref{defn:Tcompobj}), and provide examples.  We define the $\T$-completion of an object in $\cat M$ in terms of the totalization of the $\T$-cobar construction and prove its homotopy invariance.  We then introduce the related notions of strict $\T$-completeness and strong $\T$-completeness, which are also formulated in terms of the $\T$-cobar construction.  We give examples and show that strictly $\T$-complete, fibrant objects are $\T$-complete.  Finally, we point out that strictly $\T$-complete objects may be best understood as $\infty$-$\T$-algebras, i.e, as $\T$-algebras, up to an infinite family of higher structure maps.  We then dualize all of our theory of derived completion, to formulate a theory of derived cocompletion along a comonad.  To conclude the section, we consider the relationship between derived cocompletion and assembly and, dually, between derived completion and coassembly.

In section \ref{sec:htpicdesc} we are finally ready to define homotopic (co)descent for simplicially enriched categories (Definitions \ref{defn:htpicdesc} and \ref{defn:htpiccodesc}).    In the case of a (co)monad acting on a simplicial model category, we provide criteria for  homotopic (co)descent (Theorems \ref{thm:equiv-monad} and \ref{thm:equiv-comonad}), which are  analogous to the Main Theorem in Mandell's paper \cite{mandell}.  The criteria are formulated in terms of strict (co)completion, motivating our introduction of this notion. We point out that versions of ``faithfully flat (co)descent'' (Corollaries \ref{cor:ffd} and \ref{cor:ffcod}) are immediate consequences of these criteria.  We then explain how to construct the (co)descent spectral sequence that arises naturally from the action of a (co)monad on a simplicial model category as special cases of the extended homotopy spectral sequence of Bousfield and Kan \cite[X.6]{bousfield-kan}. The notions of both derived (co)completion and homotopic (co)descent play a role in interpreting these spectral sequences, which we can view as interpolating backwards along the simplicial canonical (co)descent data functor.  There is a long tradition in topology of studying generalizations of the Adams spectral sequence of this sort (e.g., \cite{bendersky-curtis-miller}, \cite{bendersky-hunton}, \cite{bendersky-thompson}, \cite{dwyer-miller-neisendorfer} and \cite {carlsson-chgring}), to which our construction and interpretation of this spectral sequence clearly adheres.

The examples motivating our work are studied in section \ref{sec:htpicGroth}.  We begin by developing the general theory of homotopic Grothendieck descent and of derived completion along a morphism of monoids in a monoidal model category.  We observe in particular that for the monad associated to a monoid map $\vp:B\to A$ to satisfy homotopic descent means that the category of $B$-modules is locally homotopy Tannakian, in a reasonable sense (Definition \ref{defn:htpicGrothdesc} and Remark \ref{rmk:lhT}).  We also show that a version of faithfully flat descent holds in this homotopical context (Corollary \ref{cor:ffd-groth}).

We then sketch a number of intriguing concrete applications, which will be further developed in later articles.  In particular, we indicate how both the stable Adams spectral sequence and the Adams-Novikov spectral sequence arise as descent spectral sequences associated to unit maps of ring spectra.  Fixing some prime $\ell$ and a noetherian $\mathbb Z[\frac 1 \ell]$-algebra $A$, we then describe the close relationship between the Lichtenbaum-Quillen conjecture for $A$ and homotopic descent along  the Dwyer-Friedlander map from the algebraic K-theory spectrum of $A$ to its \'etale K-theory spectrum. Finally, we sketch an application of homotopic Grothendieck descent to the construction of a spectral sequence converging from Quillen homology of a ring spectrum to homotopy.

Turning all the arrows around, we then develop a theory of homotopic Grothen\-dieck codescent and derived cocompletion along any morphism in a cartesian model category.  We show in particular that the comonad associated to pulling back over a Kan fibration of simplicial sets admitting a section satisfies homotopic codescent.

In the appendices the reader will find the proof of Theorem \ref{thm:assembly}, as well as a brief review of the various simplicial and model structures that come in handy in this article.

\subsection{Perspectives}

\begin{itemize}

\item We have chosen in this paper to work with simplicial model categories, both because they are a tractable model for $(\infty,1)$-categories and because, as Rezk, Schwede and Shipley proved in \cite{rezk-schwede-shipley}, a large class of model categories are at least Quillen equivalent to simplicial model categories.  In collaboration with P. M\"uller, we are developing analogous theories of homotopic descent and codescent for dg-categories (categories enriched over chain complexes) and for spectral categories (categories enriched over spectra).  We suspect that much of the theory translates directly into both of these contexts, opening the door to new and significant applications.

\item It would be very interesting to determine how the theory developed here is related to Lurie's theory of $(\infty,1)$-monads \cite[Chapter 3]{lurie} and, in particular to his Barr-Beck-type monadicity theorem (Theorem 3.4.5 in \cite{lurie}).  The approach taken here is quite different from that of Lurie. In particular the concepts applied here are considerably less sophisticated than those Lurie introduces, perhaps because our goals were different.

Most of the results in this paper are stated in terms of monads on simplicially enriched categories such that the underlying endofunctor of the monad is a simplicial functor, and the multiplication and the unit of the monad are simplicial natural transformations, which are concepts just powerful enough for our purposes. The goal of this paper is not to develop the ultimate theory of descent, but rather to set up a minimally complicated framework that is general enough to describe a wide range of particular descent theories that are relevant in homotopy theory and their related spectral sequences.

\item It seems likely that many well-known (co)descent-type spectral sequences and descent-related notions, beyond those treated in section \ref{sec:htpicGroth}, can be expressed within the framework developed in this paper.  For example, the homotopy fixed point spectral sequence, as set up by  Davis in \cite {davis}, is almost certainly a descent spectral sequence for an appropriate monad, while the \'etale cohomological descent spectral sequence developed by Jardine \cite[\S 6.1]{jardine} is most probably a special case of a codescent spectral sequence for a well-chosen comonad.  

\item Motivated by the close links between descent theory and Galois theory, in joint work with V. Karpova, we are exploring explore the relationship between homotopic Hopf-Galois theory, as developed in \cite{hess:hhg}, and homotopic descent theory.
\end{itemize} 

\subsection{Notation and conventions}
\begin{itemize}
\item Let $\cat C$ be a category, and let $A,B\in \ob \cat C$.  In these notes, the class of morphisms from $A$ to $B$ is denoted $\cat C(A,B)$.  The identity morphism on an object $A$ is usually denoted $A$ as well.
\medskip

\item If $C$ is an object of a category $\cat C$, then $C/\cat C$ denotes the under category of morphisms with domain $C$, while $\cat C/C$ denotes the overcategory of morphisms with codomain $C$.
\medskip

\item Arbitrary ordinary categories are denoted $\cat C$ or $\cat D$, while simplicially enriched categories are called $\cat S$ or $\cat S'$.  Finally, we write $\cat M$ or $\cat M'$ for model categories, whether or not they are simplicial.
\medskip

\item  If $F:\cat C\adjunct{}{} \cat D:U$ is a pair of adjoint functors, then $F$ (for ``free'') is the left adjoint, while $U$ (for ``underlying'') is the right adjoint.  We refer to this adjoint pair as the \emph{$(F,U)$-adjunction}.
\medskip

\item Let $\bold \Delta$ denote the category with $\ob \bold \Delta =\mathbb N$, where $\bold \Delta (m,n)$ is the set of order-preserving maps from $\{0,...,m\}$ to $\{0,..,n\}$. The standard $n$-simplex is denoted $\Delta[n]$.

 If $\cat C$ is any other category, then $\cat C^{\bold \Delta^{op}}$ (respectively, $\cat C^{\bold \Delta}$) is the category of simplicial (respectively, cosimplicial) objects in $\cat C$.  We write $\cat {sSet}$ for $\cat {Set}^{\bold \Delta ^{op}}$.
\medskip

\item If $X$ is an object in $\cat C$, then $cc^\bullet X$ denotes the constant cosimplicial object that is $X$ in each level, where all cofaces and codegeneracies are identity maps.  Similarly, $cs_{\bullet}X$ denotes the constant simplicial object that is $X$ in each level.
\medskip

\item If $\cat M$ is a model category, we always assume that the categories $\cat M^{\bold \Delta}$ and $\cat M^{\bold \Delta^{op}}$ are endowed with their Reedy model category structure \cite [Ch 15]{hirschhorn}.  The details of their Reedy structure that are necessary for this paper can be found in the appendix \ref{sec:reedy}.
\end{itemize}

\subsection{Acknowledgments}

The author would like to express great appreciation to Jack Morava, for the highly stimulating exchanges of email, motivated by \cite{morava}, that launched this project.  She would also like to thank Bill Dwyer very warmly for having suggested that she consider both the general (co)monad approach to (co)descent and $\infty$-categories in the context of this project.   She is moreover deeply grateful  Emmanuel Farjoun  for having provided the opportunity to talk about this project in an informal seminar at Hebrew University, where the excellent questions asked by the audience greatly aided in refining the theory. Finally, she would like to thank Daniel Davis for having pointed out a mistake in the original proof of Proposition \ref{prop:criterion}, and an anonymous referee for having asked very good questions.

\section{Classical descent and codescent}\label{sec:classical}

Here we recall the classical theory of descent and codescent, as expressed in terms of monads and comonads.  We take some care in doing so, to establish clearly notation and terminology that we need throughout the paper and to ensure that the reader understands exactly what we are generalizing when we define homotopic descent and codescent  in section \ref{sec:htpicdesc}.

We refer the reader to \cite{mesablishvili} for further details.  

\subsection{Foundations of the theory of (co)descent}

\subsubsection{(Co)monads and their (co)algebras}
\begin{defn} Let $\cat C$ be a category.  A \emph{monad} on $\cat C$ consists of an endofunctor $T:\cat C\to \cat C$, together with natural transformations $\mu:T\circ T \to T$ and $\eta:Id_{\cat C}\to T$ such that $\mu$ is appropriately associative and unital.  In other words, $\T=(T,\mu, \eta)$ is a monoid in the category of endofunctors of $\cat C$, which is monoidal under composition.

Dually, a \emph{comonad} on $\cat C$ an endofunctor $K:\cat C\to \cat C$, together with natural transformations $\Delta:K\to K\circ K$ and $\ve:K\to Id_{\cat C}$ such that $\Delta$ is appropriately coassociative and counital., i.e., $\K=(K,\Delta, \ve)$ is a comonoid in the category of endofunctors of $\cat C$.
\end{defn}

\begin{ex}  If $F:\cat C \adjunct{}{} \cat D:U$ is a pair of adjoint functors, with unit $\eta:Id_{\cat C}\to UF$ and counit $\ve: FU\to Id_{\cat D}$, then $(UF,U\ve F, \eta)$ is a monad on $\cat C$ and $(FU, F\eta U, \ve)$ is a comonad on $\cat D$.
\end{ex}

To any monad we can associate a useful category of ``algebras.''

\begin{defn}\label{defn:Talg}  Let $\T=(T,\mu,\eta)$ be a monad on a category $\cat C$.  The objects of the \emph{Eilenberg-Moore category of $\T$-algebras}, denoted $\cat C^\T$,  are pairs $(A, m)$, where $A\in \ob \cat C$ and $m\in \cat C(TA, A)$, which is appropriately associative and unital, i.e.,
$$m\circ Tm=m\circ \mu _{A}\quad \text{and}\quad m\circ \eta_{A}=Id_{A}.$$  
A morphism in $\cat C^\T$ from $(A,m)$ to $(A',m')$ is a morphism $f:A\to A'$ in $\cat C$ such that $m'\circ Tf=f\circ m$.
\end{defn}

The category $\cat C^\T$ of $\T$-algebras  is related to the underlying category $\cat C$ as follows.

\begin{rmk}\label{rmk:T-adjunct}  Let $\T=(T,\mu,\eta)$ be a monad on a category $\cat C$. The forgetful functor $U^{\T}:\cat C^\T\to \cat C$ admits a left adjoint
$$F^{\T}:\cat C\to \cat C^\T,$$
called the \emph{free $\T$-algebra functor}, which is defined on objects by
$$F^{\T}(X) = (TX, \mu_{X})$$
and on morphisms by
$$F^{\T}(f)=Tf.$$

Note that $\T$ itself is the monad associated to the $(F^\T,U^\T)$-adjunction.  On the other hand, the comonad associated to the $(F^\T,U^\T)$-adjunction is 
$$\K^\T=(F^\T U^\T, F^\T \eta U^\T, \ve^\T),$$
where $\ve^\T$ is the counit of the $(F^\T,U^\T)$-adjunction, which is given by
$$(\ve^\T)_{(A,m)}=m:(TA,\mu _{A})\to (A,m).$$  
\end{rmk}

In the case of a monad $\T$ arising from an adjunction $F:\cat C \adjunct{}{} \cat D:U$, it is natural to  wonder about the relationship between $\cat D$ and the category of $\T$-algebras, which is mediated by the following comparison functor.

\begin{defn} Let $F:\cat C \adjunct{}{} \cat D:U$ be a pair of adjoint functors, with unit $\eta:Id_{\cat C}\to UF$ and counit $\ve: FU\to Id_{\cat D}$.  Let $\T$ denote the associated monad.  The \emph{canonical $\T$-algebra functor}
$$\can ^\T:\cat D\to \cat C^\T$$
is defined on objects by
$$\can ^\T(D)=(UD, U\ve_{D})$$
and on morphisms by
$$\can ^\T(f)=Uf.$$
If $\cat D$ admits coequalizers, then $\can^\T$ has a left adjoint, the ``indecomposables'' functor
$$\operatorname{Q}^\T:\cat C^\T\to \cat D,$$
which is defined on objects by
$$\operatorname{Q}^\T(A,m)=\operatorname{coequal}(FUFA\egal{Fm}{\ve_{FA}} FA).$$

The functor $U$ is \emph{monadic} if $\can ^\T$ is an equivalence of categories.
\end{defn}

\begin{rmk} \label{rmk:monad-diagram}For any adjunction $F:\cat C \adjunct{}{} \cat D:U$ with associated monad $\T$, the diagram
$$\xymatrix{\cat C \ar @<.7ex>[r]^-{F}\ar @{=} [d]&\cat D\ar @<.7ex>[l]^-{U}\ar[d]^{\can^\T}\\ {\cat C} \ar @<.7ex>[r]^-{F^\T}&\cat C^\T\ar @<.7ex>[l]^-{U^\T}}$$
commutes, i.e., $\can ^\T \circ F=F^\T$ and $U^\T \circ \can ^\T=U$.
\end{rmk}

If all one requires is for the canonical comparison functor to be fully faithful, then the following characterization, due to Beck (as formulated e.g., in \cite [Section 3.3]{barr-wells}), proves very useful.  

\begin{thm}\label{thm:triple} Let $F:\cat C \adjunct{}{} \cat D:U$ be a pair of adjoint functors, with unit $\eta:Id_{\cat C}\to UF$ and counit $\ve: FU\to Id_{\cat D}$.  Let $\T$ denote the associated monad. If $\cat D$ admits coequalizers, then $\can ^\T: \cat D\to \cat C^\T$ is fully faithful if and only if the counit $Q^\T\circ \can ^\T(D) \to D$ of the $(Q^\T, \can ^\T)$-adjunction is an isomorphism for all $D\in \ob \cat D$.
\end{thm}

\begin{rmk} Note that, under the hypotheses of the theorem above, 
$$Q^\T\circ \can ^\T(D) =\operatorname{coequal}(FUFUD\egal{FU\ve_{D}}{\ve_{FUD}} FUD).$$
\end{rmk}

The dual situation for comonads can be described as follows.

\begin{defn}\label{defn:Kcoalg}  Let $\K=(K, \Delta, \ve)$ be a comonad on $\cat D$.  The objects of the \emph{Eilenberg-Moore category of $K$-coalgebras}, denoted $\cat D_{\K}$,  are pairs $(D, \delta)$, where $D\in \ob \cat D$ and $\delta\in \cat D(D, KD)$, which is appropriately coassociative and counital, i.e.,
$$K\delta \circ \delta = \Delta_{D}\circ \delta\quad\text{and}\quad \ve_{D}\circ \delta =Id_{D}.$$ 
A morphism in $\cat D_{\K}$ from $(D,\delta)$ to $(D',\delta')$ is a morphism $f:D\to D'$ in $\cat D$ such that $Kf\circ \delta=\delta'\circ f$.
\end{defn}

The category $\cat D_{\K}$ of $\K$-coalgebras is related to the underlying category $\cat D$ as follows.

\begin{rmk} \label{rmk:K-adjunct} Let $\K=(K, \Delta, \ve)$ be a comonad on $\cat D$. The forgetful functor $U_{\K}:\cat D_{\K}\to \cat D$ admits a right adjoint
$$F_{\K}:\cat D\to \cat D_{\K},$$
called the \emph{cofree $\K$-coalgebra functor}, which is defined on objects by
$$F_{\K}(X) = (KX, \Delta_{X})$$
and on morphisms by
$$F_{\K}(f)=Kf.$$
Note that  that $\K$ itself is the comonad associated to the $(U_{\K},F_{\K})$-adjunction. On the other hand, the monad associated to the $(U_{\K},F_{\K})$-adjunction is
$$\T_{\K}=(F_{\K}U_{\K}, F_{\K}\ve U_{\K}, \eta_{\K}),$$
where $\eta_{\K}$ is the unit of the $(U_{\K},F_{\K})$-adjunction, which is given by
$$(\eta_{\K})_{(D,\delta)}=\delta: (D,\delta)\to (KD, \Delta _{D}).$$
\end{rmk}

As in the monad case, if the comonad $\K$ arises from an adjunction $F:\cat C \adjunct{}{} \cat D:U$, then a comparison functor, defined below, mediates between $\cat C$ and $\cat D_{\K}$.

\begin{defn} Let $F:\cat C \adjunct{}{} \cat D:U$ be a pair of adjoint functors, with unit $\eta:Id_{\cat C}\to UF$ and counit $\ve: FU\to Id_{\cat D}$.  Let $\K$ denote the associated comonad.  The \emph{canonical $\K$-coalgebra functor}
$$\can _{\K}:\cat C\to \cat D_{\K}$$
is defined on objects by
$$\can _{\K}(C)=(FC,F\eta_{C})$$
and on morphisms by
$$\can _{\K}(f)=Ff.$$
If $\cat C$ admits equalizers, then $\can _{\K}$ has a right adjoint, the ``primitives'' functor
$$\operatorname{Prim}_{\K}:\cat D_{\K}\to \cat C,$$
which is defined on objects by 
$$\operatorname{Prim}_{\K}(D, \delta)=\operatorname{equal} (UD \egal {U\delta}{\eta_{UD}} UFUD).$$

The functor $F$ is \emph{comonadic} if $\can _{\K}$ is an equivalence of categories.
\end{defn}

\begin{rmk}\label{rmk:comonad-diagram} For any adjunction $F:\cat C \adjunct{}{} \cat D:U$ with associated comonad $\K$, the diagram
$$\xymatrix{\cat C \ar @<.7ex>[r]^-{F}\ar [d]_{\can_{\K}}&\cat D\ar @<.7ex>[l]^-{U}\ar@{=}[d]^{}\\ {\cat C}_{\K} \ar @<.7ex>[r]^-{U_{\K}}&\cat D\ar @<.7ex>[l]^-{F_{\K}}}$$
commutes, i.e., $U_{\K}\circ \can _{\K}=F$ and $\can _{\K}\circ U=F_{\K}$.
\end{rmk}

If all one requires is for the canonical comparison functor to be fully faithful, then the following characterization, due to Beck (where we dualize the formulation in e.g., in \cite [Section 3.3]{barr-wells}), proves very useful.

\begin{thm}\label{thm:cotriple} Let $F:\cat C \adjunct{}{} \cat D:U$ be a pair of adjoint functors, with unit $\eta:Id_{\cat C}\to UF$ and counit $\ve: FU\to Id_{\cat D}$.  Let $\K$ denote the associated comonad. If $\cat C$ admits equalizers, then $\can _{\K}: \cat C\to \cat D_{\K}$ is fully faithful if and only if the unit $C\to \operatorname{Prim}_{\K}\circ \can_{\K}(C)$ of the $(\can _{\K},\operatorname{Prim}_{\K})$-adjunction is an isomorphism for all $C\in \ob \cat C$.
\end{thm}

\begin{rmk} Note that, under the hypotheses of the theorem above, 
$$\operatorname{Prim}_{\K}\circ \can_{\K}(C)=\operatorname{equal} (UFC \egal {U\eta_{C}}{\eta_{UFC}} UFUFC).$$
\end{rmk}

\subsubsection{Monads and descent}

When the adjunction from which we start is 
$$F^\T:\cat C \adjunct{}{} \cat C^\T:U^\T$$
 itself, for some monad $\T$ on a category $\cat C$, we can formulate the well-known notion of descent in terms of algebras over a monad and coalgebras over a comonad.

\begin{defn}\label{defn:desccat}  The \emph {descent category} of a monad $\T$ on a category $\cat C$, denoted $\op D(\T)$, is the category $(\cat C^\T)_{\K^\T}$ of $\K^\T$-coalgebras in the category of $\T$-algebras. The objects of $\op D(\T)$ are called \emph{descent data}.

The monad $\T$ satisfies \emph{descent} if 
$$\can_{\K^\T}: \cat C\to \op D(\T)$$
 is fully faithful.  If $\can_{\K^\T}$ is an equivalence of categories, i.e., if the free $\T$-algebra functor $F^\T:\cat C\to \cat C^\T$ is comonadic, then $\T$ is satisfies \emph{effective descent}.
\end{defn}

To formulate a correct homotopical generalization of the notion of descent, we need to unfold this rather intricate definition.  We begin by analyzing the objects in $\op D(\T)$.

\begin{rmk}\label{rmk:chardescdat}  Let $\T=(T,\mu, \eta)$ be any monad on $\cat C$. A descent datum for $\T$  is a triple $(H, m, \delta)$, where $m: TH\to H$ (respectively $\delta: H\to TH$) is a morphism in $\cat C$ that is associative and unital (respectively,  coassociative and counital) and 
$$\xymatrix{TH \ar[r]^m\ar [d]_{T\delta}&H\ar [d]^\delta\\ T^2H\ar[r]^{\mu_{H}}&TH}$$
commutes.  We use the symbol $H$ for the underlying object in $\cat C$, as descent data are of a Hopf-like nature, being endowed with compatible multiplication and comultiplication.

For any object $C$ in $\cat C$, the canonical descent datum associated to $C$ is
$$\can_{\K^\T}(C)=(TC, \mu _{C}, T\eta_{C}).$$
It follows that for any object $C$,
$$F_{\K^\T}F^\T(C)=(T^2C, \mu_{TC}, T\eta_{TC}\big)=\can _{\K^\T} (TC).$$
\end{rmk}

We next examine in more detail the notion of (effective) descent.

\begin{rmk} If $\T=(T,\mu, \eta)$ is a monad on a category $\cat C$ admitting equalizers, then for all descent data $(H,m,\delta)$,
$$\operatorname{Prim}_{\K^\T}(H,m,\delta)=\operatorname{equal}(H\egal{\delta}{\eta_{H}} TH),$$
where the equalizer is computed in $\cat C$.    In particular, for all objects $C$ in $\cat C$,
 $$\operatorname{Prim}_{\K^\T}\can_{\K^\T}(C)=\operatorname{equal}(TC\egal{T\eta_{C}}{\eta_{TC}} T^2C).$$

As a consequence of Beck's theorem for comonads (Theorem \ref{thm:cotriple}) we obtain the following characterization of descent, of which we prove a homotopic version (Theorem \ref{thm:htpic-coBeck}) later in this article.

\begin{thm}\label{thm:coBeck}  If $\T=(T,\mu, \eta)$ is a monad on a category $\cat C$ admitting equalizers, then $\T$ satisifes descent if and only if 
$$C\xrightarrow {\eta_{C}}TC\egal{T\eta_{C}}{\eta_{TC}} T^2C $$
is an equalizer for all objects $C$ in $\cat C$.
\end{thm} 

 If $\T$ satisfies effective descent, then, in addition, any descent datum $(H,m,\delta)$ must be naturally isomorphic to the canonical descent datum associated to its object of ``primitives'', i.e.,
$$(H,m,\delta)\cong\can_{\K^\T}\operatorname{Prim}_{\K^\T}(H,m,\delta).$$
\end{rmk}

\subsubsection{Comonads and codescent}

Starting now from the adjunction $$U_{\K}:\cat D_{\K}\adjunct{}{} \cat D:F_{\K}$$ for some monad $\K$ on a category $\cat D$, we present the elements of the theory of codescent.

\begin{defn}\label{defn:codesccat}  The \emph{codescent category} of a comonad $\K$ on a category $\cat D$, denoted $\op D^{co}(\K)$, is the category $(\cat D_{\K})^{\T_{\K}}$ of $\T_{\K}$-algebras in the category of $\K$-coalgebras. The objects of $\op D^{co}(\K)$ are called \emph{codescent data}.

The comonad $\K$ satisfies  \emph{codescent} if
$$\can^{\T_{\K}}:\cat D\to \op D^{co}(\K)$$
is fully faithful.  If $\can^{\T_{\K}}$ is an equivalence of categories, i.e., if the cofree $\K$-coalgebra functor $F_{\K}:\cat D\to \cat D_{\K}$ is monadic, then $\K$ satisfies \emph{effective codescent}.
\end{defn}

We again spell out more explicitly what this definition means, in order to generalize it correctly later.

\begin{rmk}  Let $\K=(K,\Delta, \ve)$ be any comonad on $\cat D$.  A codescent datum for $\K$ is a triple $(H, \delta, m)$, where $\delta: H\to KH$ (respectively $m: KH\to H$) is a morphism in $\cat D$ that is coassociative and counital (respectively,  associative and unital) and 
$$\xymatrix{KH\ar [d]_{\Delta_{H}}\ar[r]^m&H\ar [d]^\delta\\ K^2H \ar [r]^{Km}&KH}$$
commutes. We again use the symbol $H$ for the underlying object in $\cat D$, as codescent data are also of a Hopf-like nature.

For any object $D$ in $\cat D$, the canonical codescent datum associated to $D$ is
$$\can^{\T_{\K}}(D)=(KD,\Delta_{D},K\ve_{D}).$$
In particular, for all $D\in \ob D$,
$$F^{\T_{\K}}F_{\K}(D)=(K^2D, \Delta_{KD},K{\eta_{KD}})=\can ^{\T_{\K}} (KD).$$
\end{rmk}

Finally, we examine in more detail the notion of (effective) codescent.

\begin{rmk} If $\K=(K,\Delta, \ve)$ is a comonad on a category $\cat D$ admitting coequalizers, then for all codescent data $(H,\delta,m)$,
$$\operatorname{Q}^{\T_{\K}}(H,\delta,m)=\operatorname{coequal}(KH\egal{m}{\ve_{KH}} H),$$
where the coequalizer is computed in $\cat D$.  In particular, for all objects $D$ in $\cat D$,
$$\operatorname{Q}^{\T_{\K}}\can ^{\T_{\K}}(D)=\operatorname{coequal}(K^2D\egal{K\ve_{D}}{\ve_{KD}} KD).$$

As a consequence of Beck's theorem for monads (Theorem \ref{thm:triple}) we obtain the following characterization of codescent, of which we prove a homotopic version (Theorem \ref{thm:htpic-Beck}) later in this article.

\begin{thm}\label{thm:Beck}  If $\K=(K,\Delta, \ve)$ is a comonad on a category $\cat D$ admitting coequalizers, then $\K$ satisifes codescent if and only if 
$$K^2D\egal{K\ve_{D}}{\ve_{KD}} KD \xrightarrow {\ve_{D}} D$$
is a coequalizer for all objects $D$ in $\cat D$.
\end{thm}

 If $\K$ satisfies effective codescent, then, in addition, any codescent datum $(H,\delta,m)$ must be naturally isomorphic to the canonical codescent datum associated to its object of ``indecomposables'', i.e.,
$$(H,\delta,m)\cong\can^{\T_{\K}}\operatorname{Q}^{\T_{\K}}(H,\delta,m).$$
\end{rmk}

\subsection{Grothendieck descent and its dual}
\subsubsection{Grothendieck descent for modules}\label{sec:grothendieck}
The example of descent that we present here is a slight generalization of classical Grothendieck descent for modules over rings.

Let $(\cat M, \wedge, I)$ be a monoidal category with all coequalizers. Let $\vp: B\to A$ be a morphism of monoids in $\cat M$, which induces an adjunction
$$-\underset B\wedge  A: \cat {Mod}_{B}\adjunct{}{} \cat {Mod}_{A}:\vp^*$$
and therefore a monad $\T_{\vp}$ with underlying endofunctor
$$\vp^* (-\underset B\wedge A):\cat {Mod}_{B}\to \cat {Mod}_{B}$$
and a comonad $\K_{\vp}$ with underlying endofunctor
$$\vp^*(-)\underset B\wedge A:\cat {Mod}_{A}\to \cat {Mod}_{A}.$$

The central problem of Grothen\-dieck descent theory for modules is to determine when $-\underset B\wedge A$ is comonadic, i.e., when 
$$\can _{\K_{\vp}}:\cat{Mod}_{B}\to (\cat {Mod}_{A})_{\K_{\vp}}$$
is an equivalence of categories. Since, as is well known (cf. \cite[\S 8]{mesablishvili}), $\vp^*$ is always monadic, i.e., 
$$\can ^{\T_{\vp}}:\cat {Mod}_{A}\xrightarrow \simeq (\cat {Mod}_{B})^{\T_{\vp}}$$  
is an equivalence, it follows that $\T_{\vp}$ is of effective descent if and only if $-\underset B\wedge A$ is comonadic.  

A great deal is known about conditions under which $-\underset B\wedge A$ is comonadic.  For example, if  $\cat M$ is the category of abelian groups and $\vp$ is a homomorphism of commutative rings, then $-\underset B\wedge A$ is comonadic if and only if $A$ is pure as a $B$-module.

There is a useful alternate description of effective descent for modules in monoidal categories, in terms of the following generalization of an important notion from ring theory.  We suppose henceforth that the functors 
$$X\wedge -, -\wedge X:\cat M\to \cat M$$
preserve colimits for all objects $X$ in $\cat M$, so that $-\underset A\wedge-$ defines a monoidal structure on the category $_{A}\cat {Mod}_{A}$ of a $A$-bimodules.

\begin{defn}  Let $A$ be a monoid. An \emph{$A$-co-ring} is a comonoid in the monoidal category $(_{A}\cat {Mod}_{A},-\underset A\wedge -, A)$ of $A$-bimodules.  In other words, an $A$-co-ring is an $A$-bimodule $W$ that is endowed with a coassociative, counital comultiplication $\psi:W\to W\underset A\wedge W$ that is a morphism of $A$-bimodules.
\end{defn}

\begin{ex}\label{ex:co-ring} Let $\vp:B\to A$ be any morphism of monoids in $\cat M$.  The \emph{descent co-ring associated to $\vp$}, denoted $W_{\vp}$ and also known as the \emph{canonical co-ring on $\vp$}, has as underlying $A$-bimodule $A\underset B\wedge A$, endowed with a comultiplication $\psi_{can}$, which is equal to the composite
$$A\underset B\wedge A\cong A\underset B\wedge B\underset B\wedge A\xrightarrow{A\underset B\wedge\vp\underset B\wedge A} A\underset B\wedge A\underset B\wedge A\cong (A\underset B\wedge A) \underset A\wedge (A \underset B\wedge A).$$
The morphism $\bar \mu:A\underset B\wedge A\to A$ induced by the multiplication map of $A$ is the counit of $\psi_{can}$. 
\end{ex}

\begin{rmk}  The descent co-ring $W_{\vp}$ is a Hopf algebroid over $A$ (generally, without antipode), where the left and right units are
$$\vp\underset B\wedge A: A\to W_{\vp}$$
and
$$A\underset B\wedge \vp:A\to W_{\vp}.$$
\end{rmk}

\begin{defn}  Let $A$ be a monoid, and let $(W,\psi, \epsilon)$ be an $A$-co-ring.  The category $\cat {M}^W_{A}$ is the category of \emph{$W$-comodules in the category of right $A$-modules}.  In other words, an object of $\cat{M}^W_{A}$ is a right $A$-module $M$ together with a morphism $\theta:M\to M\underset A\wedge W$ of right $A$-modules such that  the diagrams
 $$\xymatrix{M\ar[rr]^\theta \ar [d]^\theta&&M\underset A\wedge W\ar [d]^{\theta\underset A\wedge W}&& M\ar [r]^(0.4)\theta\ar[dr]^=&M\underset A\wedge W\ar[d]^{M\underset A\wedge \epsilon}\\
M\underset A\wedge W\ar [rr]^{M\underset A\wedge \psi}&&M\underset A\wedge W\underset A\wedge W&&&M}$$
commute. Morphisms in $\cat M^W_{A}$ are morphisms of $A$-modules that respect the $W$-coactions.
\end{defn}

\begin{ex}Let $\vp: B\to A$ be any morphism of monoids in $\cat M$, and let $$W_{\vp}=(A\underset B\wedge A, \psi_{can},\epsilon_{can}),$$ the descent co-ring associated to $\vp$.  If the monoid $A$ is augmented, the category $\cat M^{W_{\vp}}_{A}$ is isomorphic to $\op D(\T_{\vp})$, the descent category associated to $\T_{\vp}$.  Analyzing the definition of descent data in this specific case, we see that an object of $\op D(\T_{\vp})$  is a right $A$-module $M$ endowed with a morphism $\theta: M\to M\underset B\wedge A$ of right $A$-modules such that the diagrams
$$\xymatrix{M\ar[rr]^\theta \ar [d]^\theta&&M\underset B\wedge A\ar [d]^{\theta\underset B\wedge A}&&M\ar [r]^(0.4)\theta\ar[dr]^=&M\underset B\wedge A\ar[d]^{\bar r}\\
M\underset B\wedge A\ar [rr]^{M\underset B\wedge \vp \underset B\wedge A}&&M\underset B\wedge A\underset B\wedge A&&&M}$$
commute, where $\bar r$ is induced by the right $A$-action on $M$, and we have suppressed the restriction of scalars, $\vp^*$, from the notation.  
The morphisms in $\op D(\T_\vp)$ are $A$-module morphisms respecting the structure maps.

The key to showing that $\cat M^{W_{\vp}}_{A}$ and $\op D(\T_\vp)$ are isomorphic is the observation that 
$$M\underset A\wedge W_{\vp}=M\underset A\wedge A\underset B\wedge A\cong M\underset B\wedge A$$
for all right $A$-modules $M$.
\end{ex}

\begin{rmk} Under the isomorphism $\cat M^{W_{\vp}}_{A}\cong\op D(\T_\vp)$,
$$\can _{\vp}:=\operatorname{Can}_{\K_{\vp}}: \cat{Mod}_{B}\to \op D(\T_\vp)$$
 is defined on objects by $\operatorname{Can}_{{\vp}}(M)=(M\underset B\wedge A, \theta_{M})$, with $\theta_{M}=M\underset B\wedge \vp \underset B\wedge A$.  This is a particularly nice descent situation, where the descent category is \emph{Tannakian}, i.e., equivalent to a category of comodules over a coalgebra in a monoidal category.  We return to a general consideration of this special descent framework in an upcoming article.  We expect Hovey's recent work on a homotopic version of the Eilenberg-Watts theorem \cite{hovey:e-w} to prove useful in this context.
 
 If $\cat {Mod}_{B}$ admits equalizers, then the right adjoint of the functor $\operatorname{Can}_{{\vp}}$ is $$\operatorname{Prim}_{\vp}:\op D(\T_\vp)\to \cat {Mod}_{B},$$ where
$$\operatorname{Prim}_{\vp}(N,\theta)=\operatorname{equal}( N\egal {\theta}{N\underset B\wedge \vp} N\underset B\wedge A).$$
\end{rmk}

\subsubsection{Dualizing Grothendieck descent}\label{sec:dualGroth}

We now dualize the results of the previous section, studying codescent associated to a morphism of comonoids.  We again suppose that $(\cat M, \wedge, I)$ is a monoidal category, this time with all equalizers.

Let $C$ be a comonoid in $\cat M$. If $(M,\rho)$ is a right $C$-comodule and $(N,\lambda)$ is a left $C$-comodule, then their cotensor product is defined to be
$$M\underset C\square N=\operatorname{equal}(M\otimes N\egal {\rho\otimes N}{M\otimes \lambda} M\otimes C\otimes N),$$ which is an object in $\cat M$. We assume  that $-\underset C\square -$ is naturally associative, giving rise to a monoidal structure on ${}_{C}\cat {Comod}_{C}$, the category of $C$-bicomodules.

It is helpful to keep in mind the special class of examples  in which the monoidal product $\wedge$ is the categorical product $\times$, so that $\cat M$ is \emph{cartesian}.  If $\cat M$ is cartesian, then any object $X$ has a natural comonoid structure given by the diagonal morphism $\Delta_{X}$.  Furthermore, a right (or left) coaction of  $(X,\Delta_{X})$ on an object $Y$ is determined by a morphism $Y\to X$ in $\cat M$.  More precisely,
$$\cat {Comod}_{X} \cong \cat M/X\cong {}_{X}\cat {Comod},$$
where $\cat M/X$ denotes the slice category, of which the objects are morphisms in $\cat M$ with target $X$ and the morphisms are commuting triangles. 
Note that the cotensor product of comodules corresponds to pullback of morphisms under this identification, which gives rise to a monoidal structure on ${}_{X}\cat {Comod}_{X}$. 

We consider here  the following sort of adjunction.  Let $\vp: C\to D$ be a morphism of comonoids in $\cat M$, which induces an adjunction
$$\vp_{!}:\cat {Comod}_{C}\adjunct{}{} \cat {Comod}_{D}:\cotens{-}CD,$$
where $\vp_{!}$ is the ``extension of scalars'' functor.
The associated monad $\T_{\vp}$ has underlying endofunctor
$$\vp_{!}(-)\underset D\square C:\cat {Comod}_{C}\to \cat {Comod}_{C},$$
while the endofunctor underlying the associated comonad $\K_{\vp}$ is  
$$\vp_{!}(-\underset D\square C):\cat {Comod}_{D}\to \cat {Comod}_{D}.$$

Just as the ``restriction of scalars'' functor is monadic in the case of a morphism of monoids, the functor
$$\vp_{!}:\cat {Comod}_{C}\to \cat {Comod}_{D}$$
is comonadic, i.e., 
$$\can_{K_{\vp}}:\cat {Comod}_{C}\to (\cat {Comod}_{D})_{\K_{\vp}}$$
is an equivalence of categories.  Consequently, $\K_{\vp}$ is of effective codescent if and only if $-\underset D\square C$ monadic, i.e., if and only if 
$$\can ^{\T_{\vp}}:\cat{Comod}_{D}\to (\cat {Comod}_{C})^{\T_{\vp}}$$
is an equivalence of categories.  

As in the module case, there is a useful alternate description of effective codescent.

\begin{defn}  Let $C$ be a comonoid. A \emph{$C$-ring} is a monoid in $(_{C}\cat {Comod}_{C},-\underset C\square -, C)$.  In other words, a $C$-ring is an $C$-bicomodule $V$ that is endowed with an associative, unital multiplication $\gamma:V\underset C\square V\to V$ that is a morphism of $C$-bicomodules.
\end{defn}

\begin{ex}\label{ex:codescring} Let $\vp:C\to D$ be any morphism of comonoids in $\cat M$.  The \emph{codescent ring associated to $\vp$}, denoted $V^\vp$, has as underlying $C$-bicomodule $C\underset D\square C$, endowed with a multiplication $\gamma_{can}$, which is equal to the composite
$$\cotens {(\cotens CCD)}{(\cotens CCD)} C\cong \cotens {\cotens CCD}CD \xrightarrow {\cotens {\cotens C\vp D}CD} \cotens {\cotens CDD} CD\cong \cotens CCD.$$
The morphism $\bar \Delta:C\to \cotens CCD$ induced by the comultiplication map $\Delta$ of $C$ is the unit of $\gamma_{can}$. 
\end{ex}

\begin{defn}  Let  $C$ be a comonoid, and let $(V,\gamma, \eta)$ be a $C$-ring.  The category $\cat {M}^C_{V}$ is the category of \emph{$V$-modules in the category of right $C$-comodules}.  In other words, an object of $\cat{M}^C_{V}$ is a right $C$-comodule $M$ together with a morphism $\theta:\cotens MVC \to M$ of right $C$-comodules such that  the diagrams
 $$\xymatrix{\cotens {\cotens MVC}VC \ar[rr]^{\cotens \theta VC} \ar [d]^{\cotens M\gamma C}&&\cotens MVC\ar [d]^{\theta}&& M\ar [r]^(0.4){\cotens M\eta C}\ar[dr]^=&\cotens MVC\ar[d]^{\theta}\\
\cotens MVC\ar [rr]^{\theta}&&M&&&M}$$
commute. Morphisms in $\cat M^C_{V}$ are morphisms of $C$-comodules that respect the $V$-actions.
\end{defn}

\begin{ex} Let $\vp:C\to D$ be any morphism of comonoids in $\cat M$, and let 
$$V^{\vp}=(\cotens CCD, \gamma_{can},\eta_{can}),$$
 the codescent ring associated to $\vp$.  If $C$ is coaugmented, then the category $\cat M^{C}_{V^\vp}$ is isomorphic to the category of codescent data  $\op D^{co}(\K_{\vp})$. An object of $\op D^{co}(\K_\vp)$  is a right $C$-comodule $M$ endowed with a morphism $\theta: \cotens MCD \to M$ such that the diagrams
$$\xymatrix{\cotens {\cotens MCD}CD \ar[rr]^{\cotens \theta CD} \ar [d]^{\cotens {\cotens M\vp D}CD}&&\cotens MCD\ar [d]^{\theta}&&M\ar [r]^(0.4){\bar\rho}\ar[dr]^=&\cotens MCD \ar[d]^{\theta}\\
\cotens MCD\ar [rr]^\theta &&{M}&&&M}$$
commute, where $\bar \rho$ is induced by the right $C$-coaction on $M$, and we have suppressed the extension of scalars, $\vp_{!}$, from the notation.  
The morphisms in $\op D^{co}(\K_{\vp})$ are $C$-comodule morphisms respecting the structure maps.

The key to showing that $\cat M^{C}_{V^\vp}$ and $\op D^{co}(\K_{\vp})$ are isomorphic is the observation that 
$$\cotens M{V^\vp}C=\cotens M{\cotens CCD} C\cong \cotens MCD$$
for all right $C$-comodules $M$.
\end{ex}

\begin{rmk} Under the isomorphism $\cat M^{C}_{V^\vp}\cong \op D^{co}(\K_{\vp})$,
$$\can ^\vp:=\operatorname{Can}^{\T_{\vp}}: \cat{Comod}_{D}\to \op D^{co}(\K_{\vp}),$$
 is defined on objects by $\can^{\T_{\vp}}(M)=(\cotens MCD, \theta_{M})$, with $\theta_{M}=\cotens M{\cotens \vp C D}D$.  This is a particularly nice codescent situation, where the codescent category is \emph{coTannakian}, i.e., equivalent to a category of modules over a algebra in a monoidal category.  We return to a general consideration of this special codescent framework in an upcoming article.  We expect Hovey's recent work on a homotopic version of the Eilenberg-Watts theorem \cite{hovey:e-w} to prove useful in this context.
 
 If $\cat {Comod}_{D}$ admits coequalizers, then the left adjoint of the functor $\can^{{\vp}}$ is
$$\operatorname{Q}^{\vp}(N,\theta)=\operatorname{coequal}( \cotens NCD \egal {\theta}{\cotens N\vp D} N).$$
\end{rmk}

\section{Simplicial structures and (co)descent}\label{sec:simplicial}

In order to define homotopic (co)descent, we need  the simplicial structures associated to any (co)monad that we present here.  We begin with the well-known generalized cobar and bar constructions associated to a monad and a comonad, proving in each case an important technical lemma that holds when the underlying category is a model category.  We then discuss simplicial enrichment of Eilenberg-Moore categories of (co)algebras and establish crucial simplicial adjunction isomorphisms (Proposition \ref{prop:adj-iso}).

\subsection{Generalized cobar and bar constructions}

Given a  monad on a category $\cat C$, we can canonically associate to any object in $\cat C$ a very important and well-known object in $\cat C^{\bold \Delta}$, which plays a crucial role in our discussion of derived completion and homotopic descent.

\begin{defn} Let $\T=(T,\mu,\eta)$ be a monad on a category $\cat C$. The \emph{cosimplicial $\T$-cobar construction} is a functor
$$\Om ^\bullet_{\T}:\cat C\to \cat C ^{\mathbf \Delta}$$
defined for $C\in \ob \cat C$ by
$$\Om_{\T}^{n}C= T^{n+1}(C)$$
and for all $0\leq i\leq n$,
$$d^{i}=T^{i} \eta_{T^{n-i}C}:\Om _{\T}^n C\to \Om_{\T}^{n+1}C,$$
while for all $0\leq j\leq n$,
$$s^j=T^{j}\mu_{ T^{n-j}C}:\Om _{\T}^{n+1} C\to \Om_{\T}^{n}C,$$
where $T^0:=Id_{\cat C}$.
\end{defn}

\begin{rmk}  The $\T$-cobar construction is naturally coaugmented by $\eta_{C}:C\to TC$ for all $C\in \ob \cat C$, which gives rise in the obvious way to a morphism of cosimplicial objects
$$\eta_{C}^\bullet: cc^\bullet C\to \Om ^\bullet_{\T}C.$$
\end{rmk} 

The following technical result is a crucial ingredient of the proofs of Theorems \ref{thm:equiv-monad} and \ref{thm:e2-desc}.  For a brief discussion of induced model structures, the reader is referred to appendix \ref{sec:indmodcat}.

\begin{lem}\label{lem:Reedyfib} Let $\T=(T,\mu, \eta)$ be a monad on a model category $\cat M$. Suppose that 
$U^\T:\cat M^\T\to \cat M$ right-induces a model category structure on $\cat M^\T$ and that $U_{\K^\T}:\op D(\T)\to \cat M^\T$ then left-induces a model category structure on $\op D(\T)$.

If $\Om^\bullet_{\T}(Y)$ is Reedy fibrant in $\cat M^{\bold \Delta}$, then $\can_{\K^\T}^{\bold\Delta}\Om^\bullet_{\T}(Y)$ is Reedy fibrant in $\op D(\T)^{\bold \Delta}$.
\end{lem}

\begin{proof} Throughout this proof, we use the notation of appendix \ref{sec:reedy}.  Since $T=U^\T F^\T$, it is immediate that if $j\geq 1$, then there exists $t^j\in \mor \cat M^\T$ such that $s^j=U^\T t^j$.  Moreover, since
$$\xymatrix{T^3 Z \ar [r]^{T\mu_{Z}} \ar [d]_{\mu _{TZ}}&T^2Z\ar [d]^{\mu_{Z}}\\
		T^2 Z\ar[r]^{\mu _{Z}}& TZ}$$
commutes for all objects $Z$ in $\cat M$,  the multiplication map $\mu_{Z}: T^2 Z\to TZ$ is a morphism of $\T$-algebras for all $Z$, whence the existence of $t^0\in \mor \cat M^\T$ such that  $s^0=U^\T t_{0}:T^mY\to T^{m-1}Y$ as well, for all $m$.

It follows that 
\begin{multline*}
M_{n}\Om^\bullet_{\T}Y=\prod_{0\leq i\leq n}\Om_{\T}^nY \egal{\psi_{1}}{\psi_{2}} \prod_{0\leq i<j\leq n-1} \Om_{\T}^{n-1}Y\\ =U^\T\Big(\prod_{0\leq i\leq n}F^\T(T^{n-1}Y) \egal{\upsilon_{1}}{\upsilon_{2}} \prod_{0\leq i<j\leq n-1} F^\T(T^{n-2}Y) )\Big)=:U^\T \op T_{n}(Y),
\end{multline*}
where if  
$$\prod_{0\leq i\leq n}F^\T(T^{n-1}Y)\xrightarrow {p_{i}} F^\T(T^{n-1}Y)\quad\text{and}\quad \prod_{0\leq i<j\leq n-1} F^\T(T^{n-2}Y)\xrightarrow{p_{i,j}} F^\T(T^{n-2}Y)$$ are the obvious projections, then 
$$p_{i,j}\upsilon_{1}=t^{i}p_{j}\quad \text{and}\quad p_{i,j}\upsilon_{2}=t^{j-1} p_{i}.$$
Thus, since $\Om^\bullet_{\T}Y$ is Reedy fibrant by hypothesis and the model category structure on $\cat M^\T$ is right-induced by $U^{\T}$,  the natural map
$$F^\T (T^nY) \to \lim \op T_{n}Y,$$
which is induced by $t_{0},...,t_{n-1}:F^\T (T^nY) \to F^\T (T^{n-1} Y)$, is a fibration for all $n\geq0$.  Note that we are also using here that $U^\T$ commutes with limits, as it is a right adjoint.

Applying the right Quillen functor $F_{\K^\T}:\cat M^\T \to \op D(\T)$, we obtain fibrations
$$F_{\K^\T}F^\T (T^n Y)\to F_{\K^\T}( \lim \op T_{n}Y)\cong \lim F_{\K^\T}\big(\op T_{n }(Y)\big)$$
for all $n$.  On the other hand, as already observed above,
$$F_{\K^\T}F^\T (T^n Y) = \can_{\K^\T}(\Om ^n_{\T}Y),$$
which generalizes easily to 
$$F_{\K^\T}\big(\op T_{n }(Y)\big)=M_{n}\can_{\K^\T}^{\bold\Delta} (\Om^\bullet_{\T}Y).$$
We conclude that $\can_{\K^\T}^{\bold\Delta} (\Om^\bullet_{\T}Y)$ is Reedy fibrant.
\end{proof}

\begin{rmk}\label{rmk:Reedyfib} When considering derived completion and homotopic descent later in this paper, we will be particularly interested in monads $\T$ such that $\Om ^\bullet_{\T}$ sends fibrant objects to fibrant objects. This fibrancy condition, though constraining, is not impossible to fulfill.   For example, if $\cat M=\cat {sSet}$ and, for all simplicial sets $X$, $TX$ is the simplicial set underlying a simplicial group, then $\Om^\bullet_{\T}X$ is always Reedy fibrant \cite [X.4.10]{bousfield-kan}.
\end{rmk}

Of course, there is an equally important dual construction for comonads.

\begin{defn} Let $\K=(K,\Delta, \ve)$ be a comonad on $\cat D$.  The \emph{simplicial $\K$-bar construction} is a functor
$$\Bar ^\K_{\bullet}:\cat D\to \cat D ^{\mathbf \Delta^{op}}$$
defined for $D\in \ob \cat D$ by
$$\Bar ^\K_{n}D= K^{n+1}(D)$$
and for all $0\leq i\leq n$,
$$d_{i}=K^{i} \ve_{K^{n-i}D}:\Bar ^\K_{n} D\to \Bar ^\K_{n-1}D,$$
while for all $0\leq j\leq n$,
$$s_{j}=K^{j}\Delta_{ K^{n-j}D}:\Bar ^\K_{n} D\to \Bar ^\K_{n+1} D,$$
where $K^0:=Id_{\cat D}$.
\end{defn}

\begin{rmk}  The $\K$-bar construction is naturally augmented by $\ve_{D}:KD\to D$ for all $D\in \ob \cat D$, which gives rise in the obvious way to a morphism of simplicial objects
$$(\ve_{D})_{\bullet}:\Bar ^\K_{\bullet}D\to cs_{\bullet}D.$$
\end{rmk} 

To prove Theorems \ref{thm:equiv-comonad} and \ref{thm:e2-codesc}, we need the following lemma.  The proof is strictly dual to that of Lemma \ref{lem:Reedyfib}.

\begin{lem}\label{lem:Reedycofib} Let $\K=(K,\Delta, \ve)$ be a monad on a model category $\cat M$. Suppose that 
$U_{\K}:\cat M_{\K}\to \cat M$ left-induces a model category structure on $\cat M_{\K}$ and that $U^{\T_{\K}}:\op D^{co}(\K)\to \cat M_{\K}$ then right-induces a model category structure on $\op D(\T)$.

If $\Bar^\K_{\bullet}X$ is Reedy cofibrant in $\cat M^{\bold \Delta^{op}}$, then $(\can^{\T_{\K}})^{\bold\Delta^{op}}(\Bar^\K_{\bullet}X)$ is Reedy cofibrant in $\op D^{co}(\K)^{\bold \Delta^{op}}$.
\end{lem}

\begin{rmk}\label{rmk:Reedycofib}  An analysis of the cofibrancy constraint on $\Bar^\K_{\bullet}X$, dual to that in Remark \ref{rmk:Reedyfib}, shows that it is not an unreasonble condition to impose.  For example, every object $\cat{sSet}^{\bold\Delta^{op}}$ is Reedy cofibrant \cite[IV.32]{goerss-jardine}, so in particular if $\K$ is any comonad on $\cat {sSet}$, then $\Bar^\K_{\bullet}X$ is Reedy cofibrant for all simplicial sets $X$.
\end{rmk}

\subsection{Simplicial enrichments of Eilenberg-Moore categories}\label{sec:simplenr}

We begin by observing that simplicial enrichment of a category  induces simplicial enrichment of its associated Eilenberg-Moore categories of (co)algebras over a (co)monad with underlying simplicial endofunctor.  We refer the reader to appendix \ref{sec:simplmodelcat} for the terminology and notation we use here.

\begin{lem}\label{lem:simplenr} If $\cat S$ is a simplicially enriched category, $\T=(T,\mu,\eta)$ is a monad on $\cat S$ such that $T$ is a simplicial functor, and $\K=(K,\Delta, \ve)$ is a comonad on $\cat S$ such that $K$ is a simplicial functor, then $\cat S^\T$ and $\cat S_{\K}$ are also naturally simplicially enriched.  In particular,  
\begin{enumerate}
\item $\map_{\cat S^\T}:(\cat S^\T)^{op} \times \cat S^\T\to \cat {sSet}$ is defined
for any $\T$-algebras $(A,m)$ and $(A',m')$ by
$$\map _{\cat S^\T}\big((A,m),(A'm')\big)=\operatorname{equal} \big(\map_{\cat S} (A,A') \egal {m^*}{m'_{*}\circ T_{A,A'}}\map_{\cat S}(T A, A')\big);$$
and 
\item $\map_{\cat S_{\K}}:(\cat S_{\K})^{op} \times \cat S_{\K}\to \cat {sSet}$ is defined
for any $\K$-coalgebras $(C,\delta)$ and $(C',\delta')$ by
$$\map _{\cat S_{\K}}\big((C,\delta),(C',\delta')\big)=\operatorname{equal} \big(\map_{\cat S} (C,C') \egal {\delta'_{*}}{\delta^* \circ K_{C,C'}}\map_{\cat S}(C, KC')\big).$$
\end{enumerate}
\end{lem}

\begin{rmk} In the lemma above, equalizers of simplicial maps are understood to be constructed levelwise using the following explicit model for the equalizer of two set maps.
$$\operatorname{equal}\big(X\egal{f}{g} Y)=\{x\in X\mid f(x)=g(x)\}.$$
\end{rmk}

\begin{proof} (1) For any objects $X,Y,Z$ in $\cat S$, let 
$$c:\map_{\cat S} (Y,Z)\times \map_{\cat S} (X,Y)\to \map_{\cat S} (X,Z)$$
denote the composition map of the simplicial enrichment, and let
$$i_{X}:*\to \map _{\cat S}(X.X)$$
denote the unit map.

As defined above, $\map _{\cat S^\T}\big((A,m),(A'm')\big)$ is a simplicial subset of $\map_{\cat S} (A,A')$, for all $\T$-algebras $(A,m)$ and $(A',m')$.  Since it is obvious that 
$$\map _{\cat S^\T}\big((A,m),(A'm')\big)_{0}=S^\T\big((A,m), (A',m')\big),$$
to prove that the definition above of $\map _{\cat S^\T}(-,-)$ gives a simplicial enrichment of $\cat S^\T$,
it suffices to check:
\begin{itemize}
\item[(a)] $c(f',f)\in \map _{\cat S^\T}\big((A,m),(A''m'')\big)_{n}$ for all $f\in \map _{\cat S^\T}\big((A,m),(A'm')\big)_{n}$ and $f'\in \map _{\cat S^\T}\big((A',m'),(A''m'')\big)_{n}$; and
\item [(b)]  $i_{A}(*)\in \map _{\cat S^\T}\big((A,m),(A,m)\big)_{0}$ for all possible multiplications $m:TA\to A$.
\end{itemize}

To prove (a), note that $f\in \map _{\cat S^\T}\big((A,m),(A'm')\big)_{n}$ and $f'\in \map _{\cat S^\T}\big((A',m'),(A''m'')\big)_{n}$ if and only if
$$m^*(f)=m'_{*}\big(T_{A,A'}(f)\big)\quad\text{and}\quad (m')^*(f')=m''_{*}\big(T_{A',A''}(f')\big).$$
Lemma \ref{lem:simplcomp} therefore implies that
$$m^*\big(c(f',f)\big)=m''_{*}\Big(c\big(T_{A',A''}(f'),T_{A,A'}(f)\big)\Big).$$
On the other hand, since $T$ is a simplicial functor,
$$c\big(T_{A',A''}(f'),T_{A,A'}(f)\big)=T_{A,A''}\big(c(f',f)\big),$$
and we can conclude that $c(f',f)\in  \map _{\cat S^\T}\big((A,m),(A''m'')\big)_{n}$.

The following sequence of equalities, which are true for all $\T$-algebras $(A,m)$, establishes that condition (b) holds as well.
\begin{multline*}
m^*\big( i_{A}(*)\big)=c\big(i_{A}(*), m)=m\\
=c\big(m, i_{TA}(*)\big)=c\Big(m, T_{A,A}\big(i_{A}(*)\big)\Big)=m_{*}\Big(T_{A,A}\big(i_{A}(*)\big)\Big).
\end{multline*}

The proof of (2) is strictly dual to the proof of (1) and therefore left to the reader.

\end{proof}

Tensoring and cotensoring over $\cat {sSet}$ can also sometimes carry over to Eilenberg-Moore categories of (co)algebras. Note that we use below the notation $\theta$ and $\tau$ introduced in Theorem \ref{thm:simplfunctor} and Remark \ref{rmk:tau}.

\begin{lem}\label{lem:tens-cotens}  Let $\cat S$ be a simplicially enriched category that is tensored and cotensored over $\cat {sSet}$.  Let $\T=(T,\mu,\eta)$ be  a monad on $\cat S$ and $\K=(K,\Delta, \ve)$  a comonad on $\cat S$.
\begin{enumerate}
\item If $T$ is a simplicial functor, then the Eilenberg-Moore category $\cat S^\T$  of $\T$-algebras is cotensored over $\cat {sSet}$, via
$$(-)_{\T}^{(-)}: \cat S^\T\times \cat {sSet}^{op} \to \cat S^\T: \big((A,m), L)\mapsto (A^L, m^L\circ \tau_{A,L}).$$
\item If $K$ is a simplicial functor, then the Eilenberg-Moore category $\cat S_{\K}$ of $\K$-coalgebras is tensored over $\cat {sSet}$, via
$$-\overset \K\otimes -:\cat S_{\K}\times \cat {sSet}\to \cat S_{\K}:\big((C,\delta), L\big)\mapsto \big(C\otimes L, \theta_{C,L}\circ (\delta \otimes L)\big).$$
\end{enumerate}
\end{lem}

The proof of this lemma is quite elementary and thus left to the reader. In particular, it is easy to check  that there are natural isomorphisms
$$\map_{\cat S^\T}\big((A,m), (A',m')_{\T}^L)\cong \map _{\cat{sSet}}\Big(L, \map_{\cat S^\T}\big((A,m), (A',m')\big)\Big)$$
for all $\T$-algebras $(A,m)$ and $(A',m')$ and all simplicial sets $L$, as well as
$$\map_{\cat S_{\K}}\big((C,\delta)\overset \K\otimes L, (C', \delta')\big)\cong\map_{\cat{sSet}}\Big(L, \map_{\cat S_{\K}}\big( (C,\delta), (C',\delta')\big)\Big)$$
for all $\K$-coalgebras $(C,\delta)$ and $(C', \delta')$ and all simplicial sets $L$.

When $\cat S$ is cotensored (respectively, tensored) over $\cat {sSet}$, as well as simplicially enriched, then the mapping spaces for $\T$-algebras (respectively, $\K$-coalgebras) satisfy a very useful simplicial adjunction property.

\begin{lem} \label{lem:simpl-adj} Let $\cat S$ be a simplicially enriched category,  and let $\T=(T,\mu,\eta)$  and $\K=(K,\Delta, \ve)$ be a monad and  a comonad on $\cat S$ such that $T$ and $K$ are simplicial functors.
\begin{enumerate}
\item If $\cat S$ is cotensored over $\cat {sSet}$, then for all objects $X$ in $\cat S$ and all $\T$-algebras $(A,m)$, there is a natural isomorphism
$$\map_{\cat S}\big(X, U^\T(A,m)\big)\cong \map_{\cat S^\T}\big(F^\T X, (A,m)\big).$$
\item If $\cat S$ is tensored over $\cat {sSet}$, then for all objects $Y$ in $\cat S$ and all $\K$-coalgebras $(C,\delta)$, there is a natural isomorphism
$$\map_{\cat S}\big( U_{\K}(C,\delta), Y)\cong \map_{\cat S_{\K}}\big((C,\delta), F_{\K}Y\big).$$
\end{enumerate}
\end{lem}

\begin{proof} (1) Recall the definition of the cotensoring of $\cat S^\T$ over $\cat {sSet} $ from Lemma \ref{lem:tens-cotens}.  For all $n$,
\begin{align*}
\map_{\cat S^\T}\big(F^\T X,(A,m)\big)_{n}&=\cat S^\T\big(F^\T X,  (A,m)^{\Delta[n]}\big)\\
	&\cong \cat S\Big(X, U^\T\big((A,m)^{\Delta[n]}\big)\Big)\\
	&=\cat S(X,A^{\Delta[n]})\\
	&=\map_{\cat S}(X,A)_{n}.
\end{align*}

The proof of (2) is dual to the proof above, using the tensoring of $\cat S$, rather than the cotensoring.
\end{proof}

As a consequence of this lemma, we obtain natural isomorphisms that are the key to proving our homotopical versions of the classical characterization of (co)descent.

\begin{cor}\label{cor:adj-can} Let $\cat S$ be a simplicially enriched category,  and let $\T=(T,\mu,\eta)$  and $\K=(K,\Delta, \ve)$ be a monad and  a comonad on $\cat S$ such that $T$ and $K$ are simplicial functors, and $\eta$ and $\ve$ are simplicial natural transformations.
\begin{enumerate}
\item If $\cat S$ is cotensored over $\cat {sSet}$, then for all objects $X,Y$ in $\cat S$, there is a natural isomorphism
$$\map_{\cat S}\big(X, \operatorname{Prim}_{\K^\T}\circ \can_{\K^\T}(Y)\big)\cong \map_{\op D(\T)}\big(\can_{\K^\T}(X),  \can_{\K^\T}(Y)\big).$$
\item If $\cat S$ is tensored over $\cat {sSet}$,  then for all objects $X,Y$ in $\cat S$, there is a natural isomorphism
$$\map_{\cat S}\big(Q^{\T_{\K}}\circ \can^{\T_{\K}}(X), Y\big)\cong \map_{\op D^{co}(\K)}\big(\can^{\T_{\K}}(X),  \can^{\T_{\K}}(Y)\big).$$
\end{enumerate}
\end{cor}

\begin{proof} (1) Using the formulas in Lemma \ref{lem:simplenr} and the formula for $\can_{\K^\T}$ from Remark \ref{rmk:chardescdat}, we calculate that
\begin{align*}
\map_{\op D(\T)} &\big(\can_{\K^\T}(X),  \can_{\K^\T}(Y)\big)\\
&=\operatorname{equal} \big(\map_{\cat S^\T} (F^\T X,F^\T Y) \egal {(T\eta_{Y})_{*}}{T\eta_{X}^* \circ T_{TX,TY}}\map_{\cat S^\T}(F^\T X, F^\T TY)\big)\\
&\overset {(a)}\cong\operatorname{equal} \big(\map_{\cat S} ( X,T Y) \egal {(T\eta_{Y})_{*}}{\eta_{X}^* \circ T_{X,TY}}\map_{\cat S}( X, T^2Y)\big)\\
&\overset {(b)}\cong \map _{\cat S}\big(X,\operatorname{equal}(TY\egal {(T\eta_{Y})_{*}}{(\eta_{TY})_{*}} T^2Y)\big)\\
&=\map_{\cat S}\big(X, \operatorname{Prim}_{\K^\T}\circ \can_{\K^\T}(Y)\big).
\end{align*}
Here, isomorphism (a) is a consequence of Lemma \ref{lem:simpl-adj}(1), while isomorphism (b) follows from the fact that $\eta$ is a simplicial natural transformation.

The proof of (2) is dual to the proof above.
\end{proof}
We now show that the simplicial enrichments we have defined on categories of algebras and coalgebras are compatible with the enrichment of the underlying category, in the following sense.

\begin{prop}\label{prop:simplfunct} Let $\cat S$ be a simplicially enriched category that is tensored and cotensored over $\cat {sSet}$.  If $\T=(T,\mu,\eta)$   is a monad on $\cat S$ such that $T$ is a simplicial functor, and $\mu$ and $\eta$ are simplicial natural transformations, then
the free $\T$-algebra functor $F^\T:\cat S\to \cat S^\T$ and the canonical descent data functor $\can _{\K^\T}:\cat S\to \op D(\T)$ are both simplicial functors.

Dually, if $\K=(K,\Delta, \ve)$   is a comonad on $\cat S$ such that $K$ is a simplicial functor, and $\Delta$ and $\ve$ are simplicial natural transformations, then
the free $\K$-coalgebra functor $F_{\K}:\cat S\to \cat S_{\K}$ and the canonical codescent data functor $\can^{\T_{\K}}:\cat S\to \op D^{co}(\K)$ are both simplicial functors.
\end{prop}

\begin{proof}  We treat the monad case and leave the strictly dual, comonad case to the reader.  The idea underlying this proof is that if $f:X\to Y$ is any morphism in $\cat S$, then $Tf$ is a morphism of $\T$-algebras between $F^\T X$ and $F^\T Y$ and also a morphism of descent data from $\can _{\K^\T}(X)$ to $\can _{\K^\T}(Y)$.  We simply generalize this argument to simplices of positive dimension.

Note that $T^2$ is a simplicial functor because $T$ is, so that it makes sense to require that $\mu$ be a simplicial natural transformation.  

That $\mu$ is a simplicial natural transformation means that there is a family of simplicial maps
$$\{\mu_{X}:*\to \map_{\cat S}(T^2X, TX)\mid X\in \ob \cat S\}$$
such that 
 $$(\mu _{Y})_{*}\circ T_{TX,TY}\circ T_{X,Y}=(\mu _{Y})_{*}\circ T^2_{X,Y}=(\mu_{X})^*\circ T_{X,Y}.$$  
 It follows 
$$T_{X,Y}:\map_{\cat S}(X,Y)\to \map _{\cat S}(TX,TY)$$
factors through $\map_{\cat S^\T}(F^\T X, F^\T Y)$.  Let
$$F^\T_{X,Y} :\map_{\cat S}(X,Y)\to \map_{\cat S^\T}(F^\T X, F^\T Y)$$
denote the corestriction of $T_{X,Y}$. The compatibility of $F^\T_{-,-}$ with simplicial composition and identities is an immediate consequence of the same properties of $T_{-,-}$.

To treat the canonical descent data functor, recall from Remark \ref{rmk:chardescdat} that for all $X\in \ob \cat S$,
$$\can_{\K^\T}(X)=(TX, \mu _{X}, T\eta_{X}),$$
so that 
\begin{multline*}
\map_{\op D(\T)}\big(\can_{\K^\T}(X), \can_{\K^\T}(Y)\big)_{n}=\\
\big\{g\in \map_{\cat S^\T}(F^\T X, F^\T Y)\mid (T\eta_{X})^*\big(T_{TX,TY}(g)\big)=(T\eta_{Y})_{*}(g)\big\}.\end{multline*}

Exactly as above, the fact that $\eta: Id\to T$ is a simplicial natural transformation tells us that $F^\T_{X,Y} $ factors through $\map_{\op D(\T)}\big(\can_{\K^\T}(X), \can_{\K^\T}(Y)\big)$.  We can therefore let 
$$(\can_{\K^\T})_{X,Y}:\map_{\cat S}(X,Y)\to\map_{\op D(\T)}\big(\can_{\K^\T}(X), \can_{\K^\T}(Y)\big)$$
denote the corestriction of $F^\T_{X,Y}$ and conclude that $\can_{\K^\T}$ is a simplicial functor.
\end{proof}

\begin{rmk} Under the hypotheses of Proposition \ref{prop:simplfunct}, $U^\T: \cat S^\T\to \cat S$ and $U_{\K}:\cat S_{\K}\to \cat S$ are obviously simplicial functors:  for any $(A,m),(A',m')\in \ob \cat S^\T$, 
$$U^\T_{(A,m),(A',m')}:\map_{\cat S^\T}\big((A,m),(A',m')\big) \to \map _{\cat S}(A,A')$$
is just the inclusion.  The components of $U_{\K}$ are also inclusions.
\end{rmk}

The next result, in which we further probe the nature of $\can _{\K^\T}$ and $\can ^{ \T_{\K}}$ as simplicial functors, is crucial to the proof of our criteria for homotopic (co)descent (Theorems \ref{thm:equiv-monad} and \ref{thm:equiv-comonad}) in section \ref{sec:htpicdesc}.

\begin{prop}\label{prop:adj-iso}  Let $\cat S$ be a simplicially enriched category.  If $\T=(T,\mu,\eta)$   is a monad on $\cat S$ such that $T$ is a simplicial functor, and $\mu$ and $\eta$ are simplicial natural transformations, then 
$$(\can_{\K^\T})_{X,TY}:\map_{\cat S}(X,TY)\to \map_{\op D(\T)}\big(\can_{\K^\T}(X), \can_{\K^\T}(TY)\big)$$
is an isomorphism for all $X, Y\in \ob \cat S$.

Dually, if $\K=(K,\Delta, \ve)$   is a comonad on $\cat S$ such that $K$ is a simplicial functor, and $\Delta$ and $\ve$ are simplicial natural transformations, then 
$$(\can^{\T_{\K}})_{KX,Y}:\map_{\cat S}(KX,Y)\to \map_{\op D^{co}(\K)}\big(\can^{\T_{\K}}(KX), \can^{\T_{\K}}(Y)\big)$$
is an isomorphism for all $X, Y\in \ob \cat S$.
\end{prop}

\begin{proof}  As usual, we prove the monad case and leave the comonad case to the reader.  The idea of the proof is to use the $(F^\T, U^\T)$-adjunction and the $(U_{\K^\T},F_{\K^\T})$-adjunction to prove the existence of isomorphisms
$$\map_{\cat S}(X,TY)\overset \alpha{\underset \cong\rightarrow} \map_{\cat S^\T}(F^\T X, F^\T Y) \overset \beta{\underset \cong\leftarrow}\map_{\op D(\T)}\big(\can_{\K^\T}(X), \can_{\K^\T}(TY)\big)$$
such that $\alpha=\beta\circ (\can _{\K^\T})_{X,TY}$,
implying that $(\can _{\K^\T})_{X,TY}$ is also an isomorphism, as desired.

It is not enough simply to apply Lemma \ref{lem:simpl-adj}, as we need an explicit description of the isomorphism $\alpha$.  Moreover, since $\cat S^\T$ may not be tensored over $\cat {sSet}$ even if $\cat S$ is, we cannot suppose that $\op D(\T)$ is tensored and therefore cannot necessarily apply Lemma \ref{lem:simpl-adj} to obtain the isomorphism $\beta$.

The simplicial map $\alpha$ is defined by
$$\alpha (g)=(\mu_{Y})_{*}\big(F^\T_{X,TY} (g)\big),$$
for any $g\in\map_{\cat S}(X,TY)$ and its inverse $\alpha':\map_{\cat S^\T}(F^\T X, F^\T Y)\to \map_{\cat S}(X,TY)$  by
$$\alpha'(h)=(\eta_{X})^*(h),$$
for any $h\in \map_{\cat S^\T}(F^\T X, F^\T Y)$.
To see that $\alpha'$ is indeed the inverse of $\alpha$, observe that
\begin{align*}
\alpha'\circ \alpha(g)&=(\eta_{X})^* (\mu_{Y})_{*}\big(F^\T_{X,TY}(g)\big)\\
		&\overset {(1)}= (\mu_{Y})_{*}(\eta_{X})^*\big(F^\T_{X,TY}(g)\big)\\
		&\overset {(2)}=(\mu_{Y})_{*}(\eta_{TY})_{*}(g)\\
		&=(\mu_{Y}\eta_{TY})_{*}(g)\\
		&=g.
\end{align*}
Equality (1) follows from the associativity of simplicial composition, while equality (2) holds since $\eta$ is a simplicial natural transformation.
Furthermore,
\begin{align*}
\alpha\circ\alpha '(h)&=(\mu_{Y})_{*} F^\T _{X,TY}\big((\eta_{X})^*(h)\big)\\
				&\overset {(3)}=(\mu_{Y})_{*} (\eta_{TX})^* F^\T_{TX,TY}(h)\\
				&\overset{(4)}=(\eta_{TX})^*(\mu_{Y})_{*}F^\T_{TX,TY}(h)\\
				&\overset{(5)}=(\eta_{TX})^*(\mu_{X})^*(h)\\
				&=(\mu_{X}\eta_{TX})^*(h)\\
				&=h.
\end{align*}
Equality (3) holds because the simplicial functor $F^\T$ respects composition. Associativity of simplicial composition implies equality (4), while equality (5) is a consequence of the fact that $h$ is not just an element of $\map _{\cat S}(TX,TY)$, but actually an element of its simplicial subset $\map_{\cat S^\T}(F^\T X, F^\T Y)$.

As for the simplicial map $\beta$, it is given by
$$\beta (k)=(\mu _{Y})_{*}(k),$$
for all $k\in\map_{\op D(\T)}\big(\can_{\K^\T}(X), \can_{\K^\T}(TY)\big)$, while its inverse $\beta'$ satisfies
$$\beta'(h)=(\eta_{TX})^*\big ( (F_{\K^\T})_{F^\T X, F^\T Y}(g)\big).$$
This definition of $\beta$ makes sense, since 
$$\map_{\op D(\T)}\big(\can_{\K^\T}(X), \can_{\K^\T}(TY)\big)\subset \map _{\cat S^\T}(F^\T X, F^\T (TY)\big)$$
by construction, and the object in $\cat S$ underlying $F^\T (TY)$ is $T^2Y$.  Moreover, $\mu _{Y}$ is itself a map of $\T$-algebras.  On the other hand, since $F_{\K^\T}F^\T Y=\can _{\K^\T}(TY)$ and, as is easily checked, $\eta _{TX}:TX \to T^2X$ underlies a morphism of descent data $\can _{\K^\T}(X)\to \can _{\K^\T}(TX)$, the definition of $\beta'$ is also acceptable.
The proof that $\beta$ and $\beta'$ are mutually inverse strongly resembles the proof above that $\alpha'=\alpha^{-1}$, and we therefore omit it.  

Since the components of the simplicial functor $\can_{\K^\T}$ are corestrictions of the components of $T$, it is clear that $\alpha=\beta \circ (\can_{\K^\T})_{X,TY}$, and we can therefore conclude.  
\end{proof}

\section{Derived completion and cocompletion}\label{sec:derivedcompletion}

In this section we introduce the notion of the derived (co)completion along a (co)monad, which is clearly strongly influenced by the definition of Bousfield-Kan $p$-completion of simplicial sets \cite{bousfield-kan}, as well as by Carlsson's derived completion of module spectra along ring spectra \cite{carlsson}.  Derived (co)completion plays an important role both in the formulation of conditions under which homotopic (co)descent holds (cf. Theorems \ref{thm:equiv-monad} and \ref{thm:equiv-comonad})  and in the analysis of the (co)descent spectral sequences (cf. Theorems \ref{thm:e2-desc} and \ref{thm:e2-codesc}).  To conclude this section, we describe the close relationship between assembly maps out of homotopy left Kan extensions and derived cocompletion (Corollary \ref{cor:assembly}) and sketch possible applications to the Baum-Connes and Farrell-Jones conjectures.  Dually, coassembly maps into homotopy right Kan extensions and derived completions are likewise closely related, which should have interesting implications for embedding calculus.

\subsection{Derived completion along a monad}\label{sec:completion}

\begin{convention}\label{conv:descss} Unless otherwise specificed, $\T=(T,\mu, \eta)$ denotes a monad on a simplicial model category such that $T$ is a simplicial functor.  We assume moreover that $\cat M^\T$ is endowed with the simplicial enrichment of Lemma \ref{lem:simplenr} and with a model category structure right-induced by $U^\T:\cat M^\T\to \cat M$ (Definition \ref{defn:induced}).
\end{convention}

By analogy with the definition of $p$-complete topological spaces, where $p$ is a prime, we formulate the following definitions.  We are also inspired by Carlsson's definition of the derived completion of module spectra along commutative ring spectra \cite{carlsson}.  We begin by enlarging the class of morphisms considered to be equivalences.

\begin{defn} \label{defn:Tequiv} Let $X$ and $Y$ be cofibrant objects in $\cat M$. A morphism $f:X\to Y$ of in $\cat M$ is a \emph{$\T$-equivalence} if the induced morphism of simplicial sets
 $$\map_{\cat M} (Y, U^\T A )\xrightarrow {f^*}\map_{\cat M} (X,U^\T A)$$
 is a weak equivalence for all fibrant $\T$-algebras $A$. 
 \end{defn}
 
 It is not difficult to characterize $\T$-equivalences, at least under reasonable conditions.
 
 \begin{lem}\label{lem:char-Tequiv} Let $f\in \cat M(X,Y)$, where $X$ and $Y$ are cofibrant. If $Tf$ is a weak equivalence, then $f$ is a $\T$-equivalence.  If $\cat M^\T$ admits a simplicial model category structure, then the converse holds as well.
 \end{lem}
 
 In particular, since $F^\T$ is a left Quillen functor, any weak equivalence with cofibrant source and target  is a $\T$-equivalence.
 
 \begin{proof}   If $Tf$ is a weak equivalence, then $F^\T f$ is a weak equivalence of cofibrant $\T$-algebras, since $F^\T$ is a left Quillen functor.  It follows that 
   $$(F^\T f)^*:\map_{\cat M^\T}(F^\T Y, A)\xrightarrow\sim \map_{\cat M^\T}(F^\T X, A)$$
   is a weak equivalence of simplicial sets for all fibrant $\T$-algebras $A$, and thus that 
    $$f^*:\map_{\cat M}( Y, U^\T A)\xrightarrow\sim \map_{\cat M}(X, U^\T A)$$
    is as well, by Lemma \ref{lem:simpl-adj} (1).  We conclude that
   $f$ is a $\T$-equivalence.
 
  On the other hand, if  $f$ is a $\T$-equivalence and therefore $(F^\T f)^*$ is a weak equivalence for all fibrant $\T$-algebras $A$, then $F^\T f=Tf$ is itself a weak equivalence, by  \cite[Proposition 9.7.1]{hirschhorn}.
   \end{proof}

\begin{defn}\label{defn:Tcompobj}
A fibrant object $Z$ in $\cat M$ is \emph{$\T$-complete} if for every $\T$-equivalence $f:X\to Y$, the induced morphism of simplicial sets 
 $$\map_{\cat M} (Y, Z )\xrightarrow {f^*}\map_{\cat M} (X,Z)$$
 is a weak equivalence.
 \end{defn} 
 
\begin{rmk} Obviously, if $A$ is a fibrant $\T$-algebra, then $U^\T A$ is a $\T$-complete object of $\cat M$.  In particular, if $TX$ is a fibrant object of $\cat M$, then it is $\T$-complete, since $F^\T X$ is then a fibrant $\T$-algebra.
\end{rmk}

\begin{rmk} It is clear that $\T$-completeness is a homotopy-invariant notion:  if two fibrant objects $Z$ and $Z'$ are weakly equivalent, then $Z$ is $\T$-complete if and only if $Z'$ is $\T$-complete.
\end{rmk}

When the endofunctor underlying the monad $\T$ is homotopically faithful, in the sense of the definition below, $\T$-equivalences and $\T$-complete objects are particularly easy to characterize.

\begin{defn} \label{defn:cofibhtpicfaith} Let $\cat M$ and $\cat {M'}$ be model categories. A functor $F:\cat M\to \cat {M'}$ is \emph{cofibrantly homotopically faithful} if a morphism $f$ in $\cat M$ with cofibrant source and target  is a weak equivalence only if $F(f)$ is a weak equivalence in $\cat {M'}$. 
\end{defn}

The next lemma is an immediate consequence of Lemma \ref{lem:char-Tequiv}.

\begin{lem}\label{lem:faithful} Suppose that  $\cat M^\T$ admits a simplicial model category structure. If  $T$ is homotopically faithful, then a morphism in $\cat M$ between cofibrant objects is a $\T$-equivalence if and only if it is a weak equivalence. Moreover, all fibrant objects in $\cat M$ are then $\T$-complete.
\end{lem}

Without assuming homotopic faithfulness, we can form an interesting class of $\T$-complete objects as follows,  using the $\T$-cobar construction.

\begin{lem}\label{lem:T-comp}  Let $Z$ be an object of $\cat M$.  If the $\T$-cobar construction $\Om ^\bullet_{\T}Z$ is Reedy fibrant, then $\tot \Om ^\bullet _{\T}Z$ is $\T$-complete.
\end{lem}

\begin{proof} Since $\Om ^\bullet_{\T}Z$ is Reedy fibrant, its totalization is a fibrant object of $\cat M$.  Moreover, $\Om ^n_{\T}Z$ must be fibrant for all $n$. Consequently, $F^\T(T^nZ)$ is a fibrant $T$-algebra for all $n$ because $\Om ^n_{\T}Z=T^{n+1}Z=U^\T F^\T (T^nZ)$ and $U^\T$ right-induces the model category structure on $\cat M^\T$.

It follows that if $f:X\to Y$ is a $\T$-equivalence, then  
$$\map_{\cat M} (Y, \Om ^n_{\T}Z )\xrightarrow {f^*}\map_{\cat M} (X,\Om ^n_{\T}Z)$$
is a weak equivalence of simplicial sets for all $n$.  Furthermore, since $X$ and $Y$ are cofibrant, the cosimplicial simplicial sets $\map_{\cat M} (X, \Om ^\bullet_{\T}Z )$ and $\map_{\cat M} (Y, \Om ^\bullet_{\T}Z )$ are both Reedy fibrant. The morphism of cosimplicial simplicial sets induced by $f$,
$$f^\bullet:\map_{\cat M} (Y, \Om ^\bullet_{\T}Z )\to \map_{\cat M} (X, \Om ^\bullet_{\T}Z ),$$
is thus a levelwise weak equivalence of Reedy fibrant objects.  Applying $\tot$, we obtain a weak equivalence of simplicial sets
$$\tot f^\bullet:\tot \map_{\cat M} (Y, \Om ^\bullet_{\T}Z )\to \tot\map_{\cat M} (X, \Om ^\bullet_{\T}Z ).$$

Recall that $\tot$ can be calculated as an equalizer.  Moreover, if $G:\cat D\to \cat M$ is a functor from any small category to a simplicial model category, then
$$\lim _{\cat D}\map_{\cat M} (X, G)\cong \map _{\cat M}(X, \lim _{\cat D}G).$$
Thus,  
$$\tot \map_{\cat M} (W, \Om ^\bullet_{\T}Z )\cong \map _{\cat M}(W, \tot \Om ^\bullet_{\T}Z )$$
for all objects $W$, and we can conclude.
\end{proof}

Lemma \ref{lem:T-comp} justifies the next definition.

\begin{defn}\label{defn:modelTcomp}  Let $\widehat X$ be a fibrant replacement of an object $X$ in $\cat M$. If $\Om ^\bullet _{\T}\widehat X$ is Reedy fibrant, then $\tot \Om ^\bullet _{\T}\widehat X$ is a \emph{model of the derived $\T$-completion of $X$}.
\end{defn}

Homotopy invariance of derived $\T$-completion is guaranteed under the following conditions.

\begin{lem}\label{lem:T-comp-inv} Let $X$ and $Y$ be objects in $\cat M$, and let $\tot \Om ^\bullet _{\T}\widehat X$ and $\tot \Om ^\bullet _{\T}\widehat Y$ be models of their derived $\T$-completions.  If $T$ preserves weak equivalences between fibrant objects, then any weak equivalence $X\xrightarrow \sim Y$ induces a weak equivalence $\tot \Om ^\bullet _{\T}\widehat X\xrightarrow \sim\tot \Om ^\bullet _{\T}\widehat Y$.
\end{lem}

\begin{proof} An elementary argument using the lifting axiom of a model category shows that a weak equivalence $w:X\xrightarrow\sim Y$ induces a weak equivalence $\hat w:\widehat X\xrightarrow \sim \widehat Y$ between the fibrant replacements of $X$ and $Y$. The hypotheses on $T$ imply that $T^n\hat w$ is a weak equivalence for all $n$ and therefore that 
$$T^{\bullet+1} \hat w:\Om ^\bullet_{\T}\widehat X\to \Om ^\bullet _{\T}\widehat Y$$
is a levelwise weak equivalence of Reedy fibrant cosimplicial objects.  Consequently, 
$$\tot T^{\bullet+1} \hat w:\tot \Om ^\bullet_{\T}\widehat X\to \tot\Om ^\bullet _{\T}\widehat Y$$
is a weak equivalence in $\cat M$.
\end{proof}

Lemmas \ref{lem:T-comp} and \ref{lem:T-comp-inv} imply  that all models of the derived $\T$-completion of an object of $\cat M$ are indeed $\T$-complete, as well as weakly equivalent to each other if $T$ preserves weak equivalences between fibrant objects, whence the notation introduced below.

\begin{notn}\label{notn:Tcomp}   Let $X$ be an object in $\cat M$.  If  $\tot \Om ^\bullet _{\T}\widehat X$  is a model of its derived $\T$-completion, then, abusing notation slightly,  we write
$$X^\wedge_{\T}:= \tot \Om ^\bullet _{\T}\widehat X.$$ 
\end{notn}

\begin{rmk}  If $\mathbb R$ is the monad on $\cat {sSet}$ of \cite[I.2]{bousfield-kan}, then for all simplicial sets $X$, our notion of the derived $\mathbb R$-completion of $X$ agrees with that of Bousfield and Kan.
\end{rmk}

Two stronger notions of $\T$-completeness, one of which requires no model category structure, arise naturally in our discussion of homotopic descent. We refer the reader to section \ref{sec:external} for an introduction to external cosimplicial structure and to section \ref{sec:ext} for a discussion of cohomological fibrant resolutions.  

\begin{defn}\label{defn:strong-T}  Let $\T$ be a monad on a category $\cat C$ that admits all finite limits and colimits.  Let $Z$ be an object in $\cat C$, and let $\eta^\bullet : cc^\bullet Z\to \Om^\bullet_{\mathbb T}(Z)$ denote the natural coaugmentation.
\begin{enumerate}
\item The object $Z$  is \emph{strictly $\T$-complete} if there is an external cosimplicial SDR
$$\xymatrix{cc^\bullet Z \ar @<.7ex>[r]^-{\eta^\bullet}&\Om ^{\bullet}_{\T}(Z)\circlearrowright h\ar @<.7ex>[l]^-{\rho^\bullet}}$$(cf. Definition \ref{defn:extSDR}).
\item If $\cat C$ is  actually a simplicial model category, then $Z$ is \emph{strongly $\T$-complete} if the natural coaugmentation $\eta^\bullet$ is a cohomological fibrant resolution of $Z$ (cf., Definition \ref{defn:homres}).
\end{enumerate} 
\end{defn}

\begin{rmk} The terminology used in the definition above was chosen by analogy with that frequently used in category theory for describing the hierarchy of monoidal functors:  any strictly monoidal functor is strongly monoidal, and any  strongly monoidal functor is monoidal.  It follows from Remark \ref{rmk:homres} that any strictly $\T$-complete object in a simplicial model category is strongly $\T$-complete.  Moreover Remark \ref{rmk:conseqcohomfibres} implies that a fibrant object is $\T$-complete if it is strongly $\T$-complete, at least when the associated extended homotopy spectral sequences converge.
\end{rmk}

Strict $\T$-completeness is not an overly restrictive condition.

\begin{lem}\label{lem:exist-strict}  Suppose that $\T$ is the monad associated to an adjunction $$F:\cat C\adjunct{}{} \cat D:U,$$
where $\cat C$ admits all finite limits and colimits.  If $Y$ is any object of $\cat D$, then $UY$ is strictly $\T$-complete.
\end{lem}

\begin{proof} The argument applied here follows a well-trodden path. Let $\ve$ denote the counit of the $(F,U)$-adjunction.  The coaugmented cosimplicial object
$$UY\xrightarrow {\eta_{UY}} \Om^\bullet_{\T}(UY)$$
admits an ``extra codegeneracy''  at each level, i.e., is contractible.  More explicitly, from level $n$ to level $n-1$ we have
$$s^n=T^n(U\ve_{Y}): T^{n+1} (UY) \to T^n(UY),$$
commuting appropriately with the cofaces and the other codegeneracies of the $\T$-cobar construction.  We leave this straightforward calculation to the reader.  

By Proposition \ref{prop:barr}, there is thus an external cosimplicial SDR 
$$\xymatrix{cc^\bullet UY \ar @<.7ex>[r]^-{\eta^\bullet}&\Om ^{\bullet}_{\T}(UY)\circlearrowright h\ar @<.7ex>[l]^-{\rho^\bullet},}$$
i.e., $\eta^\bullet$ is an external homotopy equivalence, and $UY$ is strongly $\T$-complete.
\end{proof}

\begin{rmk}\label{rmk:ex-Tcomp} When applied to  the adjunction $F^\T:\cat C\adjunct{}{} \cat C^\T:U^\T$, for some monad $\T$ on $\cat C$, Lemma \ref{lem:exist-strict} implies that  $U^\T (A,m)$ is strictly $\T$-complete for all $\T$-algebras $(A,m)$.  In particular, $TX$ is strictly $\T$-complete for all objects $X$ in $\cat C$, since $TX=U^\T F^\T X$.
\end{rmk}

A result analogous to Lemma \ref{lem:faithful} holds when the endofunctor underlying the monad satisfies the following variant of faithfulness.

\begin{defn}\label{defn:cosimpfaith}  Let $\cat C$ and $\cat D$ be categories admitting all finite limits and colimits.  A functor $F: \cat C \to \cat D$ is \emph{cosimplicially faithful} if $f^\bullet \in \cat C^{\bold \Delta}(X^\bullet, Y^\bullet)$ is an external homotopy equivalence whenever $F^{\bold \Delta}(f^\bullet)$ is. In other words, the cosimplicial prolongation of $F$ reflects external homotopy equivalences.
\end{defn}

\begin{lem} Let $\T=(T,\mu, \eta)$ be a  monad on a category $\cat C$ admitting all finite limits and colimits. If $T$ is cosimplicially faithful, then every object in $\cat C$ is strictly $\T$-complete.
\end{lem}

\begin{proof}  This proof is similar to that of Lemma \ref{lem:exist-strict}.  Let $\ve$ denote the counit of the $(F^\T,U^\T)$-adjunction, and let $Z$ be any object of $\cat C$.

The coaugmented cosimplicial object
$$cc^\bullet F^\T(Z) \xrightarrow {F^\T\eta_{Z}} (F^\T)^{\bold \Delta}\Om _{\T}^\bullet Z$$
admits an ``extra codegeneracy'' at each level,  i.e., is contractible.  More explicitly, from level $n$ to level $n-1$ we have
$$s^n=\ve_{F^\T T^nZ}: F^\T T^{n+1} Z \to F^\T T^n Z,$$
commuting appropriately with the cofaces and the other codegeneracies.  We leave this straightforward calculation to the reader.  

Applying $(U^\T)^{\bold \Delta}$ to this contractible coaugmented cosimplicial object and then calling on Proposition \ref{prop:barr}, we obtain an external cosimplicial SDR
$$\xymatrix{T^{\bold \Delta} cc^\bullet Z \ar @<.7ex>[r]^-{T^{\bold \Delta}\eta^\bullet}&T^{\bold \Delta}\Om ^{\bullet}_{\T}Z\circlearrowright h\ar @<.7ex>[l]^-{\rho^\bullet}.}$$
Since $T$ is cosimplicially faithful, we can conclude that $\eta^\bullet :cc^\bullet Z \to \Om_{\T}^\bullet Z$ is an external homotopy equivalence, i.e., that $Z$ is strictly $\T$-complete.
\end{proof}

\begin{rmk}\label{rmk:inftyTalg}  Let $\T$ be a monad on a category $\cat C$ admitting all finite limits and colimits. A careful analysis of Definition \ref{defn:strong-T} leads one naturally to view strictly $\T$-complete objects as ``$\infty\text{-}\T$-algebras''.  As follows from Definition \ref{defn:contractible}, fitting the coaugmentation $\eta^\bullet : cc^\bullet Z \to \Om ^{\bullet}_{\T}(Z)$ into an external SDR
$$\xymatrix{cc^\bullet Z \ar @<.7ex>[r]^-{\eta^\bullet}&\Om ^{\bullet}_{\T}(Z)\circlearrowright h\ar @<.7ex>[l]^-{\rho^\bullet}}$$ is equivalent to constructing a family 
$$\{m_{n}:T^nZ\to T^{n-1}Z\mid n\geq 1\}\subset \mor \cat C$$
such that for all $n\geq 1$ and all $0\leq i\leq k-1$,

$$m_{n}\circ T^{i}\eta_{T^{n-i-1}Z}=\begin{cases} T^{i}\eta_{T^{n-i-2}Z}\circ m_{n-1}&:0\leq i\leq n-2\\ Id_{T^{n-1}Z}&:i=n-1,
\end{cases}$$
while
$$m_{n}\circ T^{i}\mu_{T^{n-i-1}Z}=\begin{cases}T^{i}\mu_{T^{n-i-2}Z}\circ m_{n+1}&: 0\leq i\leq n-2\\ m_{n}\circ m_{n+1}&:i=n-1\end{cases}$$

The first two identities are the transcription of the relationship between the ``extra codegeneracy'' and the cofaces of the $\T$-cobar construction, while the last two specify the relationship of the ``extra codegeneracy'' to the other codegeneracies.

It is evident from the definition of a $\T$-algebra that if $(Z,m)$ is a $\T$-algebra, then we can set
$$m_{n}=T^{n-1}m:T^n Z\to T^{n-1} Z$$
for all $n\geq 1$, and the identities above will be satisfied.  This alternate proof that an object  of $\cat C$ underlying a $\T$-algebra is necessarily strictly $\T$-complete (cf. Remark \ref{rmk:ex-Tcomp}) inspires our vision of strictly $\T$-complete objects as $\infty\text{-}\T$-algebras.
\end{rmk}

\subsection{Derived cocompletion along a comonad}\label{sec:cocompletion}

\begin{convention}\label{conv:codescss} Unless otherwise specified, $\K=(K,\Delta, \ve)$ denotes a comonad on a simplicial model category $\cat M$ such that $K$ is a simplicial functor.  We assume moreover that $\cat M_{\K}$ is endowed with the simplicial enrichment of Lemma \ref{lem:simplenr} and that  $U_{\K}:\cat M_{\K} \to \cat M$ left-induces a model category structure on $\cat M_{\K}$ (Definition \ref{defn:induced}).  
\end{convention}

 All of the  proofs in this section are formally dual to those in the previous section, so they are omitted.

Dualizing the construction of the previous section, we formulate the following definitions.  We begin by enlarging the class of morphisms considered to be equivalences.

\begin{defn}  Let $X$ and $Y$ be fibrant objects in $\cat M$. A morphism $f:X\to Y$ of in $\cat M$ is a \emph{$\K$-equivalence} if the induced morphism of simplicial sets
 $$\map_{\cat M} ( U_{\K}C, X )\xrightarrow {f_{*}}\map_{\cat M} (U_{\K}C,Y)$$
 is a weak equivalence for all cofibrant $\K$-coalgebras $C$. 
 \end{defn}
 
 It is again not difficult to characterize $\K$-equivalences, at least under reasonable conditions.
 
 \begin{lem}\label{lem:char-Kequiv}  Let $f\in \cat M(X,Y)$, where $X$ and $Y$ are fibrant. If $Kf$ is a weak equivalence, then $f$ is a $\K$-equivalence.  If $\cat M_{\K}$ admits a simplicial model category structure, then the converse holds as well.
\end{lem}
 
 In particular, since $F_{\K}$ is a right Quillen functor, if a morphism with fibrant source and target is a weak equivalence, then it is a $\K$-equivalence.

\begin{defn}
A cofibrant object $Z$ in $\cat M$ is \emph{$\K$-cocomplete} if for every $\K$-equivalence $f:X\to Y$, the induced morphism of simplicial sets 
 $$\map_{\cat M} (Z,X )\xrightarrow {f_{*}}\map_{\cat M} (Z,Y)$$
 is a weak equivalence.
 \end{defn} 
 
\begin{rmk} Obviously, if $C$ is a cofibrant $\K$-algebra, then $U_{\K}C$ is a $\K$-cocomplete object of $\cat M$.  In particular, if $KX$ is a cofibrant object of $\cat M$, then it is $\K$-cocomplete, since $F_{\K} X$ is then cofibrant $\K$-coalgebra.
\end{rmk}

\begin{rmk} It is clear that $\K$-cocompleteness is a homotopy-invariant notion:  if two cofibrant objects $Z$ and $Z'$ are weakly equivalent, then $W$ is $\K$-cocomplete if and only if $W'$ is $\K$-cocomplete.
\end{rmk}

The next lemma is an immediate consequence of Lemma \ref{lem:char-Kequiv}, once we introduce a slight variation on the notion of homotopic faithfulness formulated in Definition \ref{defn:cofibhtpicfaith}.

\begin{defn} \label{defn:fibhtpicfaith} Let $\cat M$ and $\cat {M'}$ be model categories. A functor $F:\cat M\to \cat {M'}$ is \emph{fibrantly homotopically faithful} if a morphism $f$ in $\cat M$ with fibrant source and target  is a weak equivalence only if $F(f)$ is a weak equivalence in $\cat {M'}$. 
\end{defn}

\begin{lem}\label{lem:cofaithful} Suppose that  $\cat M_{\K}$ admits a simplicial model category structure. If  $K$ is fibrantly homotopically faithful, then a morphism in $\cat M$ between fibrant objects is a $\K$-equivalence if and only if it is a weak equivalence. Moreover, all cofibrant objects in $\cat M$ are then $\K$-cocomplete.
\end{lem}
 
Without assuming homotopic faithfulness, we can form an interesting class of $\K$-cocomplete objects as follows,  using the $\K$-bar construction.

\begin{lem}\label{lem:K-cocomp}  Let $Z$ be an object of $\cat M$.  If the $\K$-bar construction $\Bar^\K_{\bullet}Z$ is Reedy cofibrant, then $|\Bar^\K_{\bullet}Z|$ is $\K$-cocomplete.
\end{lem}

Lemma \ref{lem:K-cocomp} justifies the next definition.

\begin{defn}\label{defn:modelKcocomp}  Let $\widetilde X$ be a cofibrant replacement of an object $X$ in $\cat M$. If $\Bar^\K_{\bullet}\widetilde X$ is Reedy cofibrant, then $|\Bar^\K_{\bullet}\widetilde X|$ is a \emph{model of the derived $\K$-cocompletion of $X$}.
\end{defn}

Homotopy invariance of derived $\K$-cocompletion is guaranteed under the following conditions.

\begin{lem}\label{lem:K-cocomp-inv} Let $X$ and $Y$ be objects in $\cat M$, and let $\Bar^\K_{\bullet}\widetilde X$ and $\Bar^\K_{\bullet}\widetilde Y$ be models of their derived $\K$-cocompletions.  If $K$ preserves weak equivalences between cofibrant objects, then any weak equivalence $X\xrightarrow \sim Y$ induces a weak equivalence $|\Bar^\K_{\bullet}\widetilde X|\xrightarrow \sim |\Bar^\K_{\bullet}\widetilde Y|$.
\end{lem}

Lemmas \ref{lem:K-cocomp} and \ref{lem:K-cocomp-inv} imply  that all models of the derived $\K$-cocompletion of an object of $\cat M$ are indeed $\K$-cocomplete, as well as weakly equivalent to each other if $K$ preserves weak equivalences between cofibrant objects, and thus motivate the notation introduced below.

\begin{notn}\label{notn:Kcocomp}  Let $X$ be an object in $\cat M$.  If  $\Bar^\K_{\bullet}\widetilde X$  is a model of its derived $\K$-cocompletion, then, abusing notation slightly,  we write
$$X^\vee_{\K}:= |\Bar^\K_{\bullet}\widetilde X|.$$ 
\end{notn}

Two stronger notions of $\K$-cocompleteness, one of which requires no model category structure, arise naturally in our discussion of homotopic codescent. We refer the reader to section \ref{sec:external} for an introduction to external simplicial structure and to section \ref{sec:ext} for a discussion of homological cofibrant resolutions.  

\begin{defn}\label{defn:strong-K}  Let $\K$ be a comonad on a category $\cat D$ that admits all finite limits and colimits.  Let $Z$ be an object in $\cat D$, and let $\ve_{\bullet} : \Bar_{\bullet}^\K Z\to cs_{\bullet }Z$ denote the natural augmentation.
\begin{enumerate}
\item The object $Z$  is \emph{strictly $\K$-cocomplete} if there is an external simplicial SDR
$$\xymatrix{cs_{\bullet}Z \ar @<.7ex>[r]^-{\sigma_{\bullet}}&\Bar_{\bullet}^\K Z\circlearrowright h\ar @<.7ex>[l]^-{\ve_{\bullet}}}$$(cf. Definition \ref{defn:extSDR}).
\item If $\cat D$ is  actually a simplicial model category, then $Z$ is \emph{strongly $\K$-cocomplete} if the natural augmentation $\ve_{\bullet}$ is a homological cofibrant resolution of $Z$ (cf., Definition \ref{defn:homres}).
\end{enumerate} 
\end{defn}

\begin{rmk} It follows from Remark \ref{rmk:homres} that any strictly $\K$-cocomplete object in a simplicial  model category is strongly $\K$-complete.  Moreover Remark \ref{rmk:conseqcohomfibres} implies that a cofibrant object is $\K$-cocomplete if it is strongly $\K$-cocomplete, at least when the associated extended homotopy spectral sequences converge.
\end{rmk}

Strict $\K$-cocompleteness is also not an overly restrictive condition.

\begin{lem}  Suppose that $\K$ is the comonad associated to an adjunction $$F:\cat C\adjunct{}{} \cat D:U.$$  If $X$ is any object of $\cat C$, then  $FX$  is strictly $\K$-cocomplete.
\end{lem}

\begin{rmk}\label{rmk:ex-Kcocomp} When applied to  the adjunction $U_{\K}:\cat D_{\K}\adjunct{}{} \cat D:F_{\K}$, for some comonad $\K$ on $\cat D$, this lemma implies that $U_{\K}(C,\delta)$ is strictly $\K$-cocomplete for all $\K$-coalgebras $(C,\delta)$.  In particular, $KY$ is strictly $\K$-cocomplete for all objects $Y$ in $\cat D$, since $KY=U_{\K} F_{\K} Y$.
\end{rmk}

A result analogous to Lemma \ref{lem:cofaithful} holds when the endofunctor underlying the comonad satisfies the following variant of faithfulness.

\begin{defn}\label{defn:simpfaith}  Let $\cat C$ and $\cat D$ be categories admitting all finite limits and colimits.  A functor $F: \cat C \to \cat D$ is \emph{simplicially faithful} if $f_\bullet \in \cat C^{\bold \Delta^{op}}(X_\bullet, Y_\bullet)$ is an external homotopy equivalence whenever $F^{\bold \Delta^{op}}(f_\bullet)$ is. In other words, the simplicial prolongation of $F$ reflects external homotopy equivalences.
\end{defn}

\begin{lem} Let $\K=(K,\Delta, \ve)$ be a  comonad on a category $\cat D$ admitting all finite limits and colimits. If $K$ is simplicially faithful, then every object in $\cat D$ is strictly $\K$-cocomplete.
\end{lem}

\begin{rmk}\label{rmk:inftyKcoalg}  Let $\K$ be a comonad on a category $\cat D$ admitting all finite limits and colimits. A careful analysis of Definition \ref{defn:strong-K} leads one naturally to view strictly $\K$-complete objects as ``$\infty\text{-}\K$-coalgebras''.   The formulas involved are dual to those in Remark \ref{rmk:inftyTalg}, and lead easily to the conclusion that if an object of $\cat D$ underlies a $\K$-coalgebra, then it is strictly $\K$-cocomplete, whence our vision of strictly $\K$-cocomplete objects as $\infty\text{-}\K$-coalgebras.
\end{rmk}

\subsection{Assembly, derived cocompletion and isomorphism conjectures}\label{sec:assembly}
 
 In this section, as an indication of the importance of the theory developed above, we describe the intriguing relationship between assembly maps and derived cocompletion.
 
 Let $\Phi: \cat C\to \cat D$ be a functor between small categories, and let $\cat M$ be a simplicial model category that is combinatorial or cofibrantly generated.  The functor categories $\cat M^{\cat C}$ and $\cat M^{\cat D}$ are then also simplicial model categories that are combinatorial or cofibrantly generated, where fibrations and weak equivalences are defined objectwise.
 
The functor $\Phi$ induces a Quillen pair of adjoint functors
$$\Phi_{*}:\cat M^{\cat C}\adjunct{}{} \cat M^{\cat D}:\Phi^*,$$
where, if $Y:\cat C\to \cat M$ is any functor, then $\Phi_{*}(Y):\cat D\to \cat M$ denotes its left Kan extension, and $\Phi^*$ is given by precomposition with $\Phi$.
We let $\K_{\Phi}=(K_{\Phi}, \Delta_{\Phi},\ve_{\Phi})$ denote the comonad associated to this adjunction, i.e., 
 $$K_{\Phi}:=\Phi_{*}\Phi^*: \cat M^{\cat D}\to \cat M^{\cat D}.$$

Recall that for all $X\in \cat M^{\cat D}$ and all $D\in \ob \cat D$, 
$$K_{\Phi}(X)(D)=\Phi_{*}\Phi^*(X)(D)=\colim_{\cat {Simp}_{\Phi}(D)} (X\circ \operatorname{dom}),$$
where $\cat {Simp}_{\Phi}(D)$ is the subcategory of the overcategory $\cat D/D$ with object set 
$$\{ f:\Phi (C)\to D \in \mor \cat D\mid C\in \ob \cat C\},$$
while for $f:\Phi(C)\to D$ and $f':\Phi (C')\to D$,
$$\cat {Simp}_{\Phi}(D)(f,f')=\big\{g\in \cat C(C,C')\mid f'\circ \Phi (g)=f\big\},$$
and 
$$\operatorname{dom}:\cat {Simp}_{\Phi}(D)\to \cat D: (\Phi (C)\xrightarrow f D)\mapsto  \Phi(C)$$
is the ``domain'' functor.
Since $\cat M$ is a simplicial model category, it is natural to replace the colimit in the formula above for $K_{\Phi}(X)(D)$ by a homotopy colimit, obtaining a homotopy invariant replacement 
$$K^{ho}_{\Phi}(X)(D)=\operatorname{hocolim}_{\cat {Simp}_{\Phi}(D)} (X\circ \operatorname{dom}),$$ 
of the comonad $K_{\Phi}$, which is the \emph{homotopy left Kan extension} of $\Phi^*(X)$.  The \emph{assembly map} of the functor $X: \cat D \to \cat M$ evaluated at an object $D$ is then the canonical map
$$A_{X, D}: K^{ho}_{\Phi}(X)(D)\to X(D).$$

On the other hand, for all $X\in \cat M^{\cat D}$ and all $D\in \ob \cat D$,
$$|\Bar ^{\K_{\Phi}}_{\bullet}(X)(D)|\in \cat M,$$
which is just $X^\vee_{\K_{\Phi}}(D)$ if $\Bar ^{\K_{\Phi}}_{\bullet}(X)$ is Reedy cofibrant, is another natural ``homotopy replacement'' of $K_{\Phi}(X)(D)$, and the natural map
$$B_{X,D}:|\Bar ^{\K_{\Phi}}_{\bullet}(X)(D)|\to X(D)$$
could also reasonably be termed an ``assembly map.''  We describe below the relationship between these two notions of assembly map.

The article \cite{balmer-matthey}  of Balmer and Matthey on model-theoretic formulations of the Baum-Connes and Farrell-Jones conjectures motivates our examination the relationship between $A_{X,D}$ and $B_{X,D}$.  Let $G$ be a discrete group, and let $\cat D$ be the orbit category of $G$, while $\cat C$ is its full subcategory with as objects those quotients $G/H$ where $H$ is virtually cyclic.  Let $\Phi$ denote the inclusion of $\cat C$ into $\cat D$.  Let $X_{G}$ denote a functor from $\cat D$ into some nice simplicial model category of spectra $\cat M$ such that $\pi_{*}\big(X_{G}(G/H)\big)$ is canonically isomorphic to $K_{*}(H)$ for every subgroup $H$ of $G$.  Here, $K_{*}(-)$ can denote any of a number of relevant $K$- or $L$-theories, for which we refer the reader to the extensive literature, but in particular to \cite{balmer-matthey}.

The Isomorphism Conjecture for $G$ at a subgroup $H$ with respect to the chosen $K$- or $L$-theory states that
$$A_{X_{G},G/H}:K^{ho}_{\Phi}(X_{G})(G/H)\to X_{G}(G/H)$$
should be a weak equivalence of spectra.  If $A_{X_{G}, G/H}$ is weakly equivalent to $B_{X_{G},G/H}$ for all $H\leq G$, then the Isomorphism Conjecture holds at all $H\leq G$ if and only if $X_{G}$  is derived $\K_{\Phi}$-cocomplete.    One could apply such an equivalence in two directions:  to obtain examples of $\K_{\Phi}$-cocomplete objects in those cases where the Isomorphism Conjectures are known to hold and to gain potentially useful insight into the meaning of the Isomorphism Conjectures in still-unresolved cases.

Motivated by this discussion of the Baum-Connes and Farrell-Jones conjectures, we formulate the following definition. 

\begin{defn}  Let $\Phi: \cat C\to \cat D$ be a functor between small categories, let $\cat M$ be a simplicial model category that is combinatorial or cofibrantly generated, and let $D$ be an object of $\cat D$.  A functor $X:\cat D\to \cat M$ \emph{satisfies the $\Phi$-Isomorphism Conjecture at $D$} if $A_{X,D}:K^{ho}_{\Phi}(X)(D) \to X(D)$ is a weak equivalence.  If $A_{X,D}$ is a weak equivalence for all $D$, then $X$ \emph{satisfies the $\Phi$-Isomorphism Conjecture}
\end{defn}

The comparison result below is a first step towards characterizing in terms of derived cocompletion those functors that satisfy Isomorphism Conjectures.
 
\begin{thm}\label{thm:assembly}  Let $\Phi: \cat C\to \cat D$ be a full functor between small categories, and let $\cat M$ be a simplicial model category that is combinatorial or cofibrantly generated.  

For all functors $X:\cat D\to \cat M$  and for all $D\in \cat D$, there is a commutative diagram in $\cat M$:
$$\xymatrix {K^{ho}_{\Phi}(X)(D)\ar [dr]^{A_{X,D}}\\ \bullet \ar [u]^{a_{X,D}} \ar [d]_{b_{X,D}}&X(D).\\ |\Bar ^{\K_{\Phi}}_{\bullet}(X)(D)|\ar [ur]_{B_{X,D}}}$$
\end{thm}

In order not to interrupt further the general flow of this article, we refer the reader to appendix \ref{sec:proof} for the proof of this theorem.

The following characterization is an immediate consequence of Theorem \ref{thm:assembly}. 

\begin{cor}\label{cor:assembly} Let $\Phi: \cat C\to \cat D$ be a full functor between small categories, and let $\cat M$ be a simplicial model category that is combinatorial or cofibrantly generated.   A functor $X:\cat D\to \cat M$ such that $a_{X,D}$ and $b_{X,D}$ are weak equivalences for all objects $D$ satisfies the $\Phi$-Isomorphism Conjecture if and only if it is derived $\K_{\Phi}$-cocomplete.
\end{cor}

\begin{rmk}\label{rmk:calculus} It is, of course, possible to dualize all of the discussion above to the case of (homotopy) right Kan extensions and derived completion.  Homotopy right Kan extensions and their associated coassembly maps play an important role, for example, in the embedding calculus, in constructing the stages of the Taylor tower of a contravariant functor $F$ from the category $\cat O(M)$ of open subsets of a manifold $M$ to the category $\cat {Top}$ of topological spaces.  Under reasonable conditions on $F$, the $k^{\text{th}}$-stage $T_{k}F$ of the Taylor tower of $F$ is its homotopy right Kan extension with respect to the inclusion functor $\Phi_{k}:\cat O_{k}(M)^{op}\to \cat O(M)^{op}$, where $\cat O_{k}(M)$ is the full subcategory of $\cat O(M)$ with as objects the disjoint unions of at most $k$ open balls.  The functor $F$ is said to be  \emph{$k$-excisive} if the coassembly map $F\to T_{k}F$ is a weak equivalence.

It would be interesting to determine the relationship between $k$-excision and derived completeness with respect to the monad arising from right Kan extension over $\Phi_{k}$. 
\end{rmk}

\section{Homotopic descent and codescent and associated spectral sequences}\label{sec:htpicdesc}

We are now ready to introduce our theory of homotopic (co)descent for simplicially enriched categories.  Since simplicially enriched categories are a model of $(\infty, 1)$-categories (also called $\infty$-categories by homotopy theorists), our theory can also be considered to be a model for the theory of $\infty$-(co)descent.

In the case of a (co)monad acting on simplicial model category, we provide criteria for  homotopic (co)descent (Theorems \ref{thm:equiv-monad} and \ref{thm:equiv-comonad}), which are  strongly analogous to the Main Theorem in Mandell's paper \cite{mandell}.  The criteria are formulated in terms of strong (co)completion, providing another important motivation for the study of these notions in the previous section. 

In the last part of this section we consider spectral sequences that arise naturally from the action of a (co)monad on a simplicial model category.  The notions of both derived (co)completion and homotopic (co)descent play a role in interpreting these spectral sequences.

Many well-known spectral sequences, such as the Adams and Adams-Novikov spectral sequences and various ``descent'' spectral sequences for bundles, arise as (co)descent spectral sequences in the way we describe here.  In the section \ref{sec:htpicGroth} we sketch certain of these examples. 

\subsection{Homotopic descent}

The definition of homotopic descent requires only simplicial enrichment and not the full structure of a simplicial model category.  We suggest that the reader compare the definition below to Definition \ref{defn:desccat} and Proposition \ref{prop:adj-iso}.

\begin{defn}\label{defn:htpicdesc} Let $\cat S$ be a simplicially enriched category.  Let $\T=(T,\mu,\eta)$   be a monad on $\cat S$ such that $T$ is a simplicial functor, and $\mu$ and $\eta$ are simplicial natural transformations.

If each component
$$(\can_{\K^{\T}})_{X,Y}:\map_{\cat S}(X,Y) \to \map_{\op D(\T)}\big(\can_{\K^{\T}}(X), \can_{\K^{\T}}(Y)\big)$$
of the simplicial functor $\can_{\K^{\T}} :\cat S\to \op D(\T)$ is a weak equivalence of simplicial sets, then $\T$ satisfies \emph{homotopic descent}.  If, in addition, every descent datum  is isomorphic in the path-component category $\pi_{0}\op D(\T)$ (cf. Definition \ref{defn:pathcomp}) to an object of the form $\can_{\K^\T}(X)$, then $\T$ satisfies \emph{effective homotopic descent}. 
\end{defn}

Our criterion for homotopic descent does require model category structure and is expressed in terms of the following important subcategory of the underlying model category.

\begin{notn} If $\cat M$ is a simplicial model category, and $\T=(T,\mu,\eta)$   is a monad on $\cat M$, we let $\cat M^\wedge_\T$ denote the full simplicial subcategory of $\cat M$ determined by the bifibrant, strictly $\T$-complete objects.

The observation in Remark \ref{rmk:ex-Tcomp} implies that if $T$ preserves bifibrant objects, then $T$ fixes $\cat M^\wedge_{\T}$, i.e., $\T$ restricts and corestricts to a monad on $\cat M^\wedge_{\T}$, which we denote $\widehat\T$.
\end{notn}

\begin{thm}\label{thm:equiv-monad}  Let $\cat M$ be a simplicial model category. Let $\T=(T,\mu,\eta)$   be a monad on $\cat M$ such that $T$ is a simplicial functor preserving bifibrant objects, and $\mu$ and $\eta$ are simplicial natural transformations. Suppose that 
$U^\T:\cat M^\T\to \cat M$ right-induces a model category structure on $\cat M^\T$, that $U_{\K^\T}:\op D(\T)\to \cat M^\T$ then left-induces a model category structure on $\op D(\T)$, and that the usual simplicial enrichment of $\op D(\T)$ (Lemma \ref{lem:simplenr}) satisfies the additional axiom of Definition \ref{defn:simplmodel}. 

If $\Om ^\bullet_{\T}$ preserves fibrant objects, then $\widehat \T$ satisfies homotopic descent.
\end{thm}

This theorem implies that a form of ``faithfully flat descent'' holds. Recall the notion of cosimplicial faithfulness from Definition \label{defn:cosimpfaith}.

\begin{cor}\label{cor:ffd} Under the hypotheses of Theorem \ref{thm:equiv-monad}, if $\Om ^\bullet_{\T}$ preserves fibrant objects and $T$ is cosimplicially faithful, then the restriction of $\T$ to the full subcategory of bifibrant objects satisfies homotopic descent. 
\end{cor}

Theorem \ref{thm:equiv-monad} is an immediate consequence of the following proposition, which generalizes Proposition \ref{prop:adj-iso}, in the case where the underlying category is a simplicial model category.

\begin{prop}\label{prop:criterion} Under the hypotheses of Theorem \ref{thm:equiv-monad}, if $Y$ is  fibrant and strictly $\T$-complete, and $\Om^\bullet_{\T}Y$ is Reedy fibrant, then 
$$(\can _{\K^\T})_{X,Y}:\map _{\cat M}(X, Y)\to \map_{\op D(\T)}\big( \can_{\K^\T}(X), \can _{\K^\T}(Y)\big)$$
is a weak equivalence of simplicial sets, for all cofibrant objects $X$.
\end{prop}

\begin{proof} Since $Y$ is strictly $\T$-complete, there is an external cosimplicial SDR
$$\xymatrix{cc^\bullet Y \ar @<.7ex>[r]^-{\eta^\bullet}&\Om ^{\bullet}_{\T}(Y)\circlearrowright h\ar @<.7ex>[l]^-{\rho^\bullet}.}$$
Applying $\can_{\K^\T}$ levelwise, we obtain a new SDR
$$\xymatrix{cc^\bullet \can_{\K^\T}(Y) \ar @<.7ex>[r]^-{\can_{\K^\T}\eta^\bullet}&\can_{\K^\T}^{\bold\Delta}\Om ^{\bullet}_{\T}(Y)\circlearrowright h\ar @<.7ex>[l]^-{\can_{\K^\T}\rho^\bullet}.}$$
By Lemma \ref{lem:Reedyfib} $\can_{\K^\T}^{\bold\Delta}\Om ^{\bullet}_{\T}(Y)$ is Reedy fibrant, which implies that $cc^\bullet \can_{\K^\T}(Y)$ is as well, since a retract of a fibrant object is always fibrant, in any model category.

Consider the following commutative diagram of maps of fibrant cosimplicial simplicial sets, in which the righthand vertical arrow is an isomorphism by Proposition  \ref{prop:adj-iso}.
$$\xymatrix{\map_{\cat M}(X,cc^\bullet Y) \ar @<.7ex>[r]^-{\eta^\bullet_{*}}\ar [d]_{(\can_{\K^\T})_{X, cc^\bullet Y}}&\map_{\cat M}\big(X,\Om ^{\bullet}_{\T}(Y)\big)\circlearrowright h\ar @<.7ex>[l]^-{\rho^\bullet_{*}}\ar [d]_{\cong}^{(\can_{\K^\T})_{X,\Om^\bullet_{\T}Y}}\\
\map_{\op D(\T)}(\can_{\K^\T}(X),cc^\bullet \can_{\K^\T}(Y)\big) \ar @<.7ex>[r]^-{\can_{\K^\T}\eta^\bullet_{*}}&\map_{\op D(\T)}(\can_{\K^\T}(X),\can_{\K^\T}^{\bold\Delta}\Om ^{\bullet}_{\T}(Y)\big)\circlearrowright h\ar @<.7ex>[l]^-{\can_{\K^\T}\rho^\bullet_{*}}}$$
The top and the bottom of this diagram consist of  external cosimplicial SDR's of simplicial sets. 

By Corollary \ref {cor:extequiv-weq}, applying $\tot$ to the diagram above leads to the commutative diagram of fibrant simplicial sets
$$\xymatrix{\tot\map_{\cat M}(X,cc^\bullet Y) \ar @<.7ex>[r]^-{\sim}\ar [d]_{\tot(\can_{\K^\T})_{X, cc^\bullet Y}}&\tot\map_{\cat M}\big(X,\Om ^{\bullet}_{\T}(Y)\big)\ar @<.7ex>[l]^-{\sim}\ar [d]_{\cong}^{\tot(\can_{\K^\T})_{X,\Om^\bullet_{\T}Y}}\\
\tot\map_{\op D(\T)}(\can_{\K^\T}(X),cc^\bullet \can_{\K^\T}(Y)\big) \ar @<.7ex>[r]^-{\sim}&\tot\map_{\op D(\T)}(\can_{\K^\T}(X),\can_{\K^\T}^{\bold\Delta}\Om ^{\bullet}_{\T}(Y)\big)\ar @<.7ex>[l]^-{\sim},}$$
whence 
$$\tot(\can_{\K^\T})_{X, cc^\bullet Y}:\tot\map_{\cat M}(X,cc^\bullet Y)\xrightarrow \sim \tot\map_{\op D(\T)}(\can_{\K^\T}(X),cc^\bullet \can_{\K^\T}(Y)\big)$$
is a weak equivalence.  Since
$$\xymatrix{\tot\map_{\cat M}(X,cc^\bullet Y)\ar [rrr]^-{\tot(\can_{\K^\T})_{X, cc^\bullet Y}}_-{\sim}&&&\tot\map_{\op D(\T)}(\can_{\K^\T}(X),cc^\bullet \can_{\K^\T}(Y)\big)\\
\map_{\cat M}(X,Y)\ar[rrr]^-{(\can_{\K^\T})_{X, Y}}\ar[u]^\cong&&&\map_{\op D(\T)}\big (\can_{\K^\T}(X),\can_{\K^\T}(Y)\big)\ar[u]_{\cong}}$$
commutes, and the vertical arrows are isomorphisms by Remark \ref{rmk:tot-cst}, we conclude that 
$$(\can_{\K^\T})_{X, Y}:\map_{\cat M}(X,Y)\to \map_{\op D(\T)}\big (\can_{\K^\T}(X),\can_{\K^\T}(Y)\big)$$
is a weak equivalence as desired.
\end{proof}

\begin{rmk}\label{rmk:effective}  Let $\cat M$ be a simplicial model category. Let $\T=(T,\mu,\eta)$  be a monad on $\cat M$ such that $T$ is a simplicial functor, and $\mu$ and $\eta$ are simplicial natural transformations.  Let $\mathbb{DT}$ denote the monad on $\op D(\T)$ with underlying endofunctor $F_{\K^\T}U_{\K^\T}$ (cf. Remarks \ref{rmk:T-adjunct}  and \ref {rmk:T-adjunct}).

According to Remark \ref{rmk:monad-diagram}
$$\can _{\K^\T}T=\can _{\K^\T}(U^{\T}F^\T)=F_{\K^\T}F^\T=(F_{\K^\T}U_{\K^\T})\can _{\K^\T}.$$
In fact, careful checking of the definitions of the cofaces and codegeneracies shows that for all objects $Y$ in $\cat M$,
$$\can _{\K^\T}^{\bold \Delta}\Om^\bullet_{\T}(Y)=\Om^\bullet _{\mathbb{DT}}\big(\can _{\K^\T}(Y)\big).$$
Thus, if $Y$ is strictly $\T$-complete, i.e.,  there is an external cosimplicial SDR
$$\xymatrix{cc^\bullet Y \ar @<.7ex>[r]^-{\eta^\bullet}&\Om ^{\bullet}_{\T}(Y)\circlearrowright h\ar @<.7ex>[l]^-{\rho^\bullet},}$$
then $\can _{\K^\T}(Y)$ is strictly $\mathbb{DT}$-complete, as there is an induced 
external cosimplicial SDR
$$\xymatrix{cc^\bullet \can _{\K^\T}(Y) \ar @<.7ex>[r]^-{\can _{\K^\T}^{\bold \Delta}\eta^\bullet}&\Om ^{\bullet}_{\mathbb{DT}}\can _{\K^\T}(Y)\circlearrowright h\ar @<.7ex>[l]^-{\can _{\K^\T}^{\bold \Delta}\rho^\bullet}.}$$ The image of the restriction of $\can _{\K^\T}$ to $\cat M^\wedge_{\T}$ therefore consists of strictly $\mathbb{DT}$-complete descent data.
 
 Observe that since $U_{\K^\T}\can_{\K^\T}=F^\T$,
 $$\xymatrix{U_{\K^\T}^{\bold \Delta}\can_{\K^\T}^{\bold \Delta}\eta^\bullet=(F^\T)^{\bold \Delta}\eta^\bullet: cc^\bullet F^\T(Y) \ar [r]&(F^\T)^{\bold \Delta}\Om ^{\bullet}_{\mathbb {T}}(Y),}$$
 which can easily seen to be an external homotopy equivalence in $(M^\T)^{\bold \Delta}$, since 
 $$\mu_{Y}:F^\T U^\T  F^\T Y \to F^\T Y$$
 is a morphism of $\T$-algebras. (Note that $\mu_{Y}$ does \textbf{not} underlie a morphism of descent data from $\can_{\K^\T}(TY)$ to $\can _{\K^\T}(Y)$, in general.) Thus, if $U_{\K^\T}$ is cosimplicially faithful, then every descent datum in the image of $\can_{\K^\T}$ is strictly $\mathbb{DT}$-complete, whence $\widehat \T$ satisfies effective descent, if $\Om^\bullet_{\T}$ preserves fibrant objects.

It is natural to wonder more generally under what conditions any strictly $\mathbb{DT}$-complete descent datum is weakly equivalent to the canonical desent datum of a strictly $\T$-complete object of $\cat M$.  We intend to investigate this question in future work, i.e., to determine when the homotopic descent of Theorem \ref{thm:equiv-monad} is effective.
\end{rmk}

To conclude this section, we prove a homotopic version of the classical characterization of descent (Theorem \ref{thm:coBeck}). This ``homotopic Beck criterion'' has a somewhat different flavor from our previous criterion for homotopic descent, as it is formulated in terms of just the bottom two stages of the $\T$-cobar construction.

\begin{thm}\label{thm:htpic-coBeck}  Let $\cat M$ be a simplicial model category. Let $\T=(T,\mu,\eta)$   be a monad on $\cat M$ such that $T$ is a simplicial functor preserving bifibrant objects, and $\mu$ and $\eta$ are simplicial natural transformations. Suppose that 
$U^\T:\cat M^\T\to \cat M$ right-induces a model category structure on $\cat M^\T$, that $U_{\K^\T}:\op D(\T)\to \cat M^\T$ then left-induces a model category structure on $\op D(\T)$,  and that  $\op D(\T)$ is endowed with its usual simplicial enrichment (Lemma \ref{lem:simplenr}). 

If $\can_{\K^\T}$ preserves fibrant objects,  and the natural morphism $$\hat\eta_{Y}:Y\to \operatorname{Prim}_{\K^\T}\circ \can _{\K^\T}(Y)$$ is a weak equivalence for all fibrant objects $Y$, then the restriction and corestriction of $\T$ to the full subcategory of $\cat M$ determined by the bifibrant objects satisfies homotopic descent.
\end{thm}

\begin{proof} If $X$ is cofibrant and $Y$ is fibrant, then the hypotheses above imply that 
$$(\hat\eta_{Y})_{*}:\map _{\cat M}(X,Y)\to \map _{\cat M}\big(X,  \operatorname{Prim}_{\K^\T}\circ \can _{\K^\T}(Y)\big)$$
is a weak equivalence of fibrant simplicial sets.  On the other hand, by Corollary \ref{cor:adj-can} 
$$\map _{\cat M}\big(X,  \operatorname{Prim}_{\K^\T}\circ \can _{\K^\T}(Y)\big)\cong\map _{\op D(\T)}\big(\can _{\K^\T}(X), \can _{\K^\T}(Y)\big),$$
and we can conclude.
\end{proof}

\subsection{Homotopic codescent}

We now dualize the definitions and results of the previous section.  Since the proofs dualize as well, we omit them.

We suggest this time that the reader compare the definition below to Definition \ref{defn:codesccat} and Proposition \ref{prop:adj-iso}.

\begin{defn}\label{defn:htpiccodesc} Let $\cat S$ be a simplicially enriched category.  Let $\K=(K, \Delta, \ve)$   be a comonad on $\cat S$ such that $K$ is a simplicial functor, and $\Delta$ and $\ve$ are simplicial natural transformations.

If each component
$$(\can^{\T_{\K}})_{X,Y}:\map_{\cat S}(X,Y) \to \map_{\op D^{co}(\K)}\big(\can^{\T_{\K}}(X), \can^{\T_{\K}}(Y)\big)$$
of the simplicial functor $\can^{\T_{\K}}:\cat S\to \op D^{co}(\K)$ is a weak equivalence of simplicial sets, then $\K$ satisfies \emph{homotopic codescent}.  If, in addition, every codescent datum is isomorphic in the path-component category $\pi_{0}\op D^{co}(\K)$ to an object of the form $\can^{\T_{\K}}(X)$, then $\K$ satisfies \emph{effective homotopic codescent}. 
\end{defn}

The criterion for homotopic codescent for comonads on simplicial model categories is formulated in terms of the following subcategory of the underlying model category.

\begin{notn} If $\cat M$ is a simplicial model category, and $\K=(K, \Delta, \ve)$   is a comonad on $\cat M$, we let $\cat M^\vee_\K$ denote the full simplicial subcategory of $\cat M$ determined by the bifibrant, strictly $\K$-cocomplete objects.

The observation in Remark \ref{rmk:ex-Kcocomp} implies that if $K$ preserves bifibrant objects, then $K$ fixes $\cat M^\vee_{\K}$, i.e., $\K$ restricts and corestricts to a comonad on $\cat M^\vee_{\K}$, which we denote $\check{\K}$.
\end{notn}

\begin{thm}\label{thm:equiv-comonad}  Let $\cat M$ be a simplicial model category. Let $\K=(K, \Delta, \ve)$   be a comonad on $\cat M$ such that $K$ is a simplicial functor, and $\Delta$ and $\ve$ are simplicial natural transformations. Suppose that 
$U_{\K}:\cat M_{\K}\to \cat M$ left-induces a model category structure on $\cat M_{\K}$, that $U^{\T_{\K}}:\op D^{co}(\K)\to \cat M_{\K}$ then right-induces a model category structure on $\op D(\T)$, and that the usual simplicial enrichment of $\op D^{co}(\K)$ (Lemma \ref{lem:simplenr}) satisfies the additional axiom of Definition \ref{defn:simplmodel}. 

If $\Bar^\K_{\bullet}$ preserves cofibrant objects, then $\check\K$ satisfies homotopic codescent.
\end{thm}

This theorem implies that a form of ``faithfully flat codescent'' holds. Recall the notion of simplicial faithfulness from Definition \ref{defn:simpfaith}.

\begin{cor}\label{cor:ffcod} Under the hypotheses of Theorem \ref{thm:equiv-comonad}, if $\Bar^\K_{\bullet}$ preserves cofibrant objects and  $K$ is simplicially faithful, then the restriction of $\K$ to the full subcategory of bifibrant objects satisfies homotopic codescent. 
\end{cor}

Theorem \ref{thm:equiv-comonad} is an immediate consequence of the following proposition, which generalizes Proposition \ref{prop:adj-iso}, in the case where the underlying category is a simplicial model category.

\begin{prop}\label{prop:cocriterion} Under the hypotheses of Theorem \ref{thm:equiv-comonad}, if $X$ is cofibrant and strictly $\K$-cocomplete, and $\Bar^\K_{\bullet }X$ is Reedy cofibrant, then 
$$(\can^{\T_{\K}})_{X,Y}:\map_{\cat M}(X,Y) \to \map_{\op D^{co}(\K)}\big(\can^{\T_{\K}}(X), \can^{\T_{\K}}(Y)\big)$$
is a weak equivalence of simplicial sets, for all fibrant objects $Y$.
\end{prop}

\begin{rmk}\label{rmk:coeffective}  Let $\cat M$ be a simplicial model category. Let $\K=(K, \Delta, \ve)$   be a comonad on $\cat S$ such that $K$ is a simplicial functor, and $\Delta$ and $\ve$ are simplicial natural transformations.  Let $\mathbb{DK}$ denote the comonad on $\op D^{co}(\K)$ with underlying endofunctor $F^{\T_{\K}}U^{\T_{\K}}$.

An analysis dual to that in Remark \ref{rmk:effective} shows that   the image of the restriction of $\can^{\T_{\K}}$ to $\cat M^\vee_{\K}$ consists of strictly $\mathbb{DK}$-cocomplete codescent data.  Moreover, if $U^{\T_{\K}}$ is simplicially faithful, then $\check\K$ satisfies effective homotopic codescent, if  $\Bar^\K_{\bullet}$ preserves cofibrant objects.  We intend to determine in future work under what conditions any strictly $\mathbb{DK}$-cocomplete codescent datum is weakly equivalent to the canonical codesent datum of a strictly $\K$-cocomplete object of $\cat M$, i.e., to determine when the homotopic codescent of Theorem \ref{thm:equiv-comonad} is effective.  
\end{rmk}

To conclude this section, we state a homotopic version of the classical characterization of codescent (Theorem \ref{thm:Beck}). This ``homotopic Beck criterion'' has a somewhat different flavor from our previous criterion for homotopic codescent, as it is formulated in terms of just the bottom two stages of the $\K$-bar construction. The proof is strictly dual to that of Theorem \ref{thm:htpic-coBeck}, so we omit it.

\begin{thm}\label{thm:htpic-Beck}  Let $\cat M$ be a simplicial model category. Let $\K=(K, \Delta, \ve)$   be a comonad on $\cat M$ such that $K$ is a simplicial functor, and $\Delta$ and $\ve$ are simplicial natural transformations. Suppose that 
$U_{\K}:\cat M_{\K}\to \cat M$ left-induces a model category structure on $\cat M_{\K}$, that $U^{\T_{\K}}:\op D^{co}(\K)\to \cat M_{\K}$ then right-induces a model category structure on $\op D(\T)$, and that  $\op D^{co}(\K)$ is endowed with its usual simplicial enrichement (Lemma \ref{lem:simplenr}). 

If $\can^{\T_{\K}}$ preserves cofibrant objects,  and the natural morphism $$\hat\ve_{X}:Q^{\T_{\K}}\circ \can^{\T_{\K}} (X)\to X$$ is a weak equivalence for all cofibrant objects $X$, then the restriction of $\K$ to the full subcategory of $\cat M$ determined by the bifibrant objects satisfies homotopic codescent.
\end{thm}

\subsection{Descent and codescent spectral sequences}

\subsubsection{The descent spectral sequence}\label{sec:descSS}

Let $\T=(T,\mu, \eta)$ be a monad on a simplicial model category $\cat M$ such that Convention \ref{conv:descss} holds. We require in addition that $\mu$ and $\eta$ be simplicial natural transformations.

\begin{defn}\label{defn:descass} Let $X, Y\in \cat M$, and let $f\in \cat M(X,Y)$.  Let $$\xymatrix@1{p:\widetilde X\wefib &\; X}$$ be a cofibrant replacement of $X$, and let 
$$\xymatrix@1{j^\bullet: \Om^\bullet_{\T}Y\;\cof^-\sim&\;\widehat\Om^\bullet_{\T}Y}$$
 denote a fibrant replacement of $\Om^\bullet_{\T}Y$ in $\cat M^{\mathbf \Delta}$.

The extended homotopy spectral sequence \cite[X.6]{bousfield-kan} of the fibrant cosimplicial simplicial set
$$ \map_{\cat M} (\widetilde X,\widehat\Om^\bullet_{\T}Y),$$
with basepoint 
$$j^n\circ \eta_{T^nY}\circ \cdots \circ \eta_{TY}\circ \eta_{Y}\circ f\circ p:\widetilde X\to \widehat \Om_{\T} ^n Y$$
in level $n$, denoted $E^\T_{f}$,
is the \emph{$\T$-descent spectral sequence} for the pair $(X,Y)$ with respect to $f$.
\end{defn}

\begin{rmk}  The $\T$-descent spectral sequence is clearly independent, up to isomorphism, of the choices of cofibrant and fibrant replacements made in the definition above, which justifies our use of the definite article in the definition.
\end{rmk}

\begin{rmk}  If $\mathbb R$ is the monad on $\cat {sSet}$ of \cite[I.2]{bousfield-kan}, then the $\mathbb R$-descent spectral sequence is exactly the unstable Adams spectral sequence studied there.
\end{rmk}
 
\begin{rmk}  The analysis in sections X.6 and X.7 of \cite {bousfield-kan} shows that
for all $t\geq s\geq 0$,
$$(E^\T_{f})_{2}^{s,t}=\pi^s\pi_{t}\map_{\cat M} (\widetilde X,\widehat\Om^\bullet_{\T}Y)_{f},$$
where, for $t\geq 2$, the cosimplicial cohomotopy $\pi^*\big(\pi_{t}\map (\widetilde X,\widehat\Om^\bullet_{\T}Y)\big)_{f}$ is the usual cohomology of the normalized cochain complex of the cosimplicial abelian group $\pi_{t}\map_{\cat M} (\widetilde X,\widehat\Om^\bullet_{\T}Y)_{f}$.  

Moreover, under conditions described in section IX.5 of \cite {bousfield-kan}, the spectral sequence $E^\T_{f}$ abuts to $\pi_{*}\tot \map_{\cat M}(\widetilde X,\widehat\Om^\bullet_{\T}Y)$.  The convergence issue is subtle, especially since the spectral sequence is fringed, and we do not claim to say anything new about convergence here.
\end{rmk} 

We can now easily identify the term to which a $\T$-descent spectral sequence abuts, at least if $\Om _{\T}^\bullet$ preserves fibrant objects.

\begin{lem}\label{lem:einfty} Suppose that $\Om _{\T}^\bullet$ preserves fibrant objects. If $X$ is cofibrant and $Y$ is fibrant, then the $\T$-descent spectral sequence for $(X,Y)$ abuts to 
$$\pi_{*} \map_{\cat M}(X,Y _{\T}^\wedge )_{f},$$
for any choice of basepoint $f:X\to Y$, if it indeed converges.
\end{lem}

\begin{proof}   Since $\Om^\bullet _{\T}Y$ is Reedy fibrant by hypothesis, 
$$\tot \map_{\cat M}(X,\Om^\bullet_{\T}Y)\cong\map_{\cat M} (X, \tot \Om^\bullet_{\T}Y)=\map_{\cat M}(X,Y _{\T}^\wedge ).$$
\end{proof}

Under reasonable conditions, the $E_{2}$-term of the $\T$-descent spectral sequence admits an interesting interpretation.  We refer the reader to Definition \ref{defn:ext} for the meaning of $\ext$ in simplicial model categories, to Remark \ref{rmk:effective} for the definition of the monad $\mathbb{DT}$ and to Definition \ref{defn:strong-T} for the definition of strong completeness.

\begin{thm}\label{thm:e2-desc}  Suppose that 
$U^\T:\cat M^\T\to \cat M$ right-induces a model category structure on $\cat M^\T$, that $U_{\K^\T}:\op D(\T)\to \cat M^\T$ then left-induces a model category structure on $\op D(\T)$, and that the usual simplicial enrichment of $\op D(\T)$ (Lemma \ref{lem:simplenr}) extends to a simplicial model category structure. 

Let $X$ and $Y$ be a cofibrant object  and a fibrant  object in $\cat M$, respectively, and let $f\in \cat M(X,Y)$.  If $\Om ^\bullet_{\T}Y$ is Reedy fibrant and $\can_{\K^\T}(Y)$ is strongly $\mathbb{DT}$-complete, then 
$$(E^\T_{f})_{2}^{s,t}\cong \ext^{s,t}_{\op D(\T)}\big( \can_{\K^\T} (X), \can_{\K^\T} (Y)\big)_{\can_{\K^\T}f}$$
for any choice of basepoint $f$.
\end{thm}

\begin{rmk}  We can interpret Theorem \ref{thm:e2-desc} to mean that the $\T$-descent spectral sequence interpolates from $$\pi_{*}\map_{\op D(\T)}\big ( \can_{\K^\T} (X), \can_{\K^\T} (Y)\big)_{\can_{\K^\T}f}$$ to $$\pi_{*}\map_{\cat M}(X,Y^\wedge_\T)_{f},$$ at least when $\can_{\K^\T}(Y)$ is strongly $\mathbb{DT}$-complete.  

Since we are assuming that $U_{\K^\T}$ left-induces the model structure on $\op D(\T)$, it follows that $U_{\K^\T}$ is homotopically faithful, whence all fibrant objects in $\op D(\T)$ are $\mathbb{DT}$-complete.  This may be an indication that the requirement of strong $\mathbb {DT}$-completeness on $\can_{\K^\T}(Y)$ is not a great constraint.

When $Y$ is strictly $\T$-complete, $\can_{\K^\T}(Y)$ is strictly $\mathbb{DT}$-complete, but also $\map_{\cat M}(X,Y)$ and $\map_{\op D(\T)}\big ( \can_{\K^\T} (X), \can_{\K^\T} (Y)\big)$ are weakly equivalent by Theorem \ref{thm:equiv-monad}.  The spectral sequence almost certainly collapses in this case.  
\end{rmk}

\begin{proof}[Proof of Theorem \ref{thm:e2-desc}]   By Proposition \ref{prop:adj-iso}, the simplicial functor $\can _{\K^\T}$ induces an isomorphism
$$\map _{\cat M}(X, \Om ^n_{\T}Y)\cong \map_{\op D(\T)}\big(\can _{\K^\T}(X), \can _{\K^\T}(\Om^n_{\T}Y)\big)$$
for all $n\geq 0$. 
There is therefore an isomorphism of cosimplicial simplicial sets
$$\map_{\cat M} (X, \Om ^\bullet_{\T}Y)\cong \map _{\op D(\T)}\big( \can_{\K^\T}(X), \can_{\K^\T}^{\bold\Delta}(\Om ^\bullet_{\T}Y)\big),$$
where the functor $\can_{\K^\T}$ is applied levelwise in the second component on the righthand side.

Observe that $\can_{\K^\T}^{\bold \Delta}(\Om ^\bullet_{\T}Y)$ is a cohomological fibrant resolution of $\can_{\K^\T} (Y)$ in $\op D(\T)^{\bold \Delta}$, since $\can_{\K^\T}(Y)$ is strongly $\mathbb{DT}$-complete, and Lemma \ref{lem:Reedyfib} implies that $\can_{\K^\T}^{\bold\Delta}\Om ^{\bullet}_{\T}(Y)$ is Reedy fibrant.  By definition of Ext in $\op D(\T)$, we can therefore conclude.
\end{proof}

\subsubsection{The codescent spectral sequence}

Let $\K=(K, \Delta, \ve)$ be a comonad on a simplicial model category $\cat M$ such that Convention \ref{conv:codescss} holds and such that $\Delta$ and $\ve$ are simplicial natural transformations.

\begin{defn}\label{defn:codescass} Let $X, Y\in \cat M$, and let $f\in \cat M(X,Y)$.  Let $$\xymatrix@1{j:Y\;\wecof &\; \widehat Y}$$ be a fibrant replacement of $Y$, and let 
$$\xymatrix@1{p_\bullet: \widetilde\Bar^\K_{\bullet }X \;\wefib&\;\Bar^\K_{\bullet }X}$$
 denote a cofibrant replacement of $\Bar^\K_{\bullet }X$ in $\cat M^{\mathbf \Delta^{op}}$.

The extended homotopy spectral sequence \cite[X.6]{bousfield-kan} of the fibrant cosimplicial simplicial set
$$ \map_{\cat M} (\widetilde\Bar^\K_{\bullet }X,\widehat Y),$$
with basepoint 
$$j\circ f\circ \ve _{X}\circ \cdots \circ \ve_{K^{n}X}\circ p_{n}:\widetilde\Bar^\K_{n }X\to \widehat Y$$
in level $n$, denoted $E^\K_{f}$, is the \emph{$\K$-codescent spectral sequence} for the pair $(X,Y)$ with respect to $f$.
\end{defn}

\begin{rmk}  The $\K$-codescent spectral sequence is clearly independent, up to isomorphism, of the choices of cofibrant and fibrant replacements made in the definition above.
\end{rmk}

\begin{rmk}  The analysis in sections X.6 and X.7 of \cite {bousfield-kan} shows that
for all $t\geq s\geq 0$,
$$(E^\K_{f})_{2}^{s,t}=\pi^s\pi_{t}\map_{\cat M} (\widetilde\Bar^\K_{\bullet }X,\widehat Y),$$
where, for $t\geq 2$, the cosimplicial cohomotopy $\pi^*\big(\pi_{t}\map (\widetilde\Bar^\K_{\bullet }X,\widehat Y)\big)$ is the usual cohomology of the normalized cochain complex of the cosimplicial abelian group $\pi_{t}\map_{\cat M} (\widetilde\Bar^\K_{\bullet }X,\widehat Y)$.  

Moreover, under conditions described in section IX.5 of \cite {bousfield-kan}, the spectral sequence $E^\K_{f}$ abuts to $\pi_{*}\tot \map_{\cat M}(\widetilde\Bar^\K_{\bullet }X,\widehat Y)$.  Again, we do not claim to say anything new about convergence here.
\end{rmk}

In the next two results, we identify both the term to which a $\K$-codescent spectral sequence abuts and its $E_{2}$-term, under reasonable cofibrancy hypotheses.  The proofs are dual to those of Lemma \ref{lem:einfty} and Theorem \ref{thm:e2-desc}.

\begin{lem}\label{lem:co-einfty} Suppose that $\Bar^\K_{\bullet}$ preserves cofibrant objects. If $X$ is cofibrant and $Y$ is fibrant, then the $\K$-descent spectral sequence for $(X,Y)$ abuts to 
$$\pi_{*} \map_{\cat M}(X^\vee_{\K},Y )_{f},$$
for any choice of basepoint $f:X\to Y$, if it indeed converges.
\end{lem}

We refer the reader to Definition \ref{defn:ext} for the meaning of $\ext$ in simplicial model categories, to Remark \ref{rmk:coeffective} for the definition of the comonad $\mathbb{DK}$ and to Definition \ref{defn:strong-T} for the definition of strong completeness.

\begin{thm}\label{thm:e2-codesc} Suppose that 
$U_{\K}:\cat M_{\K}\to \cat M$ left-induces a model category structure on $\cat M_{\K}$ and that $U^{\T_{\K}}:\op D^{co}(\K)\to \cat M_{\K}$ then right-induces a model category structure on $\op D(\T)$.
 and that the usual simplicial enrichment of $\op D^{co}(\K)$ (Lemma \ref{lem:simplenr}) extends to a simplicial model structure. 

Let $X$ and $Y$ be a cofibrant object  and a fibrant  object in $\cat M$, respectively.  If $\Bar^\K_{\bullet}X$ is Reedy cofibrant and $ \can^{\T_{\K}} (X)$ is strongly $\mathbb{DK}$-cocomplete, then 
$$(E^\K_{f})_{2}^{s,t}\cong \ext^{s,t}_{\op D^{co}(\K)}\big( \can^{\T_{\K}} (X), \can^{\T_{\K}} (Y)\big)_{\can^{\T_{\K}}f}$$
for any choice of basepoint $f$.
\end{thm} 

\begin{rmk}  We can interpret Theorem \ref{thm:e2-codesc} to mean that the $\K$-codescent spectral sequence interpolates from 
$$\pi_{*}\map_{\op D^{co}(\K)}\big ( \can^{\T_{\K}}(X), \can^{\T_{\K}} (Y)\big)_{\can^{\T_{\K}}f}$$
 to 
 $$\pi_{*}\map_{\cat M}(X^\vee_{\K},Y)_{f},$$ when $ \can^{\T_{\K}} (X)$ is strongly $\mathbb{DK}$-cocomplete. 
 
 Since we are assuming that $U^{\T_{\K}}$ right-induces the model structure on $\op D^{co}(\K)$, it follows that $U^{\T_{\K}}$ is homotopically faithful, whence all cofibrant objects in $\op D^{co}(\K)$ are $\mathbb{DK}$-cocomplete.  This may be an indication that the requirement of strong $\mathbb {DK}$-cocompleteness on $\can_{\K^\T}(X)$ is not a great constraint. 

When $X$ is strictly $\K$-cocomplete, $\can^{\T_{\K}}(X)$ is strictly $\mathbb{DK}$-cocomplete, but also $\map_{\cat M}(X,Y)$ and $\map_{\op D^{co}(\K)}\big ( \can^{\T_{\K}}(X), \can^{\T_{\K}} (Y)\big)$ are weakly equivalent by Theorem \ref{thm:equiv-comonad}.  The spectral sequence almost certainly collapses in this case.
\end{rmk}

\section{Homotopic Grothendieck descent and its dual}\label{sec:htpicGroth}

We return in this section to the monads and comonads introduced in sections \ref{sec:grothendieck} and \ref{sec:dualGroth} and apply to them the theories of derived (co)completion and of homotopic (co)descent.  We also examine the associated (co)descent spectral sequences and see that certain of them are already very familiar.

\subsection{The Grothendieck framework}

Throughout this section, let 
$$\vp:B\to A$$ 
be a monoid morphism in a simplicial model category $\cat M$ that is also endowed with a monoidal structure $(\wedge, I)$.  We study here the theories of derived completion and of homotopic descent for the monad $\T_{\vp}$ associated to the adjunction
$$-\underset B\wedge A:\cat {Mod}_{B}\adjunct{}{} \cat {Mod}_{A}:\vp^*.$$
We first need to determine conditions under which $\cat {Mod}_{B}$ and $\cat {Mod}_{A}$ are equipped with a simplicial model category structure.

\begin{defn}\label{defn:monsimpl}  A simplicial model category $\cat M$, which is also endowed with a monoidal structure $(\wedge, I)$, is \emph{monoidally simplicial} if 
\begin{enumerate}
\item the simplicial-tensoring functor
$-\otimes -:\cat M \times \cat {sSet} \to \cat S$ 
is monoidal and op-monoidal, i.e., there are appropriately ``associative'' natural transformations
$$\tau:(X\otimes L)\wedge (X'\otimes L')\to (X\wedge X')\otimes (L\times L')$$
and
$$\upsilon: (X\wedge X')\otimes (L\times L')\to (X\otimes L)\wedge (X'\otimes L'),$$
and 
\item the monoidal product in $\cat M$ is itself a simplicial functor, i.e.,
the mapping space functor $\map_{\cat M}(-,-):\cat M^{op}\times \cat M\to \cat {sSet}$ is monoidal, and the associated natural transformation
$$\xi:\map_{\cat M} (X,Y)\times \map_{\cat M} (X',Y') \to \map_{\cat M} (X\wedge X', Y\wedge Y')$$
is appropriately compatible with the composition
\end{enumerate}
\end{defn}

\begin{rmk}  If $\cat M$ is  monoidally simplicial, then for all objects $X$ and $X'$ in $\cat M$ and all simplicial sets $L$ and $L'$, there is morphism in $\cat M$
$$(X^L \wedge (X')^{L'})\otimes (L\times L')\xrightarrow \upsilon (X^L\otimes L)\wedge ((X')^{L'}\otimes L') \xrightarrow {ev_{X}\wedge ev_{Y}}X\wedge X',$$
where $ev$ is the counit of the tensoring/cotensoring-adjunction, i.e., the ``evaluation'' map.  Taking the transpose of this composite, we obtain a natural map
$$\zeta: X^{L}\otimes (X')^{L'}\to (X\wedge X')^{L\times L'}.$$
\end{rmk}

\begin{ex}  Let $\cat M$ be a simplicial model category that is also endowed with a monoidal structure. As S. Schwede observed to the author, it is not difficult to show that if the simplicial tensoring and cotensoring on $\cat M$ is induced by a strong symmetric monoidal functor from simplicial sets  to  $\cat M$, then $\cat M$ is monoidally simplicial.  For example, the  geometric realization functor gives rise to the simplicial structure on the category of topological spaces, while the simplicial suspension spectrum functor induces the simplicial structure on the category of symmetric spectra.
\end{ex} 

\begin{rmk} If $A$ is any monoid in $\cat M$, then there is a monad $\T_{A}$ on $\cat M$ with underlying endofunctor $-\wedge A$, which is a simplicial functor if $\cat M$ is monoidally simplicial.  Applying Lemmas \ref{lem:simplenr} and \ref{lem:tens-cotens} , we see that $\cat M^{\T_{A}}$, which is isomorphic to $\cat {Mod}_{A}$, is simplicially enriched and cotensored over $\cat {sSet}$.  As stated in the next lemma,  we can actually complete this structure on $\cat {Mod}_{A}$ to that of a simplicial model category.  
\end{rmk}

\begin{convention}\label{conv:groth}  Henceforth, we assume that $\cat M$ is a cofibrantly generated, mon\-oid\-ally simplicial model category such that all objects are small relative to the entire category and the monoid axiom  is satisfied.  We can therefore apply Theorem 4.1 (1) in \cite{schwede-shipley} to obtain a cofibrantly generated model category structure on $\cat {Mod}_{A}$ for all monoids $A$.
\end{convention}
 
 Because we are assuming that $\cat M$ is monoidally simplicial, its module categories actually admit more structure than just that of a model category. 
 
\begin{lem}  Let $\cat M$ satisfy the criteria of Convention \ref{conv:groth}.  If $A$ is any monoid in $\cat M$, then $\cat {Mod}_{A}$ inherits a simplicial model category structure from $\cat M$, where the cotensoring, denoted $(-)^{(-)}_{A}$, and the simplicial enrichment, denoted $\map_{A}$, are defined as in Lemmas \ref{lem:tens-cotens} and \ref{lem:simplenr}, respectively, and
 $$-\overset A\otimes -:\cat {Mod}_{A}\times \cat{sSet} \to \cat {Mod}_{A}$$ 
 is defined for all $A$-modules $(M,\rho)$ and simplicial sets $K$ by
$$M\overset A\otimes K=(M\otimes K, \rho'),$$
where $\rho'$ is the composite
{\smaller $$(M\otimes K)\wedge A\cong (M\otimes K)\wedge (A\otimes \Delta [0])\xrightarrow\tau (M\wedge A)\otimes (K\times \Delta [0])\cong (M\wedge A)\otimes K \xrightarrow {\rho\otimes K} M\otimes K.$$} 
\end{lem}

The proof is straightforward and left to the reader.

Recall from section \ref{sec:grothendieck} that if $A$ is augmented, then $(\cat {Mod}_{A})^{\T_{\vp}}\simeq \cat {Mod}_{B}$ and that $\op D(\T_{\vp})\cong \cat M_{A}^{W_{\vp}}$, the category of comodules over the descent co-ring $W_{\vp}$ in the category of right $A$-modules.  Finally, 
$$\can _{\vp}:=\can _{\K_{\vp}}:\cat{Mod}_{B}\to \cat M_{A}^{W_{\vp}}: M\mapsto (M\underset B\wedge A, M\underset B\wedge \vp\underset B\wedge A).$$
We write $\map_{W_{\vp},A}(-,-)$ for the induced mapping space functor on $(\cat M_{A}^{W_{\vp}})^{op}\times \cat M_{A}^{W_{\vp}}$.

For the theory of derived completion from section \ref{sec:completion} to be applicable, the functor $\vp^*(-\underset B\wedge A): \cat{Mod}_{B}\to \cat {Mod}_{B}$ must be simplicial.  It is clear that this is always true, since $\vp^*(-\underset B\wedge A)$ is a composite of simplicial functors, by definition of the mapping spaces in the module categories.  Note that it is also evident that the multiplication $\mu_{\vp}$ and the unit $\eta_{\vp}$ of the monad $\T_{\vp}$ are simplicial natural transformations.

In this Grothendieck framework, Definition \ref{defn:modelTcomp} and Notation \ref{notn:Tcomp} translate into the following definition.

\begin{defn}  Let $M$ be a $B$-module.  Let $\widehat M$ be a fibrant replacement of $M$ such that $\Om ^\bullet_{\T_{\vp}}\widehat M$ is Reedy fibrant.  The \emph{derived completion of $M$ along $\vp$} is
$$M^\wedge_{\vp}:=M^\wedge_{\T_{\vp}}=\tot \Om ^\bullet_{\T_{\vp}}\widehat M.$$
\end{defn}

\begin{rmk}\label{rmk:amitsur} When $\cat M$ is a monoidal model category of spectra, the definition above agrees with Carlsson's definition of derived completion of module spectra \cite{carlsson}, when $A$ is cofibrant as a $B$-algebra.  Moreover $\Om^\bullet_{\T_{\vp}}M$ is exactly the Amitsur complex of $M$ with respect to $\vp$ (cf., e.g,. Definition 8.2.1 in \cite{rognes}). 
\end{rmk}

\begin{rmk} Since $\cat {Mod}_{A}$ is a simplicial model category when $\cat M$ is monoidally simplicial, Lemma \ref{lem:char-Tequiv} implies that a morphism $f:M\to N$ of cofibrant $B$-modules is a $\T_{\vp}$-equivalence if and only if $f\wedge_{B}A:M\wedge_{B}A\to N\wedge_{B}A$ is a weak equivalence. 
\end{rmk}

\begin{rmk}  Recall from Remark \ref{rmk:inftyTalg} that a strictly $\T_{\vp}$-complete $B$-module $M$ can be seen as an $\infty$-$\T_{\vp}$-algebra.  In other words, $M$ is a sort of $\infty$-$A$-module, with respect to the monoidal product $-\underset B\wedge -$.
\end{rmk}

Specializing Definition \ref{defn:htpicdesc} to the Grothendieck framework, we obtain the following notion of homotopic descent.  Note that this definition makes sense even if the module categories admit only a simplicial enrichment, rather than a full simplicial model category structure.

\begin{defn}\label{defn:htpicGrothdesc} The monad $\T_{\vp}$ satisfies \emph{homotopic descent} if each component
$$(\can _{\vp})_{M,M'}:\map_{B}(M,M')\to \map_{W_{\vp},A}\big(\can_{\vp}(M),\can_{\vp}(M')\big)$$
of the simplicial functor $\can_{\vp}$ is a weak equivalence of simplicial sets.
\end{defn}

\begin{rmk}\label{rmk:lhT}  Since the descent category $\op D(\T _{\vp})$ is isomorphic to a category of comodules, $\cat M_{A}^{W_{\vp}}$,  this definition implies that if $\T_{\vp}$ satisfies homotopic descent, then $\cat {Mod}_{B}$ is locally homotopy Tannakian.
\end{rmk}

The criterion for homotopic descent in the Grothendieck framework is a special case of Theorem \ref{thm:equiv-monad}.

\begin{thm}\label{thm:Grothdesc-crit} Let $(\cat {Mod}_{B})^\wedge_{\T_{\vp}}$ denote the full simplicial subcategory of $\cat {Mod}_{B}$ determined by the bifibrant, strictly $\T_{\vp}$-complete objects. Suppose that the forgetful functor $U_{\K^{\T_{\vp}}}:\cat M_{A}^{W_{\vp}}\to \cat {Mod}_{A}$ left-induces a model category structure on $M_{A}^{W_{\vp}}$ that is compatible with its usual simplicial enrichment.

If $\vp^*(-\underset B\wedge A):\cat{Mod}_{B}\to \cat {Mod}_{B}$ preserves bifibrant objects and $\Om^\bullet_{\T_{\vp}}$ preserves fibrant objects, then $\T_{\vp}$ restricts and corestricts to a monad $\widehat \T_{\vp}$ on $(\cat {Mod}_{B})^\wedge_{\T_{\vp}}$ that satisfies homotopic descent.
\end{thm}

The translation of Corollary \ref{cor:ffd} in this case gives a homotopical version of the usual faithfully flat descent.  Note that $\vp^*(-\underset B\wedge A)$ is cosimplicially faithful if a morphism $f^\bullet :M^\bullet \to N^\bullet$ of cosimplicial $B$-modules is an external homotopy equivalence whenever $f^\bullet\underset B\wedge A$ is an external homotopy equivalence of cosimplicial $A$-modules.

\begin{cor}\label{cor:ffd-groth} Under the hypotheses of Theorem \ref{thm:Grothdesc-crit}, if 
$\Om ^\bullet_{\T_{\vp}}$ preserves fibrant objects, and  $\vp^*(-\underset B\wedge A)$ is cosimplicially faithful, then the restriction of $\T_{\vp}$ to the full subcategory of bifibrant $B$-modules satisfies homotopic descent. 
\end{cor}

The relevant spectral sequence in this case is defined as follows.  The choice of terminology is justified by the examples sketched below.

\begin{defn}  Let $M$ and $N$ be right $B$-modules, and let $f:M\to N$ be a morphism of $B$-modules.  The \emph{$\vp$-Adams spectral sequence for $(M,N)$ with respect to $f$}, denoted $E^\vp_{f}$, is the $\T_{\vp}$-descent spectral sequence for $(M,N)$ with respect to $f$.
\end{defn}

The next result is an immediate specialization of Lemma \ref{lem:einfty} and Theorem \ref{thm:e2-desc}.

\begin{prop}Let $M$ and $N$ be a cofibrant and a fibrant right $B$-module, respectively, and let $f:M\to N$ be a morphism of $B$-modules. 
\begin{enumerate}
\item  If $\Om^\bullet_{\T_{\vp}}$ preserves fibrant objects, then the $\vp$-Adams spectral sequence for $(M,N)$ with respect to $f$ abuts to $\pi_{*}\map_{B}(M, N^\wedge_\vp)_{f}$, if it indeed converges.
\item Suppose that the forgetful functor $U_{\K^{\T_{\vp}}}:\cat M_{A}^{W_{\vp}}\to \cat {Mod}_{A}$ left-induces a model category structure on $M_{A}^{W_{\vp}}$ that is compatible with its usual simplicial enrichment. If $\Om^\bullet_{\T_{\vp}}N$ is Reedy fibrant and $\can_{\vp}(N)$ is strongly $\mathbb{DT}_{\vp}$-complete, then 
$$(E^\vp_{f})_{2}^{s,t}\cong \ext^{s,t}_{W_{\vp}, A}\big( \can_{\vp} (M), \can_{\vp} (N)\big)_{\can_{\vp}f}$$
for any choice of basepoint $f$.
\end{enumerate}
\end{prop}

In a future article we will treat in detail the following special, very important cases of $\vp$-Adams spectral sequences.  Here we sketch only the proposed applications.

\subsubsection{The stable Adams spectral sequence}\label{sec:ass} Let $H\mathbb F_{p}$ denote the Eilenberg-MacLane spectrum of $\mathbb F_{p}$, in some appropriate monoidal model category $\cat M$ of spectra, and let $\eta_{H\mathbb F_{p}}:S\to H\mathbb F_{p}$ denote its unit.  As we sketch below, it is possible to show that the $\eta_{H\mathbb F_{p}}$-Adams spectral sequence is the usual stable Adams spectral sequence.  

Note that in this case
$$W_{\eta_{H\mathbb F_{p}}}=H\mathbb F_{p} \wedge H\mathbb F_{p},$$
which is actually a comonoid in the category of commutative $S$-algebras and not just an $S$-co-ring.  
The stable homotopy groups of $W_{\eta_{H\mathbb F_{p}}}$ form the dual mod $p$ Steenrod algebra $\mathcal A_{p}^\sharp$, while for any spectrum $X$, the stable homotopy groups of $\can_{\eta_{H\mathbb F_{p}}}(X)=X\wedge H\mathbb F_{p}$ are its singular homology groups with coefficients in $\mathbb F_{p}$. 

Applying Proposition \ref{prop:criterion} and assuming that all the necessary model category conditions hold, we see that  if $Y$ is  fibrant and strictly $\T_{\eta_{H\mathbb F_{p}}}$-complete, $\Om^\bullet_{\T_{\eta_{H\mathbb F_{p}}}}(Y)$ is Reedy fibrant, and $X$ is cofibrant, then the simplicial map
$$(\can _{\eta_{H\mathbb F_{p}}})_{X,Y}:\map _{\cat M}(X, Y)\to \map_{W_{\eta_{H\mathbb F_{p}}},H\mathbb F_{p}}\big( X\wedge H\mathbb F_{p}, Y\wedge H\mathbb F_{p}\big)$$
is a weak equivalence, where we have suppressed the coactions from the notation for $\can_{\eta_{H\mathbb F_{p}}}(X)$ and $\can_{\eta_{H\mathbb F_{p}}}(Y)$.
\medskip

\subsubsection{The Adams-Novikov spectral sequence}\label{sec:anss} Let $MU$ denote the complex cobordism spectrum, in some appropriate monoidal model category $\cat M$ of spectra, and let $\eta_{MU}:S\to MU$ denote its unit. We again sketch the argument to show that the $\eta_{MU}$-Adams spectral sequence is the  Adams-Novikov spectral sequence.  

Note that in this case
$$W_{\eta_{MU}}=MU \wedge MU,$$
which again is actually a comonoid in the category of commutative $S$-algebras and not just an $S$-co-ring.  
The stable homotopy groups of $W_{\eta_{MU}}$ form the Hopf algebra of cooperations on $MU_{*}$-homology, while for any spectrum $X$, the stable homotopy groups of $\can_{\eta_{MU}}(X)=X\wedge MU$ are its $MU_{*}$-homology groups. 

As in the previous case, Proposition \ref{prop:criterion} tells us that  if $Y$ is  fibrant and strictly $\T_{\eta_{MU}}$-complete, $\Om^\bullet_{\T_{\eta_{MU}}}(Y)$ is Reedy fibrant,  and $X$ is cofibrant, then the simplicial map 
$$(\can _{\eta_{MU}})_{X,Y}:\map _{\cat M}(X, Y)\to \map_{W_{\eta_{MU}},MU}\big( X\wedge MU, Y\wedge MU\big)$$
is a weak equivalence, where we have suppressed the coactions from the notation for $\can_{\eta_{MU}}(X)$ and $\can_{\eta_{MU}}(Y)$.
\medskip

\subsubsection{Algebraic vs. \'etale K-theory}\label{sec:lq} Let $R=\mathbb Z[\frac 1\ell]$, where $\ell$ is a fixed prime, and let $A$ be a noetherian $R$-algebra. In \cite{dwyer-friedlander} Dwyer and Friedlander constructed a natural map of ring spectra
$$\vp^{DF}_{A}:K_{A}\to \widehat K^{\text{\'et}}_{A},$$
where $K_{A}$ is the algebraic K-theory spectrum of $A$, and $\widehat K^{\text{\'et}}_{A}$ is its \'etale K-theory spectrum, i.e.,
$$\pi_{*}K_{A}\cong K_{*}(A)\quad\text{and}\quad \pi_{*}\widehat K^{\text{\'et}}_{A}\cong \widehat K^{\text{\'et}}_{*}(A).$$
More generally, if $v\geq 1$ and $M(v)$ denotes the mod $\ell^v$-Moore spectrum with bottom cell in stable dimension $0$,  then
$$\pi_{*}(M(v)\wedge K_{A})\cong K_{*}(A;\mathbb Z/\ell^v\mathbb Z)\quad\text{and}\quad \pi_{*} (M(v)\wedge \widehat K^{\text{\'et}}_{A})\cong \widehat K^{\text{\'et}}_{*}(A;\mathbb Z/\ell^v\mathbb Z).$$
The famous Lichtenbaum-Quillen conjecture lists conditions under which the homomorphism
$$(\vp^{DF}_{A})_{*}:K_{*}(A;\mathbb Z/\ell^v\mathbb Z) \to \widehat K^{\text{\'et}}_{*}(A;\mathbb Z/\ell^v\mathbb Z)$$
induced by the Dwyer-Friedlander map should be an isomorphism.  We refer the reader to the extensive literature on this conjecture for cases in which it is known to hold.

Note that the associated canonical descent co-ring is the $\widehat K^{\text{\'et}}_{A}$-bimodule
$$W_{\vp^{DF}_{A}}=\widehat K^{\text{\'et}}_{A}\underset {K_{A}}\wedge \widehat K^{\text{\'et}}_{A},$$
endowed with the usual canonical comultiplication and counit.

Suppose that we can choose a model for the Dwyer-Friedlander map such that the necessary conditions (e.g., Convention \ref{conv:descss}) hold  and so that $\Om^\bullet _{\T_{\vp^{DF}_{A}}}$ preserves fibrant objects. The $\vp^{DF}_{A}$-Adams spectral sequence for a cofibrant $K_{A}$-module $M$ and a fibrant  $K_{A}$-module $N$  such that $N\underset {K_{A}}\wedge \widehat K^{\text{\'et}}_{A}$ is strongly $\mathbb {DT}_{\vp^{DF}_{A}}$-complete then satisfies
$$(E^{\vp^{DF}_{A}}_{f})_{2}^{s,t}\cong \ext^{s,t}_{W_{\vp^{DF}_{A}}, \widehat K^{\text{\'et}}_{A}}\big( M\underset {K_{A}}\wedge \widehat K^{\text{\'et}}_{A}, N\underset {K_{A}}\wedge \widehat K^{\text{\'et}}_{A}\big)_{f\underset {K_{A}}\wedge \widehat K^{\text{\'et}}_{A}}$$
and abuts to 
$$\pi_{*}\map_{K_{A}}(M, N^\wedge_{\vp^{DF}_{A}})_{f}$$
for any choice of basepoint $f:M\to N$, if it indeed converges.  Moreover, Proposition \ref{prop:criterion} implies that the simplicial map 
$$(\can _{\vp^{DF}_{A}})_{M,N}:\map _{K_{A}}(M, N)\to \map_{W_{\vp^{DF}_{A}},\widehat K^{\text{\'et}}_{A}}\big( M\underset {K_{A}}\wedge \widehat K^{\text{\'et}}_{A}, N\underset {K_{A}}\wedge \widehat K^{\text{\'et}}_{A}\big)$$
is a weak equivalence if $N$ is actually strongly $\T_{\vp^{DF}_{A}}$-complete, where we have suppressed the coactions from the notation for $\can_{\vp^{DF}_{A}}(M)$ and $\can_{\vp^{DF}_{A}}(N)$. 

Let $S$ denote the sphere spectrum, i.e., the unit object in $\cat M$. In the case where $M=K_{A}$ and $N=M(v)\wedge K_{A}$, and $f$ is determined by the inclusion of the basepoint into $M(v)$, 
$$\map _{K_{A}}(M, N)\cong\map_{\cat M}(S,M(v)\wedge K_{A}),$$
while
$$\map_{W_{\vp^{DF}_{A}},\widehat K^{\text{\'et}}_{A}}\big( M\underset {K_{A}}\wedge \widehat K^{\text{\'et}}_{A}, N\underset {K_{A}}\wedge \widehat K^{\text{\'et}}_{A}\big)\cong \map_{W_{\vp^{DF}_{A}},\widehat K^{\text{\'et}}_{A}}\big(  \widehat K^{\text{\'et}}_{A}, M(v)\wedge \widehat K^{\text{\'et}}_{A}\big),$$
which is a simplicial subset of 
$$\map_{\widehat K^{\text{\'et}}_{A}}\big(  \widehat K^{\text{\'et}}_{A},M(v)\wedge \widehat K^{\text{\'et}}_{A}\big)\cong  \map_{\cat M}\big(  S, M(v)\wedge\widehat K^{\text{\'et}}_{A}\big).$$

Let $\Psi_{A,v}=\pi_{*}(\can _{\vp^{DF}_{A}})_{K_{A},M(v)\wedge K_{A}}$.  It follows from the analysis above that there is a commutative diagram of group homomorphisms
$$\xymatrix{K_{*}(A;\mathbb Z/\ell^v\mathbb Z)\ar [dr]_{\Psi_{A,v}}\ar[rr]^{(\vp^{DF}_{A})_{*}}&& \widehat K^{\text{\'et}}_{*}(A;\mathbb Z/\ell^v\mathbb Z),\\
	&\pi_{*}\map_{W_{\vp^{DF}_{A}},\widehat K^{\text{\'et}}_{A}}\big(  \widehat K^{\text{\'et}}_{A}, \widehat K^{\text{\'et}}_{A}\big)\ar [ur]_{\iota_{*}}}$$
where $\iota_{*}$ is induced by the inclusion.
There is a thus a close relationship, mediated by $\iota_{*}$, between the Lichtenbaum-Quillen conjecture for $A$ and the question of when $\T_{\vp^{DF}_{A}}$ satisfies homotopic descent.
In particular, if $M(v)\wedge K_{A}$ is bifibrant as a $K_{A}$-module and strictly $\T_{\vp^{DF}_{A}}$-complete, then $\Psi_{A,v}$ is an isomorphism, so that the Lichtenbaum-Quillen conjecture holds for $A$ and $v$ if and only if  $\iota_{*}$ is an isomorphism.
\medskip

\subsubsection{Quillen homology} Let $\cat M$ be a particularly nice, simplicially monoidal model category.  Let $A$ be a monoid in $\cat M$ that is endowed with an augmentation $\ve: A\to I$.  Let 
$$\xymatrix@1{A\cof^\vp &\; Q\wefib &\; I}$$
be a factorization of $\ve$ in the category of $A$-algebras (i.e., monoids in the category of $A$-bimodules) into a cofibration followed by an acyclic fibration.  The descent co-ring associated to the cofibration $\vp$ is then
$$W_{\vp}=Q\underset A\wedge Q.$$
If $M$ is a cofibrant $A$-module and $N$ is a fibrant $A$-module, then the $\vp$-Adams spectral sequence  abuts to 
$$\pi_{*} \map_{A} (M, N^\wedge_{\vp}),$$
if it indeed converges.
If, in addition, $\can_{\vp}(N)$ is strongly $\mathbb {DT}_{\vp}$-complete, then the $E_{2}$-term is isomorphic to 
$$E_{2}^{s,t}=\ext_{A, W_{\vp}}\big(\can_{\vp} (M), \can_{\vp} (N)\big).$$
When, as is often the case, the forgetful functor $A\text{-}\cat{Alg}\to {}_{A}\cat {Mod}$ preserves cofibrations, and $A$ is cofibrant in ${}_{A}\cat {Mod}$, then $Q$ is a cofibrant replacement of $I$ in the category left $A$-modules. Under reasonable conditions, therefore, the $A$-module underlying $\can_{\vp} (M)$ is a model of the homotopy orbits (or homotopy indecomposables) of the $A$-action on $M$:
$$\can_{\vp} (M)=M\underset A \wedge Q = M_{hA}.$$
 Thus, the $\vp$-Adams spectral sequence in this case can be seen as converging from Quillen homology to homotopy.
 
The elaboration of this example is joint work in progress with J. E. Harper.

\subsection{The dual Grothendieck framework}\label{sec:dualhtpicGroth}

We now specialize our definitions of derived cocompletion and of the codescent spectral sequence to the dual Grothen\-dieck framework (cf. section \ref{sec:dualGroth}).  Though establishing the theoretical foundations of this special form of codescent consists essentially of a formal game of dualizing the previous case, it is a game worth playing as the spectral sequences we obtain are both important and of a rather different nature from those we studied in the previous section. A helpful reference for the material in this section is \cite{janelidze-tholen}.

To simplify the discussion somewhat, particularly in terms of the model category structures involved, we consider only the cartesian case here:  we fix a simplicial model category $\cat M$, endowed with the monoidal structure given by the categorical product $\times$.  Recall from section \ref{sec:dualGroth} that any object $X$ of $\cat M$ is a comonoid, where the comultiplication is the canonical diagonal map $\delta_{X}:X\to X\times X$ and the counit is the unique map $\epsilon_{X}:X\to e$ to the terminal object.  Moreover, the category  of right $(X,\delta_{X},  \epsilon_{X})$-comodules is isomorphic to the overcategory $\cat M/X$.

For any object $X$, the category $\cat M/X$ is endowed with a simplicial model category structure inherited from $\cat M$ \cite[Proposition II.2.6]{quillen}. The cofibrant objects in $\cat M/X$ are those morphisms $f:X\to B$ in $\cat M$ such that $X$ is cofibrant, while the fibrant objects are the fibrations $\xymatrix@1{p: X\ar @{->>}[r]&\; B}$.  If $\cat M$ is cellular (\cite[Definition 12.1.1]{hirschhorn}) and proper (\cite[Definition 13.1.1]{hirschhorn}), then $\cat M/X$ is also cellular and proper, for all objects $X$ \cite[Proposition 12.1.6]{hirschhorn}.  Examples of proper, cellular cartesian categories include $\cat {sSet}$ and $\cat{Top}$, the category of topological spaces.

Throughout this section, let 
$$\vp:E\to B$$ 
be any  morphism in the simplicial model category $\cat M$, which we view as a morphism of comonoids.  We study here the theories of derived cocompletion and of homotopic codescent for the comonad $\K_{\vp}$ associated to the adjunction
$$\vp_{!}:\cat M/E\adjunct {}{} \cat M/B:\vp^*,$$
where $\vp_{!}$ is given by postcomposing with $\vp$, while $\vp^*$ is given by pullback, where a specific representative of each pullback has been fixed.   In other words, we develop the theories of derived cocompletion and of homotopic codescent along a morphism in $\cat M$ for bundles in $\cat M$.

Let $\T_{\vp}$ denote the monad associated to the $(\vp_{!},\vp^*)$-adjunction. Recall from section \ref{sec:dualGroth} that $\cat M/E \simeq (\cat M/B)_{\K_{\vp}}$ and that 
$$\op D^{co}(\K_{\vp})\cong (\cat M/E)^{\T_{\vp}}\cong \cat M^{E}_{V^\vp},$$ 
the category of modules in $\cat M/E$ over the codescent ring 
$$V^\vp=(E\underset B\times E, \gamma_{can},\eta_{can}),$$
where $E\underset B\times E \to E$ is given by projection onto the second factor, and $\gamma_{can}$ is the composite
$$(E\underset B\times E)\underset E\times (E\underset B\times E) \cong E\underset B\times E\underset B\times E \xrightarrow {E\underset B\times \vp\underset B\times E} E\underset B\times B\underset B\times E\cong E\underset B\times E.$$
Furthermore, 
$$\can ^{\vp}:=\can ^{\T_{\vp}}:\cat M/B\to \cat M_{V^\vp}^{E}: (X\xrightarrow f B)\mapsto \big(\vp^*(f), X\underset B\times \vp\underset B\times E\big),$$
where 
$$\xymatrix{X\underset B\times E \ar [d]_{\vp^*(f)}\ar [r]&X\ar [d]^f\\ E\ar[r]^\vp& B}$$
is a pullback diagram. Finally
$$F^\T=\can ^{\vp}\circ \vp_{!}:\cat M/E\to \cat M^{E}_{V^\vp}.$$
We let $\map_{V^{\vp},E}(-,-)$ denote the induced mapping space functor on $(\cat M_{V^{\vp}}^{E})^{op}\times \cat  M_{V^{\vp}}^{E}$.

\begin{rmk}\label{rmk:cocycle} A careful analysis of the definition of $\T_{\vp}$ shows that the existence of a $\T_{\vp}$-algebra structure on a map $f:X\to E$ is equivalent to the existence of a map $\tau:X\underset B\times E \to X\underset B\times E$ satisfying both a cocycle condition (corresponding to the associativity of the multiplication) and a normalization condition (corresponding to the unitality of the multiplication).  If $\can^\vp$ is an equivalence, i.e., if $\K_{\vp}$ satisfies effective codescent, then any bundle over $E$ endowed with such a map $\tau$ can be pushed forward to a bundle over $B$, and every bundle over $B$ arises in this way, up to isomorphism. 
\end{rmk}

For the theory of derived cocompletion from section \ref{sec:cocompletion} to be applicable to this situation, 
the functor $\vp_{!}\vp^*: \cat M/B\to \cat M/B$ must be simplicial.
It is clear that this condition  always holds, since $\vp_{!}\vp^*$ is a composite of simplicial functors, by definition of the mapping spaces in the overcategories.  Note that it is also evident that the comultiplication $\Delta_{\vp}$ and the counit $\ve_{\vp}$ of the comonad $\K_{\vp}$ are simplicial natural transformations.

In this dual Grothendieck framework, Definition \ref{defn:modelKcocomp} and Notation \ref{notn:Kcocomp}  translate into the following definition.

\begin{defn}  Let $f:X\to B$ be a morphism in $\cat M$.  Let $p:\widetilde X\xrightarrow \sim X$ be a cofibrant replacement of $X$ such that $\Bar^{\K_{\vp}}_{\bullet}(fp)$ is Reedy cofibrant.  The \emph{derived cocompletion of $f$ along $\vp$} is
$$(f)^\vee_{\vp}:=(f)^\vee_{\K_{\vp}}=|\Bar^{\K_{\vp}}_{\bullet}(fp)|.$$
\end{defn}

\begin{rmk} If $\cat M=\cat {sSet}$, then all objects are in $\cat M$ are cofibrant, whence the same is true of $\cat M/X$, for all $X$.  Moreover, $\Bar^{\K_{\vp}}_{\bullet}(f)$  is cofibrant for all $f:X\to B$, since all objects in $\cat {sSet}^{\bold\Delta^{op}}$ are Reedy cofibrant.  It makes sense therefore to write
$$(f)^\vee_{\vp}=|\Bar^{\K_{\vp}}_{\bullet}(f)|$$
for all simplicial maps $f:X\to B$.
\end{rmk}

\begin{rmk}\label{rmk:cech}  Expanding the formula for the functor $\Bar^{\K_{\vp}}_{\bullet}$, we note that if $f:X\to B$ is a morphism in $\cat M$, then 
$$\Bar^{\K_{\vp}}_{\bullet}(f)=\Bigg(\xymatrix{X\underset B\times E\ar[dr]_{\vp_{!}\vp^*(f)}\ar@<0.7ex> [r]& X\underset B\times E \underset B\times E\ar[d]^{(\vp_{!}\vp^*)^2(f)}\ar @<0.3ex>[l]\ar @<-1.7ex>[l]\ar @<1.7ex>[r]\ar @<-0.3ex>[r]&X\underset B\times E \underset B\times E\underset B\times E\cdots\ar [dl]^{(\vp_{!}\vp^*)^3(f)} \ar @<1.3ex>[l]\ar@<-0.7ex>[l]\ar @<-2.7ex>[l]\\ &B}\Bigg).$$
In particular, if $\cat M=\cat {sSet}$, then $\Bar^{\K_{\vp}}_{\bullet}(\vp)$ is the bisimplicial set, often called the \v Cech nerve, underlying the ``cohomological descent spectral sequence of a map'' (cf. e.g., \cite {pirashvili} and \cite[Section 5.3]{deligne}). 
\end{rmk}

\begin{rmk}\label{rmk:charcoalg}  From the formula for $\K_{\vp}$, one sees that a $\K_{\vp}$-coalgebra structure on $f:X\to B$ is equivalent to a  section 
$$\sigma: X\to X\underset B\times E: x\mapsto \big(x,e(x)\big) $$ 
of the projection $\ve_{f}:X\underset B\times E\to X$, which is the counit of the adjunction, such that $e(x)\in \vp^{-1}\big(f(x)\big)$, since any such section automatically satisfies the coassociativity and counit conditions of a coalgebra.

It follows that the strictly $\K_{\vp}$-cocomplete objects in $\cat M/B$ are exactly the $\K_{\vp}$-coalgebras: for all $f:X\to B$, a simplicial section $\sigma _{\bullet}:cs_{\bullet}(f) \to \Bar^{\K_{\vp}}_{\bullet}(f)$ of the simplicial augmentation $\ve_{\bullet}$ begins with a section $\sigma_{0}:X\to X\underset B\times E$ of $\ve_{f}$ as maps over $B$.  This bottom stage of the simplicial section therefore defines a $\K_{\vp}$-coalgebra structure on $f$.
\end{rmk}

Specializing Definition \ref{defn:htpiccodesc} to the dual Grothendieck framework, we obtain the following notion of homotopic codescent.  Note that this definition makes sense even if the  overcategories admit only a simplicial enrichment, rather than a full simplicial model category structure.

\begin{defn} The comonad $\K_{\vp}$ satisfies \emph{homotopic codescent} if each component
$$(\can ^{\vp})_{f,g}:\map_{\cat M/B}(f,g)\to \map_{V^\vp,E}\big(\can^\vp(f),\can^\vp(g)\big)$$
of the simplicial functor $\can^\vp$ is a weak equivalence of simplicial sets.
\end{defn}

\begin{rmk}  Since the codescent category $\op D^{co}(\K_{\vp})$ is isomorphic to a category of modules, $\cat M_{V^\vp}^{E}$, this definition implies that if $\K_{\vp}$ satisfies homotopic codescent, then $\cat M/B$ is locally homotopy coTannakian.
\end{rmk}

The criterion for homotopic codescent in the dual Grothendieck framework is a special case of Theorem \ref{thm:equiv-comonad}, where Remark \ref{rmk:charcoalg} comes clearly into play.

\begin{thm} Suppose that the forgetful functor $U^{\T_{\vp}}:\cat M_{V^{\vp}}^{E}\to \cat M/E$ right-induces a model category structure on $\cat M_{V^{\vp}}^{E}$. Let $(\cat M/B)^\vee_{\K_{\vp}}$ denote the full simplicial subcategory of $\cat M/B$ determined by the bifibrant $\K_{\vp}$-coalgebras. If $\vp_{!}\vp^*:\cat M/B\to \cat M/B$ preserves bifibrant objects, and $\Bar^{\K_{\vp}}_{\bullet}$ preserves cofibrant objects, then $\K_{\vp}$ restricts and corestricts to a comonad $\check \K_{\vp}$ on $(\cat M/B)^\vee_{\K_{\vp}}$ that satisfies homotopic codescent.
\end{thm}

\begin{rmk} The forgetful functor $U^{\T_{\vp}}:\cat M_{V^{\vp}}^{E}\to \cat M/E$ right-induces a model category structure on $\cat M_{V^{\vp}}^{E}$ (cf. appendix \ref{sec:indmodcat}) if, for example, the model category structure on $\cat M$ is cellular (and therefore cofibrantly generated), all objects are small and the so-called monoid axiom is satisfied \cite[Theorem 4.1(1)]{schwede-shipley}, e.g.,  if $\cat M=\cat {sSet}$. Moreover, by Proposition \ref{prop:schwede-shipley}, $U^{\T_{\vp}}$ also right-induces a model category structure on $\cat M_{V^{\vp}}^{E}$ if $\cat M$ is cellular and proper and $\vp$ is a fibration, since $F^{\T_{\vp}}$ will then preserve all weak equivalences.
\end{rmk}

\begin{rmk}  If $\vp: E\to B$ is a fibration, then $\vp_{!}\vp^*:\cat M/B\to \cat M/B$ preserves fibrant objects.  Thus, if $\cat M=\cat {sSet}$, and $\vp$ is a fibration, then $\vp_{!}\vp^*:\cat M/B\to \cat M/B$ preserves bifibrant objects.   The criterion for homotopic codescent therefore applies to all Kan fibrations.

If a Kan fibration $\vp:E\to B$ admits a section, then for all $f:X\to B$, the counit map $\ve_{f}:X\underset B\times E\to X$ also admits a section, i.e., every $f$ is a $\K_{\vp}$-coalgebra.  It follows that $\K_{\vp}$ itself then satisfies homotopic codescent.
\end{rmk}

The relevant spectral sequence in this case is defined as follows.

\begin{defn}  Let $i:X\to B$ and $p:Y\to B$ be  morphisms in $\cat M$, and let $f:X\to Y$ be a morphism in $\cat M$ satisfying $pf=i$.  The \emph{$\vp$-bundle spectral sequence for $(i,p)$ with respect to $f$}, denoted $E^\vp_{f}$, is the $\K_{\vp}$-codescent spectral sequence for $(i,p)$ with respect to $f$.
\end{defn}

The next result is an immediate specialization of Lemma \ref{lem:co-einfty} and Theorem \ref{thm:e2-codesc}.

\begin{prop}Let $i:X\to B$ and $p:Y\to B$ be morphisms in $\cat M$ such that $X$ is cofibrant and $p$ is a fibration. 
\begin{enumerate}
\item  If $\Bar^{\K_{\vp}}_{\bullet}$ preserves cofibrant objects in $\cat M/B$, then the $\vp$-bundle spectral sequence for $(i,p)$ with respect to $f$ abuts to $\pi_{*}\map_{\cat M/B}(i^\vee_{\vp}, p)$, if it indeed converges.
\item Suppose that $U^{\T_{\vp}}:\op D^{co}(\K)\to \cat M/E$ right-induces a model category structure on $\op D^{co}(\K)$
and that the usual simplicial enrichment of $\op D^{co}(\K)$ (Lemma \ref{lem:simplenr}) extends to a simplicial model category structure.  

If $\Bar^{\K_{\vp}}_{\bullet}(i)$ is Reedy cofibrant and $\can^\vp(i)$ is strongly $\mathbb{DK}_{\vp}$-cocomplete, then 
$$(E^\vp_{f})_{2}^{s,t}\cong \ext^{s,t}_{V^\vp, E}\big( \can^\vp (i), \can^{\vp} (p)\big)_{\can ^\vp f}$$
for any choice of basepoint $f$.
\end{enumerate}
\end{prop}

As a concrete example of the theory developed here, we present the following example.

\begin{ex}  Let $B$ be any reduced simplicial set, and let $\vp:PB\to B$ be the Kan fibration with $PB=B\underset \tau \times GB$, the twisted cartesian product of $B$ and its Kan loop group $GB$ over the universal twisting function $\tau: B\to GB$  \cite{may}. The associated codescent ring is
$$V^\vp=PB\underset B\times PB,$$
which is multiplicatively homotopy equivalent to $GB$, via the simplicial map
$$PB\underset B\times PB\to GB: \big((\lambda_{1}, b), (\lambda_{2},b)\big)\mapsto \lambda_{1}\lambda _{2}^{-1}.$$
Moreover, for all $f:X\to B$,  the canonical codescent datum $\can^\vp(f)$ is the natural projection from homotopy fiber of $f$ down to $PB$, endowed with the usual action of $GB$.  

In this case, since $\vp$ itself admits a section, the counit (projection) $\ve_{f}:X\underset B\times E\to X$ admits a canonical section for all $f:X\to B$. In other words, by Remark \ref{rmk:charcoalg}, every object in $\cat{sSet}/B$ is strongly $\K_{\vp}$-cocomplete and therefore $\K_{\vp}$-cocomplete, since all objects are cofibrant.   

Abusing notation a little, we see that for all simplicial maps $f:X\to B$ and all Kan fibrations $p:Y\to B$,
Proposition \ref{prop:cocriterion} implies that
$$\can^{\vp}_{f,p}: \map_{\cat M/B}(f,p)\to \map_{PB,GB}\big(\operatorname{hfib}(f), \operatorname{hfib}(p)\big)$$
is a weak equivalence.
\end{ex}

\appendix

\section {The proof of Theorem \ref{thm:assembly}}\label{sec:proof}

Since the proof of the theorem relating assembly to cocompletion is of a rather different nature from the rest of the paper, we  have reserved it for this appendix.

\begin{proof}[Proof of Theorem \ref{thm:assembly}]
We begin by establishing,  via a dual version of the argument in \cite[XI.10]{bousfield-kan}, a formula for $|\Bar^{\K_{\Phi}}_{\bullet}(X)(D)|$ that is analogous to the definition of $K^{ho}_{\Phi}(X)(D)$.    Let $\cat {Simp}_{K_\Phi}(X)$ be the full subcategory of the overcategory $(\cat M^{\cat D}) /X$ with object set 
$$\{ \tau :K_{\Phi}(Y)\to X \in \mor \cat M^{\cat D}\mid Y:\cat D\to \cat M \text{ such that $K_{\Phi}(Y)$ is cofibrant } \},$$
and let
$$\operatorname{Dom}:\cat {Simp}_{K_\Phi}(X)\to \cat M^\cat D: (\tau :K_{\Phi}(Y)\to X)\mapsto  K_{\Phi}(Y)$$
be the ``domain'' functor. Note that $(\ve_{\Phi})_{X}:K_{\Phi}(X) \to X$ is an object of $\cat {Simp}_{K_\Phi}(X)$.

Since $\Bar^{\K_{\Phi}}_{\bullet}X$ comes equipped with a coaugmentation down to $X$, via the counit of the comonad $\K_{\Phi}$, the functor $\Bar^{\K_{\Phi}}_{\bullet}X:\bold \Delta ^{op}\to \cat M^{\cat D}$ factors through $\cat {Simp}_{K_\Phi}(X)$, as
$$\Bar^{\K_{\Phi}}_{\bullet}X=\operatorname{Dom}\circ \widetilde\Bar^{\K_{\Phi}}_{\bullet}X.$$
The dual of Proposition XI.9.3 in \cite {bousfield-kan} implies that $\widetilde\Bar^{\K_{\Phi}}_{\bullet}X$ is homotopy right cofinal, as 
\begin{enumerate}
\item $\bold\Delta ^{op}$ is obviously right filtered, 
\medskip

\item for all objects $\tau: K_{\Phi}(Y) \to X$ in $\cat {Simp}_{K_\Phi}(X)$
$$\xymatrix{ K_{\Phi}(Y) \ar [rr]^{K_{\Phi}(\tau)\circ (\Delta_{\Phi})_{Y}}\ar [dr]_{\tau}&&K_{\Phi}(X)\ar [dl]^{(\ve_{\Phi})_{X}}\\
&X}$$
commutes, and
\medskip
  
\item for every commuting diagram
$$\xymatrix{ K_{\Phi}(Y) \ar [rr]^{\sigma}\ar [dr]_{\tau}&&K_{\Phi}(X)\ar [dl]^{(\ve_{\Phi})_{X}}\\
&X},$$
the morphism $(\Delta_{\Phi})_{X}:K_{\Phi}(X)\to K_{\Phi}^2(X)$ coequalizes $\sigma $ and $K_{\Phi}(\tau)\circ (\Delta_{\Phi})_{Y}$:
\begin{align*}
(\Delta_{\Phi})_{X}K_{\Phi}(\tau)(\Delta_{\Phi})_{Y}=&(\Delta_{\Phi})_{X}K_{\Phi}\big((\ve_{\Phi})_{X}\big)K_{\Phi}(\sigma)(\Delta_{\Phi})_{Y}\\
=&(\Delta_{\Phi})_{X}K_{\Phi}\big((\ve_{\Phi})_{X}\big)(\Delta_{\Phi})_{X}\sigma\\
=& (\Delta_{\Phi})_{X}\sigma.
\end{align*}
\end{enumerate}
Theorem 19.6.7 in \cite{hirschhorn} then implies that the natural map
$$\operatorname{hocolim}_{\bold \Delta ^{op}}\Bar^{\K_{\Phi}}_{\bullet}X  \to \operatorname{hocolim}_{\cat {Simp}_{K_\Phi}(X)} \operatorname{Dom}$$
is a weak equivalence, whence 
$$|\Bar^{\K_{\Phi}}_{\bullet}X|\sim \operatorname{hocolim}_{\cat {Simp}_{K_\Phi}(X)} \operatorname{Dom},$$
since the Reedy cofibrancy of $\Bar^{\K_{\Phi}}_{\bullet}X$ implies that the Bousfield-Kan map
$$\operatorname{hocolim}_{\bold \Delta ^{op}}\Bar^{\K_{\Phi}}_{\bullet}X  \to |\Bar^{\K_{\Phi}}_{\bullet}X|$$
is a weak equivalence \cite[Theorem 18.7.4]{hirschhorn}.  The derived $\K_{\Phi}$-cocompletion of  $X$ therefore admits a description highly analogous to that of $K^{ho}_{\Phi}(X)(D)$.

The next step in the proof consists in showing that for all $X\in \ob {\cat M}^{\cat D} $ and all $D\in \ob \cat D$,
\begin{equation}\label{eqn:goal}(\colim _{\cat {Simp}_{K_\Phi}(X)}\operatorname{Dom})(D)\cong K_{\Phi}(X)(D),\end{equation}
by expressing both sides, up to isomorphism, as the same colimit in two variables.
Let 
$$E_{X,D}:\cat {Simp}_{K_\Phi}(X)\times \cat {Simp}_{\Phi}(D) \to \cat M$$
denote the functor specified on objects by 
$$E_{X,D}\big(\tau:K_{\Phi}(Y)\to X, \alpha:\Phi(C)\to D\big)= K_{\Phi}(Y)\big(\Phi (C)\big).$$
Note that for all $\tau:K_{\Phi}(Y)\to X$,
\begin{equation}\label{eqn:Kphi}K_{\Phi}(Y)\big(\Phi (C)\big)=\colim_{\cat {Simp}_{\Phi}\big(\Phi(C)\big)} (Y\circ \operatorname{dom})=Y\big(\Phi(C)\big),\end{equation}
since the fact that $\Phi$ is full implies that  the identity map on $\Phi(C)$ is a terminal object in $\cat {Simp}_{\Phi}\big(\Phi(C)\big)$.  Thus, 
$$(\ve_{\Phi})_{X, \Phi(C)}=Id_{X(\Phi (C))},$$
which implies that if
$$\xymatrix{K_{\Phi}(Y) \ar [dr]_{\tau}\ar [rr]^\sigma &&K_{\Phi}(X)\ar [dl]^{(\ve_{\Phi})_{X}}\\&X}$$
is a morphism in $\cat {Simp}_{K_\Phi}(X)$, then $\tau_{\Phi(C)}=\sigma_{\Phi(C)}$ for all objects $C$ in $\cat C$.   In particular, observation (2) above implies that
\begin{equation}\label{eqn:tau}\tau_{\Phi (C)}=\big(K_{\Phi}(\tau)\circ (\Delta_{\Phi})_{Y}\big)_{\Phi(C)},\end{equation}
for all $\tau:K_{\Phi}(Y)\to X$.

The key to establishing (\ref{eqn:goal}) is the fact that for all objects $C$ in $\cat C$, 
\begin{equation}\label{eqn:colim}\colim_{\cat {Simp}_{K_\Phi}(X)} Y\big(\Phi (C)\big) \cong X\big(\Phi(C)\big),\end{equation}
which we prove as follows. If $M$ is an object in $\cat M$, let $\operatorname{Cst}_{M}:\cat {Simp}_{K_\Phi}(X)\to \cat M$ denote the constant functor at $M$.  

If 
$$\xymatrix{K_{\Phi}(Y) \ar [dr]_{\tau}\ar [rr]^\sigma &&K_{\Phi}(Y')\ar [dl]^{\tau'}\\&X}$$
is any morphism in $\cat {Simp}_{K_\Phi}(X)$, then evaluating at $\Phi(C)$ gives rise, by (\ref{eqn:Kphi}), to a commutative diagram in $\cat M$:
$$\xymatrix{Y\big(\Phi(C)\big) \ar [dr]_{\tau_{\Phi(C)}}\ar [rr]^{\sigma_{\Phi(C)} }&&Y'\big(\Phi(C)\big).\ar [dl]^{\tau_{\Phi(C)}'}\\&X\big(\Phi(C)\big)}$$
There is therefore a natural transformation
$$\Gamma:  (-)\big(\Phi(C)\big) \to \operatorname{Cst}_{X(\Phi(C))}:\cat {Simp}_{K_\Phi}(X) \to \cat M,$$
where $\Gamma_{\tau}=\tau_{\Phi(C)}: Y\big(\Phi(C)\big)\to X\big(\Phi(C)\big)$ for all $\tau: K_{\Phi}(Y) \to X$.
Moreover, if $M$ is any object in $\cat M$, and
$$\Psi:  (-)\big(\Phi(C)\big) \to \operatorname{Cst}_{M}:\cat {Simp}_{K_\Phi}(X) \to \cat M$$
is any natural transformation, then
$$\Psi_{(\ve_{\Phi})_{X}}: X\big(\Phi(C)\big) \to M,$$
and for all objects $\tau:K_{\Phi}(Y)\to X$ in $\cat {Simp}_{K_\Phi}(X)$
$$\Psi_{\tau}=\Psi_{(\ve_{\Phi})_{X}} \circ \Gamma_{\tau}: Y\big(\Phi(C)\big) \to M.$$
To prove this last equality, we observe that since $\Psi$ is a natural transformation, and  $K_{\Phi}(\tau)\circ (\Delta_{\Phi})_{Y}$ is a morphism from $\tau $ to $(\ve_{\Phi})_{X}$ for all $\tau:K_{\Phi}(Y)\to X$, we can deduce from (\ref{eqn:tau}) that
$$\Psi_{\tau}=\Psi_{(\ve_{\Phi})_{X}}\circ \big(K_{\Phi}(\tau)\circ (\Delta_{\Phi})_{Y}\big)_{\Phi (C)}= \Psi_{(\ve_{\Phi})_{X}}\circ \tau_{\Phi(C)}.$$

It follows from equation (\ref{eqn:colim}) that
\begin{align*}
K_{\Phi}(X)(D)&=\colim_{\cat {Simp}_{\Phi}(D)} (X\circ \operatorname{dom})\\
&=\colim_{\cat {Simp}_{\Phi}(D)} X\big(\Phi (C)\big)\\
&\cong \colim_{\cat {Simp}_{\Phi}(D)}\colim_{\cat {Simp}_{K_\Phi}(X)}   Y\big(\Phi (C)\big)\\
&\cong \colim _{\cat {Simp}_{K_\Phi}(X)\times\cat {Simp}_{\Phi}(D)} E_{X,D}.
\end{align*}
On the other hand,
\begin{align*}
(\colim _{\cat {Simp}_{K_\Phi}(X)}\operatorname{Dom})(D)&=\colim _{\cat {Simp}_{K_\Phi}(X)}\big(K_{\Phi}(Y)(D)\big)\\
&=\colim _{\cat {Simp}_{K_\Phi}(X)}\colim_{\cat {Simp}_{\Phi}(D)} (Y\circ \operatorname{dom})\\
&=\colim _{\cat {Simp}_{K_\Phi}(X)}\colim_{\cat {Simp}_{\Phi}(D)} Y\big(\Phi(C)\big)\\
&\cong \colim _{\cat {Simp}_{K_\Phi}(X)\times\cat {Simp}_{\Phi}(D)} E_{X,D},
\end{align*}
since colimits in $\cat M^{\cat D}$ are calculated objectwise.
We can conclude that 
$$(\colim _{\cat {Simp}_{K_\Phi}(X)}\operatorname{Dom})(D)\cong \colim _{\cat {Simp}_{K_\Phi}(X)\times\cat {Simp}_{\Phi}(D)} E_{X,D}\cong K_{\Phi}(X)(D).$$

To complete the proof of the theorem, we exhibit morphisms from 
$$\hocolim _{\cat {Simp}_{K_\Phi}(X)\times\cat {Simp}_{\Phi}(D)} E_{X,D}$$ to  $(\hocolim_{\cat {Simp}_{K_\Phi}(X)}\operatorname{Dom})(D)$ and to $K^{ho}_{\Phi}(X)(D)$.  Since we have already shown that $|\Bar^{\K_{\Phi}}_{\bullet}X|$ is weakly equivalent to $\hocolim_{\cat {Simp}_{K_\Phi}(X)}\operatorname{Dom}$, we can then conclude.

Since there is always a natural map from a homotopy colimit to the corresponding colimit, and a homotopy colimit over a product of categories is weakly equivalent to double homotopy colimit calculated over each factor category consecutively,
there are natural maps
$$\xymatrix{
\hocolim_{\cat {Simp}_{K_\Phi}(X)}\colim_{\cat {Simp}_{\Phi}(D)} E_{X,D}& \hocolim _{\cat {Simp}_{K_\Phi}(X)\times\cat {Simp}_{\Phi}(D)} E_{X,D}\ar [l]_-{}\ar[d]\\
&\hocolim_{\cat {Simp}_{\Phi}(D)}\colim_{\cat {Simp}_{K_\Phi}(X)} E_{X,D}.}$$

Because 
$$\colim_{\cat {Simp}_{\Phi}(D)} E_{X,D}(\tau,-)=\colim_{\cat {Simp}_{\Phi}(D)}Y\big(\Phi (C)\big)= K_{\Phi}(Y)(D)$$
for all $\tau:K_{\Phi}(Y)\to X$ by definition of $K_{\Phi}$, and
$$\colim_{\cat {Simp}_{K_\Phi}(X)} E_{X,D}(-,\alpha)=\colim_{\cat {Simp}_{K_\Phi}(X)} Y\big(\Phi (X)\big)=X\big(\Phi(C)\big)$$
for all $\alpha:\Phi (C) \to D$ by  (\ref{eqn:colim}), it follows that
$$\hocolim_{\cat {Simp}_{K_\Phi}(X)}\colim_{\cat {Simp}_{\Phi}(D)} E_{X,D}\cong (\hocolim_{\cat {Simp}_{K_\Phi}(X)}\operatorname{Dom})(D)$$
and
$$ \hocolim_{\cat {Simp}_{\Phi}(D)}\colim_{\cat {Simp}_{K_\Phi}(X)} E_{X,D}\cong K^{ho}_{\Phi}(X)(D).$$
\end{proof}

\section{Useful simplicial structures}

In this appendix we recall various simplicial structures that are necessary to this paper, in order to fix notation and terminology and to make this paper as self-contained as possible and reasonable.

\subsection{Reedy model category structure}\label{sec:reedy}

Let $\cat M$ be any model category. We describe here the fibrant objects in the Reedy model category structure on $\cat M^{\bold\Delta}$ and the cofibrant objects in the Reedy model category structure on $\cat M^{\bold\Delta^{op}}$.  We refer the reader to Chapter 15 of \cite{hirschhorn} for a complete description of these model category structures.

By Corollary 15.10.5 in \cite{hirschhorn}, if $X$ is a fibrant object of $\cat M$, then $cc^\bullet X$ is Reedy fibrant in $\cat M^{\bold\Delta}$ .  Dually, if $X$ is cofibrant, then $cs_{\bullet}X$ is Reedy cofibrant in $\cat M^{\bold\Delta^{op}}$ .

More generally, if $X^\bullet$ is any cosimplicial object in $\cat M$, then the \emph{$n^{\text{th}}$-matching object }Êof $X^\bullet$ is 
$$M_{n}X^\bullet= \operatorname{equal} \Big( \prod_{0\leq i\leq n-1}X^{n-1} \egal{\psi_{1}}{\psi_{2}} \prod_{0\leq i<j\leq n-2} X^{n-2}\Big),$$
where, if  $p_{i}:\prod_{0\leq i\leq n}X^{n-1}\to X^{n-1}$ and $p_{i,j}:\prod_{0\leq i<j\leq n-1} X^{n-2}\to X^{n-2}$ are the obvious projections, then 
$$p_{i,j}\psi_{1}=s^{i}p_{j}\quad \text{and}\quad p_{i,j}\psi_{2}=s^{j-1} p_{i}.$$
The cosimplicial object $X^\bullet$ is Reedy fibrant if the natural map 
$$\sigma_{n}:X^n\to M_{n}X^\bullet,$$
which is induced by the codegeneracies $s^0,...,s^{n-1}:X^n\to X^{n-1}$, is a fibration for all $n$.

Dually, if $X_{0}$ is any simplicial object in $\cat M$, then the \emph{$n^{\text{th}}$-latching object}Ê of $X_\bullet$ is  
$$L^nX_\bullet= \operatorname{coequal} \Big( \coprod_{0\leq i<j\leq n-2}X_{n-2} \egal{\psi_{1}}{\psi_{2}} \coprod_{0\leq i\leq n-1} X_{n-1}\Big),$$
where, if  $\iota_{i}:X_{n-1}\to \coprod_{0\leq i\leq n}X_{n-1}$ and $\iota_{i,j}:X_{n-2}\to\coprod_{0\leq i<j\leq n-1} X_{n-2}$ are the obvious summand ``inclusions,''  then 
$$\psi_{1}\iota_{i,j}=\iota_{j}s_{i}\quad \text{and}\quad \psi_{2}\iota_{i,j}=\iota_{i}s_{j-1}.$$
The simplicial object $X_\bullet$ is Reedy cofibrant if the natural map 
$$\sigma^n:L^nX_\bullet\to X_{n},$$
which is induced by the degeneracies $s_0,...,s_{n-1}:X_{n-1}\to X_{n}$, is a cofibration for all $n$.

\subsection{Simplicial enrichments, tensoring and cotensoring}

For a detailed introduction to the theory of enriched categories, we refer the reader to \cite{kelly}.  Here we recall only those elements of the theory of enrichments over $\cat{sSet}$ that we use in this paper, from the point of view usual to homotopy theorists: we think of simplicially enriched categories as ordinary categories with extra structure (cf. \cite[Ch 9]{hirschhorn}).

\begin{defn}\label{defn:enrich} Let $\cat S$ be a category.  A \emph{simplicial enrichment} of $\cat S$ consists of a function
$$\map_{\cat S}: \ob\cat S \times \ob \cat S \to \ob\cat {sSet}$$ 
together with two collections of set maps
$$\big\{c_{X,Y,Z}:\map_{\cat S}(Y,Z)\times \map_{\cat S}(X,Y) \to \map_{\cat S}(X,Z)\mid X,Y,Z\in \ob \cat S\big\}$$
called \emph{simplicial composition}, and
$$\big\{ i_{X}:*\to \map_{\cat S}(X,X)\mid X\in \ob \cat S\big\},$$
called \emph{units}, such that simplicial composition is associative and unital, in the obvious sense.  Moreover, for all $X,Y\in \ob \cat S$, there is a bijection
$$\map_{\cat S}(X,Y)_{0}\cong \cat S(X,Y)$$
that is compatible with composition.
\end{defn}

We usually drop the subscripts on simplicial composition, in the interest of simplifying notation.

\begin{rmk}\label{rmk:inducedsimplmap}  Suppose that $\cat S$ is simplicially enriched.  A morphism $f:Y\to Y'$ in $\cat S$ induces morphisms of simplicial sets 
$$f_{*}:\map_{\cat S}(X,Y)\to \map _{\cat S}(X,Y')\quad\text{and}\quad f^*:\map_{\cat S}(Y',X)\to \map _{\cat S}(Y,X)$$
for all objects $X$ in $\cat S$, which are defined as follows.  If $g \in\map_{\cat S}(X,Y)_{n}$, then
$$f_{*} (g)=c(s_{0}^nf, g),$$
which makes sense since we can view $f$ as a $0$-simplex of $\map_{\cat S}(Y,Y')$.  Similarly, if $g \in\map_{\cat S}(Y',X)_{n}$, then
$$f^*(g)=c(g, s_{0}^nf).$$
 \end{rmk}
 
Recall that associativity of composition in a category implies that if both small squares of a diagram
 $$\xymatrix{X\ar [d]_{f}\ar [r]^{a}&Y\ar [d]_{g}\ar [r]^b& Z\ar[d] _{h}\\ X'\ar [r]^{a'}&Y'\ar [r]^{b'}& Z'}$$
 commute, then so does the outer rectangle.  The next lemma, which is a simplicial analog of this fact that proves useful in section \ref{sec:simplenr}, 
 is an easy consequence of the associativity of the simplicial composition maps.  
 
 \begin{lem}\label{lem:simplcomp} Let $\cat S$ be a simplicially enriched category.  Let $f:X\to X'$, $g: Y\to Y'$ and $h:Z\to Z'$ be morphisms in $\cat S$.  Let $a\in  \map_{\cat S}(X,Y)_{n}$, $b\in \map_{\cat S}(Y,Z)_{n}$, $a'\in \map_{\cat S}(X',Y')_{n}$, and $b'\in \map_{\cat S}(Y',Z')_{n}$.
 
 If $g_{*}(a)=f^*(a')$ and $h_{*}(b)=g^*(b')$, then $h_{*}\big(c(b,a)\big)=f^*\big(c(b',a')\big)$.
 \end{lem}

\begin{defn} Let $\cat S$ and $\cat S'$ be simplicially enriched categories.  A \emph{simplicial functor} $F$ from $\cat S$ to $\cat S'$ consists of 
\begin{enumerate}
\item a function $F:\ob S\to \ob S'$,
 and
 \item a collection of simplicial maps 
 $$\big\{ F_{X,Y}:\map_{\cat S}(X,Y)\to \map _{\cat S'}(FX, FY)\mid X,Y\in \ob \cat S\big\},$$
 called the \emph{components} of the simplicial functor, 
 \end{enumerate}
 such that 
 $$c\circ (F_{Y,Z}\times F_{X,Y})=F_{X,Z}\circ c\quad\text{and}\quad F_{X,X}\circ i_{X}=i_{FX}$$
 for all $X,Y,Z\in \ob \cat S$.
\end{defn}

\begin{defn} Let $F,G:\cat S\to \cat S'$ be two simplicial functors.  A \emph{simplicial natural transformation} $\alpha$ from $F$ to $G$ consists of a collection
$$\{\alpha _{X}:*\to \map _{\cat S'}(FX, GX)\mid X\in \ob \cat S\}$$
of simplicial maps such that 
$$(\alpha_{Y})_{*}\circ F_{X,Y}=(\alpha_{X})^*\circ G_{X,Y}: \map_{\cat S}(X,Y)\to \map _{\cat S'}(FX, GY)$$
for all $X,Y\in \ob \cat S$.
\end{defn}

\begin{defn} Let $\map_{\cat{sSet}}$ denote the usual mapping space functor in $\cat{sSet}$. A simplicially enriched category $\cat S$ is \emph{tensored over $\cat{sSet}$} if there is a functor
$$-\otimes -:\cat S \times \cat {sSet} \to \cat S$$
together with natural isomorphisms 
$$\map _{\cat S}(X\otimes L, Y)\cong\map_{\cat{sSet}} \big(L, \map _{\cat S}(X,Y)\big)$$
for all $X,Y\in \ob \cat S$ and $L\in \ob \cat {sSet}$.
Dually, $\cat S$ is \emph{cotensored over $\cat{sSet}$} if there is a functor
$$(-)^{(-)}:\cat S\times \cat {sSet}^{op}\to \cat S,$$ 
 together with natural isomorphisms
$$\map_{\cat S} (X, Y^L)\cong \map_{\cat{sSet}} \big(L, \map _{\cat S}(X,Y)\big)$$
for all $X,Y\in \ob \cat S$ and $L\in \ob \cat {sSet}$.
\end{defn}

\begin{rmk} If $\cat S$ is tensored and cotensored over $\cat {sSet}$, then, for all simplicial sets $L$, the functors $-\otimes L$ and $(-)^L$ form an adjoint pair,  with counit  $ev: X^L\otimes L\to X$ and unit $coev: X\to (X\otimes L)^L$.  It follows as well that 
$$\map_{\cat S}(X,Y)_{n}\cong\cat S(X\otimes \Delta[n], Y)\cong \cat S(X, Y^{\Delta[n]})$$
for all $X,Y\in \ob \cat S$ and $n\geq 0$.
\end{rmk}

In a simplicially enriched category, one can define a notion of homotopy between morphisms as follows.

\begin{defn} \label{defn:simphop} Let $\cat S$ be a simplicially enriched category, and let $f,f':X\to Y$ be morphisms in $\cat S$, considered as 0-simplices of $\map_{\cat S}(X,Y)$.  If there is a 1-simplex $h$ of $\map_{\cat S}(X,Y)$ such that $d_{1}h=f$ and $d_{0}h=f'$, then $f$ is \emph{strictly simplicially homotopic} to $f'$.  If $f$ and $f'$ are equivalent under the equivalence relation generated by the relation of strict simplicial homotopy, then they are \emph{simplicially homotopic}, denoted $f\underset s\sim g$.

A morphism $f:X\to Y$ is a \emph{simplicial homotopy equivalence} if there is a morphism $g:Y\to X$ such that $gf\underset s\sim X$ and $fg\underset s \sim Y$.
\end{defn}

\begin{rmk}  There are alternative characterizations of simplicial homotopy, if $\cat S$ is tensored or cotensored over $\cat {sSet}$, expressed in terms of \emph{generalized intervals}, which are formed by gluing standard 1-simplices $\Delta[1]$ along their vertices, so that the geometric realization is homeomorphic to the unit interval.   If $J$ is a generalized interval, let $j_{0}$ and $j_{1}$ denote the inclusions of $\Delta[0]$ as the two endpoints of $J$.

If $\cat S$ is tensored over $\cat {sSet}$, then morphisms $f,f':X\to Y$ in $\cat S$ are simplicially homotopic if and only if there is a generalized interval $J$ and a morphism $h:X\otimes J \to Y$ in $\cat S$ such that $h(X\otimes j_{0})=f$ and $h(X\otimes j_{1})=f'$.  Dually, if $\cat S$ is cotensored over $\cat {sSet}$, then morphisms $f,f':X\to Y$ in $\cat S$ are simplicially homotopic if and only if there is a generalized interval $J$ and a morphism $h:X \to Y^J$ in $\cat S$ such that $Y^{j_{0}}\circ h=f$ and $Y^{j_{1}}\circ h=f'$.
\end{rmk}

We can now define a useful quotient of any simplicially enriched category.

\begin{defn}\label{defn:pathcomp} 
If $\cat S$ is a simplicially enriched category, then its \emph{path-component category} $\pi_{0}\cat S$ is the ordinary category with $\ob \pi_{0}\cat S=\ob\cat S$ and $\pi_{0}\cat S(X,Y)=\map_{\cat S}(X,Y)_{0}/\underset s\sim$.
\end{defn}

\subsection {Simplicial model categories}\label{sec:simplmodelcat}

We sketch here those elements of the theory of simplicial model categories that we need in this paper.  We refer the reader to, e.g., \cite[Ch 9]{hirschhorn}, for a thorough exposition of the subject, including a detailed discussion of homotopy (co)limits, as well as of  the Tot and  realization functors.

\begin{defn}\label{defn:simplmodel}  A model category that is simplicially enriched, as well as tensored and cotensored over $\cat {sSet}$, is a \emph{simplicial model category} if for every cofibration $\xymatrix@1{i:X\;\cof& \; Y}$ and every fibration $\xymatrix@1{p:E\;\fib& \;B}$ in $\cat M$, the induced map of simplicial sets
$$\map_{\cat M}(Y,E)\xrightarrow {(p_{*},i^*)} \map_{\cat M}(Y,B)\underset {\map_{\cat M}(X,B)}\times \map_{\cat M}(X,E)$$
is a fibration that is acyclic if $i$ or $p$ is a weak equivalence.
\end{defn}

\begin{thm}\cite [Theorem 9.8.5]{hirschhorn}\label{thm:simplfunctor}  Let $\cat M$ and $\cat M'$ be simplicial model categories. A functor $F:\cat M\to \cat M'$ can be extended to a simplicial functor if and only if there is a natural transformation 
$$\theta: F(-)\otimes - \Rightarrow F(-\otimes -)$$
of functors from $\cat M\times \cat{sSet}$ into $\cat M'$ such that  
\begin{enumerate}
\item a unit condition holds: $\theta_{X,\Delta_{0}}:F(X)\otimes \Delta [0] \to F(X\otimes \Delta[0])$ is an isomorphism for all $X$, compatible with the natural isomorphisms $X\otimes \Delta [0]\cong X$ and $F(X)\otimes \Delta [0]\cong F(X)$, and
\item an associativity condition holds: the two possible natural composites $$F(X)\otimes (L\times L')\to F\big((X\otimes L)\otimes L'\big)$$ built from $\theta$ and the tensoring isomorphisms are equal, for all $X$, $L$ and $L'$.
\end{enumerate}
\end{thm}

\begin{rmk}\label{rmk:tau}  The functor $F:\cat M\to \cat M'$ of the previous theorem is simplicial if and only if there is  a natural transformation
$$\tau: F\big( (-)^{(-)}\big) \Rightarrow F(-)^{(-)}$$
of functors from $\cat M\times \cat{sSet}^{op}$ into $\cat M'$, which is associative and unital in a sense similar to that above.  Indeed, given $\theta$, let $\tau_{X,L}:F(X^L)\to F(X)^L$ be the transpose, with respect to the $\big(-\otimes L, (-)^L\big)$-adjunction, of the composite
$$F(X^L)\otimes L \xrightarrow {\theta_{X^L,L}} F(X^L\otimes L)\xrightarrow{F(ev)} F(X).$$
Similarly, given $\tau$, let $\theta_{X,L}$ be the transpose, with respect to the $\big(-\otimes L, (-)^L\big)$-adjunction, of the composite
$$F(X) \xrightarrow {F(coev)} F\big( (X\otimes L)^L\big) \xrightarrow {\tau_{X\otimes L, L}} F(X\otimes L)^L.$$
\end{rmk}

\begin{defn}  Let $\cat M$ be a simplicial model category.  The \emph{totalization functor}
$$\tot: \cat M^{\bold\Delta} \to \cat M$$
is defined using the cotensoring on $\cat M$ to be
$$\tot X^\bullet= \operatorname{equal}\Big( \prod_{n\geq 0} (X^n)^{\Delta[n]} \egal{}{} \prod_{\bold\Delta(n,k); k,n\geq0} (X^k)^{\Delta[n]}\Big),$$
where the maps are given in the obvious way by projection onto different factors, either in the simplicial component or in the $\cat M$-component.
The \emph{realization functor}
$$|-|: \cat M^{\bold\Delta^{op}} \to \cat M$$
is defined using the tensoring on $\cat M$ to be
$$|X_{\bullet}|= \operatorname{coeq}\Big( \coprod_{\bold\Delta(k,n); k,n\geq0} X_{n}\otimes \Delta [k] \egal{}{} \coprod_ {n\geq 0}X_{n}\otimes \Delta [n]\Big),$$
where the maps are given in the obvious way by injection into different summands, either in the simplicial component or in the $\cat M$-component.
\end{defn}

\begin{rmk}   We recall that if $X^\bullet$ is Reedy fibrant, then $\tot X^\bullet$ is fibrant in $\cat M$.  If $f^\bullet: X^\bullet \to Y^\bullet$ is a levelwise weak equivalence between Reedy fibrant objects, then $\tot f^\bullet$ is a weak equivalence in $\cat M$.

Dually, if $X_\bullet$ is Reedy cofibrant, then $|X_{\bullet}|$ is cofibrant in $\cat M$.  If $f_{\bullet}: X_{\bullet} \to Y_\bullet$ is a levelwise weak equivalence between Reedy cofibrant objects, then $|f_{\bullet}|$ is a weak equivalence in $\cat M$.
\end{rmk}

\begin{rmk}\label{rmk:tot-cst}  In \cite [Section 1]{dwyer-miller-neisendorfer} the authors mention that if $\cat M=\cat {sSet}$ and $X$ is any simplicial set, then $\tot cc^\bullet X\cong X$.  As Daniel Davis pointed out to the author, it is easy to generalize this isomorphism to any simplicial model category. 

If $X$ is an object in a simplicial model category $\cat M$, then $X^{(-)}:\cat {sSet}^{op}\to \cat M$ is a right adjoint and therefore sends limits to limits. Moreover, in $\cat {sSet}$, 
$$\Delta[0]= \operatorname{coeq}\Big( \coprod_{\bold\Delta(n,k); k,n\geq0}  \Delta [k] \egal{}{} \coprod_ {n\geq 0}\Delta [n]\Big).$$
We conclude that 
$$X\cong X^{\Delta[0]} \xrightarrow \cong \operatorname{equal}\Big( \prod_{n\geq 0} X^{\Delta[n]} \egal{}{} \prod_{\bold\Delta(n,k); k,n\geq0} X^{\Delta[n]}\Big)=\tot cc^\bullet X.$$
\end{rmk}

\subsection{External (co)simplicial structure}\label{sec:external}

Let $\cat C$ be a category admitting all finite limits and colimits.  Quillen showed in \cite[\S II.1]{quillen} (cf. also \cite[\S 2.10]{bousfield}) that $\cat C^{\bold\Delta^{op}}$ and $\cat C^{\bold\Delta}$ both admit a simplicial enrichment that is tensored and cotensored over $\cat {sSet}$, called the \emph{external simplicial structure}.   

Recall Definition \ref{defn:simphop}.

\begin{defn}  Two (co)simplicial morphisms in $\cat C$ are \emph{externally homotopic} if they are simplicially homotopic in the external simplicial structure.  Similarly, a (co)simplicial morphism in $\cat C$ is an \emph{external homotopy equivalence} if it is a simplicial homotopy equivalence in the external simplicial structure.
\end{defn}

Bousfield proved the cosimplicial half of the following useful result in \cite[\S 2.12]{bousfield}.  To prove the simplicial half, one can simply dualize his proof.

\begin{prop}\label{prop:bousfield}\cite{bousfield} Let $\cat M$ be a simplicial model category.
\begin{enumerate}
\item Let $f^\bullet, g^\bullet: X^\bullet \to Y^\bullet$ be morphisms of Reedy fibrant cosimplicial objects in $\cat M^{\bold \Delta}$.  If $f^\bullet$ and $g^\bullet$ are externally homotopic, then $\tot f^\bullet $ and $\tot g^\bullet$ are simplicially homotopic.
\item Let $f_{\bullet}, g_{\bullet}:X_{\bullet}\to Y_{\bullet}$ be morphisms of Reedy cofibrant simplicial objects in $\cat M^{\bold \Delta ^{op}}$.  If $f_{\bullet}$ and $g_{\bullet}$ are externally homotopic, then $|f_{\bullet}|$ and $|g_{\bullet}|$ are simplicially homotopic.
\end{enumerate}
\end{prop}

In this paper, we apply the following consequence of Bousfield's result.

\begin{cor}\label{cor:extequiv-weq} Let $\cat M$ be a simplicial model category.
\begin{enumerate}
\item If $f^\bullet: X^\bullet \to Y^\bullet$ is an external homotopy equivalence of Reedy fibrant objects in $\cat M^{\bold \Delta}$, then $\tot f^\bullet:\tot X^\bullet \to \tot Y^\bullet$ is a weak equivalence in $\cat M$.
\item If $f_\bullet: X_\bullet \to Y_\bullet$ is an external homotopy equivalence of Reedy cofibrant objects in $\cat M^{\bold \Delta^{op}}$, then $|f_\bullet|: |X_\bullet| \to |Y_\bullet|$ is a weak equivalence in $\cat M$.
\end{enumerate}
\end{cor}

\begin{proof} We do the cosimplicial case and the leave the strictly dual proof of the simplicial case to the reader.

Let $g^\bullet: Y^\bullet \to X^\bullet$ be the external homotopy inverse to $f^\bullet$.  By Proposition \ref{prop:bousfield} (1), $\tot (g^\bullet)\tot( f^\bullet)=\tot (g^\bullet f^\bullet)$ is simplicially homotopic to the identity morphism of $\tot X^\bullet$ and $\tot (f^\bullet)\tot( g^\bullet)=\tot (f^\bullet g^\bullet)$ is simplicially homotopic to the identity morphism of $\tot Y^\bullet$, i.e., $\tot (f^\bullet)$ and $\tot (g^\bullet)$ are simplicial homotopy equivalences.  By Proposition 9.5.16 in \cite{hirschhorn}, we can conclude that $\tot (f^\bullet)$ and $\tot (g^\bullet)$ are weak equivalences in $\cat M$.
\end{proof}

We are particularly interested here in the following  excellent, classical source of external homotopy equivalences of cosimplicial (respectively, simplicial) objects.

\begin{defn}\label{defn:contractible} Let $\cat C$ be any category.
\begin{enumerate}
\item  A \emph{contractible cosimplicial object} in $\cat C$ consists of an object $Y^\bullet $ in $\cat C^{\bold\Delta}$ and an object $X$ in $\cat C$, together with morphisms in $\cat C$
 $$\eta: X\to Y^0\quad\text{and}\quad s^n:Y^n\to Y^{n-1}, \forall n\geq 0,$$
 where $Y^{-1}:=X$, satisfying
 $$d^0\eta=d^1\eta: X\to Y^1,$$
 $$s^0\eta=Id_{X}\quad\text{and}\quad s^nd^n=Id_{Y^{n-1}}\; \forall n\geq 0,$$
 $$s^1d^0=\eta s^0\quad\text{and}\quad s^nd^{i}=d^{i}s^{n-1}\;\forall  0\leq i\leq n-1$$
 and 
 $$s^ns^{i}=s^{i}s^{n+1}\;\forall 0\leq i\leq n.$$
 \item A \emph{contractible simplicial object} in $\cat C$ consists of an object $Y_\bullet $ in $\cat C^{\bold\Delta^{op}}$ and an object $X$ in $\cat C$, together with morphisms in $\cat C$
 $$\ve: Y_{0}\to X\quad\text{and}\quad s_{n}:Y_{{n-1}}\to Y_{n}, \forall n\geq 0,$$
 where $Y_{-1}:=X$, satisfying
 $$\ve d_{0}=\ve d_{1},$$
 $$\ve s_{0}=Id_{X}\quad\text{and}\quad d_{n}s_{n}=Id_{Y_{n-1}}\;\forall n\geq0, $$
 $$d_{0}s_{1}=s_{0}\ve\quad\text{and}\quad d_{i}s_{n}=s_{n-1}d_{i}\;\forall 0\leq i\leq n-1, $$
 and
 $$s_{i}s_{n}=s_{n+1}s_{i}\;\forall 0\leq i\leq n.$$
 \end{enumerate}
 \end{defn}
 
 \begin{rmk}  In the literature one can find definitions of contractible (co)simplicial objects that do not impose any conditions on the relation between the ``extra (co)degeneracy'' and the other (co)degeneracies.  If the (co)simplicial objects live in an abelian category, so that one can form the associated (co)chain complex, then the condition on (co)degeneracies is indeed superfluous: the other conditions suffice to prove that the (co)chain complex is acyclic. In this paper, however, we do not assume that the underlying category is abelian or that there is any sort of associated (co)chain complex, whence our need of the condition on (co)degeneracies.  In particular, the proof of the next proposition relies explicitly on this condition.
 \end{rmk}

As proved by Barr in \cite{barr}, contractible (co)simplicial objects give rise to external homotopy equivalences of the following special sort.

\begin{prop}\cite[Prop. 3.3]{barr}\label{prop:barr}  Let $\cat C$ be a category admitting all finite limits and colimits.
\begin{enumerate}
\item If $X\xrightarrow{\eta} Y^\bullet$ is a contractible cosimplicial object in $\cat C$, then there are cosimplicial maps
$$\eta^\bullet : cc^\bullet X\to Y^\bullet\quad \text{and}\quad \rho^\bullet: Y^\bullet \to cc^\bullet X$$
such that $\rho^\bullet \eta^\bullet = Id_{cc^\bullet X}$ and $\eta^\bullet \rho^\bullet$ is externally homotopic to $Id_{Y^\bullet}$.
\item If $Y_\bullet\xrightarrow \ve X$ is a contractible simplicial object in $\cat C$, then there are simplicial maps
$$\ve_{\bullet}: Y_\bullet\to cs_{\bullet}X \quad \text{and}\quad \sigma_\bullet: cs_\bullet X\to Y_{\bullet}$$
such that $\ve_{\bullet}\sigma_{\bullet} = Id_{cs_\bullet X}$ and $\sigma_\bullet \ve_{\bullet}$ is externally homotopic to $Id_{Y_\bullet}$.
\end{enumerate}
\end{prop}

Motivated by the proposition above, we formulate the following definition.

\begin{defn}\label{defn:extSDR} Let $\cat C$  be a category admitting all finite limits and colimits.
\begin{enumerate}
\item An \emph{external cosimplicial strong deformation retract (SDR)} in $\cat C$  consists of a pair of  morphisms in $\cat C^{\bold \Delta}$
$$\xymatrix{X^\bullet \ar @<.7ex>[r]^-{\eta^\bullet}&Y^\bullet\ar @<.7ex>[l]^-{\rho^\bullet}}$$
such that $ \rho^\bullet \eta^\bullet = Id_{cc^\bullet X}$ and $\eta^\bullet \rho^\bullet$ is externally homotopic to $Id_{Y^\bullet}$.
\item An \emph{external simplicial strong deformation retract (SDR)} in $\cat C$  consists of a pair of  morphisms in $\cat C^{\bold \Delta^{op}}$
$$\xymatrix{X_\bullet \ar @<.7ex>[r]^-{\sigma_\bullet}&Y_\bullet\ar @<.7ex>[l]^-{\ve_\bullet}}$$
such that $\ve_{\bullet}\sigma_{\bullet} = Id_{cs_\bullet X}$ and $\sigma_\bullet \ve_{\bullet}$ is externally homotopic to $Id_{Y_\bullet}$.
\end{enumerate}
\end{defn}

\begin{notn} An external cosimplicial SDR is denoted
$$\xymatrix{X^\bullet \ar @<.7ex>[r]^-{\eta^\bullet}&Y^\bullet\circlearrowright h\ar @<.7ex>[l]^-{\rho^\bullet}},$$
while an external simplicial SDR is denoted
$$\xymatrix{X_\bullet \ar @<.7ex>[r]^-{\sigma_\bullet}&Y_\bullet\circlearrowright h\ar @<.7ex>[l]^-{\ve_\bullet}}.$$
\end{notn}

\subsection{Ext in simplicial model categories}\label{sec:ext}

In our study of (co)descent spectral sequences, we are led to consider an Ext-type construction for (co)descent data, for which we provide the theoretical framework here. Recall that, if $R$ is a ring and $M$ and $N$ are right $R$-modules, then the elements of $\ext^n_{R}(M,N)$ are in bijective correspondence with chain homotopy classes of chain maps of degree $-n$ from a resolution of $M$ to an injective resolution of $N$, as well as with chain homotopy classes of maps of degree $-n$ from a projective resolution of $M$ to a resolution of $N$.  

We begin by introducing the maps that play the role of injective or projective resolutions in simplicial model categories. 

\begin{defn}\label{defn:homres}  Let $\cat M$ be a simplicial model category, and let $Z\in \ob\cat M$.  
\begin{enumerate}
\item A cosimplicial morphism $\zeta^\bullet: cc^\bullet Z\to Y^\bullet$ is a \emph{cohomological fibrant resolution} of $Z$ if $Y^\bullet$ is Reedy fibrant and for every commuting diagram
$$\xymatrix{cc^\bullet Z\ar@{ >->}[d]^{\sim}_{\iota^\bullet}\ar [r]^{\zeta^\bullet}&Y^\bullet\ar @{->>}[d]\\
R^\bullet \ar @{->>}[r]\ar [ur]^{\bar\zeta^\bullet}& cc^\bullet *}$$
and for every morphism $f:X\to Z$ in $\cat M$, where $X$ is cofibrant,
the induced homomorphisms
$$\pi^s\pi_{t}(\bar \zeta^\bullet)_{*}:\pi^s\pi_{t} \map_{\cat M}(X, R^\bullet)_{f}\to \pi^s\pi_{t} \map_{\cat M}(X, Y^\bullet)_{f}$$
are isomorphisms for all $t\geq s\geq 0$, where the mapping spaces are pointed with respect to the morphism induced by $f$.
\item A simplicial morphism $\zeta_\bullet: X_{\bullet}\to cs_{\bullet}Z$ is a \emph{homological cofibrant resolution} of $Z$ if $X_{\bullet}$ is Reedy cofibrant and for every commuting diagram
$$\xymatrix{cs_{\bullet}\emptyset\ar@{ >->}[d]\ar @{ >->} [r]&Q_{\bullet}\ar @{->>}[d]^\sim\\
X_{\bullet} \ar [r]^{\zeta_{\bullet}}\ar [ur]^{\bar\zeta_\bullet}& cs_{\bullet}Z}$$
and for every morphism $f:Z\to Y$ in $\cat M$, where $Y$ is fibrant,
the induced homomorphisms
$$\pi^s\pi_{t}(\bar \zeta_{\bullet})^*:\pi^s\pi_{t} \map_{\cat M}(Q_{\bullet}, Y)_{f}\to \pi^s\pi_{t} \map_{\cat M}(X_{\bullet}, Y)_{f}$$
are isomorphisms for all $t\geq s\geq 0$, where the mapping spaces are pointed by the morphisms induced in the obvious manner by $f$.
\end{enumerate}
\end{defn}

\begin{rmk} \label{rmk:homres}Any fibrant replacement of $cc^\bullet Z$ is obviously a cohomological fibrant resolution, as is any levelwise weak equivalence $cc^\bullet Z \to Y^\bullet$, where $Y^\bullet$ is Reedy fibrant.  In particular, if $Z$ is fibrant, then the identity map on $cc^\bullet Z$ is itself a cohomological fibrant resolution and therefore $\pi^s\pi_{t} \map_{\cat M}(X, Y^\bullet)_{f}$ is concentrated in bidegrees $(0,t)$, for $t\geq 0$, for all cofibrant $X$ and cohomological fibrant resolutions $cc^\bullet Z\to Y^\bullet$.

Not all cohomological fibrant resolutions are levelwise weak equivalences, however.  If $\zeta^\bullet: cc^\bullet Z \to Y^\bullet$ is an external homotopy equivalence,  and $Y^\bullet$ is Reedy fibrant, then $cc^\bullet Z$ is also Reedy fibrant, and $\zeta^\bullet$ is a cohomological fibrant resolution. The key to the proof is Bousfield's observation in section 2.12 of \cite {bousfield} that if $h^\bullet:A^\bullet \to B^\bullet$ is an external homotopy equivalence of cosimplicial abelian groups (respectively, groups or sets), then it induces an isomorphism $\pi^s h^\bullet:\pi^s A^\bullet \xrightarrow\cong \pi^s B^\bullet$ for all $s\geq 0$ (respectively, $s=0,1$ or $s=0$). On the other hand, there is no reason in general why $\zeta^n$ should be a weak equivalence for every $n$.
For example, for any monad $\T$ on a simplicial model category $\cat M$, if $(A,m)$ is a $\T$-algebra such that $\Om ^\T_{\bullet}A$ is Reedy fibrant, then there is an external homotopy equivalence $cc^\bullet A \to \Om _{\T}^\bullet A$ in $\cat M^{\bold \Delta}$ (Remark \ref{rmk:ex-Tcomp}) induced by the unit map $\eta :A\to TA$, which will certainly usually not be a weak equivalence.

Similar observations hold, of course, in the dual case.
\end{rmk}

\begin{rmk}\label{rmk:conseqcohomfibres} Let $X$ and $Z$ be a cofibrant and a fibrant object in $\cat M$, respectively, and let $f\in \cat M(X,Z)$. A cosimplicial map $\zeta^\bullet : cc^\bullet Z \to Y^\bullet$, where $Y^\bullet$ is Reedy fibrant, induces a morphism of spectral sequences between the extended homotopy spectral sequences \cite[X.6]{bousfield-kan} associated to  the fibrant simplicial sets $\map _{\cat M}(X,cc^\bullet Z) $ and $\map _{\cat M}(X,Y^\bullet )$, where basepoints are determined by $f$.  

If $\zeta^\bullet$ is a cohomological fibrant resolution, then the induced map on the $E_{2}$-terms is an isomorphism.  It follows that, if the two spectral sequences converge to $\pi_{*}\tot\map _{\cat M}(X, cc^\bullet Z)$ and to $\pi_{*}\tot\map _{\cat M}(X, Y^\bullet )$, respectively, then 
$$\pi_{*}\tot(\zeta^\bullet)_{*}: \pi_{*} \map_{\cat M}(X,Z)_{f} \to \pi_{*}\map _{\cat M}(X, \tot Y^\bullet)_{f}$$
is an isomorphism, where we have used that $\tot cc^\bullet Z\cong Z$ (Remark \ref{rmk:tot-cst}).  Thus, if the spectral sequences converge for all cofibrant $X$, then 
$$\tot(\zeta^\bullet)_{*}:  \map_{\cat M}(X,Z) \to \map _{\cat M}(X, \tot Y^\bullet)$$
is a weak equivalence for all cofibrant $X$, which implies by \cite[Proposition 9.7.1]{hirschhorn} that 
$$\tot (\zeta^\bullet): Z\to \tot Y^\bullet$$
is a weak equivalence.

Again, the obvious dual results clearly hold as well.
\end{rmk}

\begin{defn}\label{defn:ext} Let $\cat M$ be a simplicial model category. Let $X, Y\in \ob \cat M$ and $f\in \cat M(X,Y)$.
If $X$ is cofibrant, then for all $s,t\geq 0$, the group $\ext^{s,t}_{\cat M}(X,Y)_{f}$ is defined to be 
$$\pi^s\pi_{t}\map_{\cat M}(X, \widehat Y^\bullet)_{f},$$
where $Y\to \widehat Y^\bullet$ is a cohomological fibrant resolution. 
If $Y$ is fibrant, then for all $s,t\geq 0$, the group $\ext^{s,t}_{\cat M}(X,Y)_{f}$ is defined to be
$$\pi^s\pi_{t}\map_{\cat M}(\widetilde X_{\bullet},Y)_{f},$$
where $\widetilde X_{\bullet}\to X$ is a homological cofibrant resolution. The mapping spaces are pointed in the obvious way by the morphisms induced by $f$.
\end{defn}

\begin{rmk} If $X$ is cofibrant and $Y$ is fibrant, then it is clear that the value of $\ext^{s,t}_{\cat M}(X,Y)$, which is concentrated in bidegrees $(0,t)$ for $t\geq 0$, is independent of which component is resolved.   
\end{rmk}

\section{Model categories of Eilenberg-Moore (co)algebras}\label{sec:indmodcat}

In order to define model category structures on categories of Eilenberg-Moore (co)algebras, which we need for our discussion of both (co)completion and homotopic (co)descent,  we  work with model category structures that are induced across adjunctions,  of either of the types specified below.

\begin{defn}\label{defn:induced} Let $U: \cat C \to \cat M$ be a functor, where $\cat M$ is a model category. A model structure on $\cat C$ is \emph{right-induced} from $\cat M$ if $\mathsf {WE}_{\cat C}=U^{-1}(\mathsf {WE}_{\cat M})$ and $\mathsf {Fib}_{\cat C}=U^{-1}(\mathsf {Fib}_{\cat M})$.

Let $F: \cat C \to \cat M$ be a functor, where $\cat M$ is a model category. A model structure on $\cat C$ is \emph{left-induced} from $\cat M$ if $\mathsf {WE}_{\cat C}=F^{-1}(\mathsf {WE}_{\cat M})$ and $\mathsf {Cof}_{\cat C}=F^{-1}(\mathsf {Cof}_{\cat M})$.
\end{defn}

 Schwede and Shipley provided a careful treatment of the case of algebras over a monad in \cite {schwede-shipley}, to which we refer the reader for details of definitions and terminology.

\begin{prop}[\cite{schwede-shipley}, Lemma 2.3] \label{prop:schwede-shipley} Let $\cat M$ be a cofibrantly generated model category, with $\mathsf I$ a set of generating cofibrations and $\mathsf J$ a set of generating acyclic cofibrations.  Let $\T=(T,\mu,\eta)$ be a monad on $\cat M$ such that $T$ commutes with filtered direct limits.  Let $\mathsf I^\T$ and $\mathsf J^\T$ denote the images of $\mathsf I$ and $\mathsf J$ under the free $\T$-algebra functor $F^\T$. Assume that the domains of the morphisms in $\mathsf I^\T$ and $\mathsf J^\T$ are small with respect to regular $\mathsf I^\T$-cofibrations and regular $\mathsf J^\T$-cofibrations.

The forgetful functor $U^\T:\cat M^\T\to \cat M$ right-induces a cofibrantly generated  model category structure on $\cat M^\T$, with $\mathsf I^\T$ as generating cofibrations and $\mathsf J^\T$ as generating acyclic cofibrations, if 
\begin{enumerate}
\item [(a)] every regular $\mathsf J^\T$-cofibration is a weak equivalence, or
\item [(b)] every object of $\cat M$ is fibrant and every $\T$-algebra admits a path object.
\end{enumerate}
\end{prop}

The next proposition, which follows from  Corollary 5.15 in \cite{hess:hhg}, provides one example of conditions under which there exist model category structures in the coalgebra case.   Other sets of sufficient conditions in particular cases were introduced by Hovey in \cite {hovey:comod} and by Stanculescu in \cite{stanculescu}.

Before stating the proposition, we recall a bit of helpful notation and a key definition from \cite{hess:hhg}.

\begin{notn}Let $\mathsf X$ be any subset of morphisms in a category $\cat C$. 
The closure of $\mathsf X$ under formation of retracts is denoted $\widehat{\mathsf X}$, i.e., 
$$f\in \widehat{\mathsf X}\Longleftrightarrow  \exists\; g\in \mathsf X \text{ such that $f$ is a retract of $g$} .$$
\end{notn}

\begin{defn} \label{defn:postnikov} Let $\mathsf X$ be a set of morphisms in a category $\cat C$ that is closed under pullbacks.  If $\lambda$ is an ordinal and  $Y:\lambda ^{op}\to \cat C$ is a functor such that for all $\beta <\lambda$, the morphism $Y_{\beta +1}\to Y_{\beta}$ fits into a pullback
$$\xymatrix{Y_{\beta +1}\ar [d]_{}\ar [r]^{}&X_{\beta +1}\ar [d]^{x_{\beta+1}}\\ Y_{\beta}\ar [r]^{k_{\beta}}&X_{\beta}}$$
for some $x_{\beta +1}: X_{\beta +1}\to X_{\beta}$ in $\mathsf X$ and  $k_{\beta}:Y_{\beta}\to X_{\beta}$ in $\cat C$, while
$Y_{\gamma}:=\lim _{\beta<\gamma}Y_{\beta}$ for all limit ordinals $\gamma<\lambda$,
then the composition of the tower $Y$
$$\lim_{\lambda^{op}}Y_{\beta}\to Y_{0},$$
\emph{if it exists}, is an \emph{$\mathsf X$-Postnikov tower}.  The set of all $\mathsf X$-Postnikov towers is denoted $\mathsf {Post}_{\mathsf X}$.

A \emph{Postnikov presentation} of a model category $\cat M$ is  a pair of  sets of morphisms $\mathsf X$ and $\mathsf Z$ satisfying 
$$\mathsf {Fib}_{\cat M}=\widehat{\mathsf {Post}_{\mathsf X}}\quad\text { and }\quad \mathsf {Fib}_{\cat M}\cap \mathsf {WE}_{\cat M}=\widehat{\mathsf {Post}_{\mathsf Z}}$$ 
and such that for all $f\in \mor \cat M$,
there exist 
\begin{enumerate}
\item [(a)] $i\in \mathsf{Cof} $ and $p\in \mathsf {Post}_{\mathsf Z}$ such
that $f=pi$; 
\item [(b)] $j\in \mathsf{Cof}\cap \mathsf{WE}$ and $ q\in \mathsf {Post}_{\mathsf X}$ such that
$f=qj$.
\end{enumerate}
\end{defn}

\begin{prop}\label{prop:hhg}\cite{hess:hhg} Let $\cat M$ be a model category admitting a Postnikov presentation $(\mathsf X, \mathsf Z)$, and let $\K$ be a comonad on $\cat M$ such that $\cat M_{\K}$ is finitely bicomplete.  Let 
$$\mathsf W=(U_{\K})^{-1}(\mathsf{WE}_{\cat M}) \text { and } \mathsf C=(U_{\K})^{-1}(\mathsf{Cof}_{\cat M}).$$

The forgetful functor $U_{\K}:\cat M_{\K}\to \cat M$ left-induces a model category structure on $\cat M_{\K}$, where the set of fibrations is
$\widehat{\mathsf{Post}_{F_{\K}(\mathsf X)}},$
 if     
$\mathsf{Post}_{F_{\K}(\mathsf Z)}\subset \mathsf W$ and for all $f\in \mor \cat C$
there exist 
\begin{enumerate}
\item [(a)] $i\in \mathsf{C} $ and $p\in \mathsf{Post}_{F_{\K}( \mathsf Z)}$ such
that $f=pi$; 
\smallskip
\item [(b)]$j\in \mathsf{C}\cap \mathsf{W}$ and $q\in \mathsf{Post}_{F_{\K}(\mathsf X)}$ such that
$f=qj$.
\end{enumerate}
\end{prop}

\begin{rmk} Together with B. Shipley, we have provided conditions on $\cat M$ and on $\K$ such that the hypotheses of Proposition \ref{prop:hhg} are satisfied \cite{hess-shipley} and shown that these conditions are satisfied in a number of cases of interest.  In particular our results show that, under reasonable conditions on a monoidal model category $\cat M$, the category of Grothendieck descent data $\D(\T_{\vp})$ associated to a monoid morphism $\vp:B\to A$ admits a model category structure left-induced from $\cat {Mod}_{A}$.
\end{rmk}

\begin{rmk} It is an easy exercise to show that for any comonad $\K=(K,\Delta, \ve)$ on a category $\cat D$, the forgetful functor $U_{\K}:\cat D_{\K}\to \cat D$ creates colimits.  It follows that if $\cat D$ is cocomplete, then $\cat D_{\K}$ is as well.  
On the other hand,  the duals of Corollary 3.3 and Theorem 3.9 in section 9.3 of \cite{barr-wells} together imply that $\cat D_{\K}$ is complete if  $\cat D$ is complete and $K$ preserves countable inverse limits, which is true for many comonads of interest.
\end{rmk}

 \bibliographystyle{amsplain}
\bibliography{adams}

\end{document}